\documentclass{article}
\usepackage[utf8]{inputenc} % allow utf-8 input
\usepackage[T1]{fontenc}    % use 8-bit T1 fonts

\usepackage{booktabs}       % professional-quality tables
\usepackage{amsfonts}       % blackboard math symbols
\usepackage{microtype}      % microtypography

\usepackage{nicefrac}

\usepackage{bm}
\usepackage{amsmath}
\usepackage{amsthm}
\usepackage{amssymb}
\usepackage{thmtools,thm-restate}
\usepackage{thm-restate}
\usepackage{paralist}
\usepackage{comment}
\usepackage{bbm}
\usepackage{balance}
\usepackage{stmaryrd}
\usepackage{multicol}
\usepackage[export]{adjustbox}
\usepackage{multirow}
\usepackage{tabularx}
\usepackage{mathtools}

\usepackage{enumitem}
\setitemize{noitemsep,topsep=0pt,parsep=0pt,partopsep=0pt}

\newcommand{\diag}{\operatorname{diag}}

\newcommand{\innp}[1]{\left\langle #1 \right\rangle}

\newcommand{\vx}{\mathbf{x}}
\newcommand{\va}{\mathbf{a}}

\newcommand{\vd}{\mathbf{d}}

\newcommand{\cx}{\mathcal{X}}

\newcommand{\cs}{\mathcal{S}}

\newcommand{\vy}{\mathbf{y}}
\newcommand{\vz}{\mathbf{z}}

\newcommand{\vw}{\mathbf{w}}

\newcommand{\vu}{\mathbf{u}}

\newcommand{\vlambda}{\bm{\lambda}}

\newcommand{\defeq}{\stackrel{\mathrm{\scriptscriptstyle def}}{=}}

\newcommand{\rr}{\mathbb{R}}
\newcommand{\norm}[1]{\left\| #1 \right\|}

\newcommand{\abs}[1]{\lvert #1 \rvert}

\newcommand*{\vsepfbox}[1]{%
  \begingroup
    \sbox0{\fbox{#1}}%
    \setlength{\fboxrule}{0pt}%
    \mbox{\kern-\fboxsep\fbox{\unhbox0}\kern-\fboxsep}%
  \endgroup
}

\theoremstyle{plain} \numberwithin{equation}{section}
\newtheorem{theorem}{Theorem}[section]
\numberwithin{theorem}{section}
\newtheorem{corollary}[theorem]{Corollary}

\newtheorem{lemma}[theorem]{Lemma}

\theoremstyle{definition}
\newtheorem{definition}[theorem]{Definition}

\newtheorem{remark}[theorem]{Remark}

\theoremstyle{plain}
\newtheorem{assumption}{Assumption}

\DeclareMathOperator{\interior}{\mathrm{int}}
\DeclareMathOperator{\relinterior}{\mathrm{rel.int}}
\DeclareMathOperator*{\trace}{trace}
\DeclareMathOperator*{\argmin}{argmin}
\DeclareMathOperator*{\argmax}{argmax}

\DeclareMathOperator{\vertex}{\mathrm{vert}}
\DeclareMathOperator{\co}{\mathrm{conv}}

\newcommand{\hrulealg}[0]{\vspace{1mm} \hrule \vspace{1mm}}

\usepackage{subfig}
\usepackage{graphicx,wrapfig,lipsum}

\usepackage[ruled, linesnumbered]{algorithm2e} % http://www.ctan.org/pkg/algorithm2e
\usepackage{xcolor} % http://www.ctan.org/tex-archive/macros/latex/contrib/xcolor
\usepackage{tikz}
\usetikzlibrary{fit,calc}
\colorlet{pink}{red!40}
\colorlet{lightblue}{blue!30}
\DontPrintSemicolon

\usepackage[colorlinks,urlcolor=blue,citecolor=blue,linkcolor=blue]{hyperref}       % hyperlinks
\usepackage{url}            % simple URL typesetting

\usepackage{jmlr2e}
\usepackage[margin=1in]{geometry}
\usepackage{natbib}
\setcitestyle{authoryear,round,citesep={;},aysep={,},yysep={;}}
\renewcommand{\cite}[1]{\citep{#1}}
\date{}

\usepackage[vvarbb]{newtxmath}

\begin{document}

\title{Second-order Conditional Gradient Sliding}

\author{\name Alejandro Carderera \email  \href{mailto:alejandro.carderera@gatech.edu}{alejandro.carderera@gatech.edu}\\
       \addr Department of Industrial and Systems Engineering\\
       Georgia Institute of Technology\\
       Atlanta, USA
       \AND
       \name Sebastian Pokutta \email \href{mailto:pokutta@zib.de}{pokutta@zib.de} \\
       \addr Institute of Mathematics\\
       Zuse Institute Berlin and Technische Universität Berlin\\
       Berlin, Germany}
       
\maketitle

\begin{abstract}

Constrained second-order convex optimization algorithms are the method of choice when a high accuracy solution to a problem is needed, due to their local quadratic convergence. These algorithms require the solution of a constrained quadratic subproblem at every iteration. We present the \emph{Second-Order Conditional Gradient Sliding} (SOCGS) algorithm, which uses a projection-free algorithm to solve the constrained quadratic subproblems inexactly and uses inexact Hessian oracles (subject to an accuracy requirement). When the feasible region is a polytope the algorithm converges quadratically in primal gap after a finite number of linearly convergent iterations. Once in the quadratic regime the SOCGS algorithm requires $\mathcal{O}(\log(\log 1/\varepsilon))$ first-order and inexact Hessian oracle calls and $\mathcal{O}(\log (1/\varepsilon) \log(\log1/\varepsilon))$ linear minimization oracle calls to achieve an $\varepsilon$-optimal solution. This algorithm is useful when the feasible region can only be accessed efficiently through a linear optimization oracle, and computing first-order information of the function, although possible, is costly.

\end{abstract}

\section{Introduction}

We focus on the optimization problem defined as
\begin{align}
    \min_{\vx \in \cx} f(\vx), \label{eq:OptProblem}
\end{align}
where $\cx$ is a polytope and $f:\cx \rightarrow \rr$ is $\mu$-strongly convex, has $L$-Lipschitz continuous gradients and has $L_2$-Lipschitz continuous Hessian.

An immensely powerful approach to tackle Problem~\eqref{eq:OptProblem} is to construct a second-order approximation to $f(\vx)$ at the current iterate using $\nabla f(\vx)$ and $\nabla^2 f(\vx)$, and move in the direction that minimizes this approximation, giving rise to a family of methods known as \emph{Newton} methods, first developed for unconstrained problems \cite{kantorovich1948functional}. Variants of the former converge globally and have a local quadratic convergence rate when minimizing a self-concordant function or a strongly convex function with Lipschitz continuous Hessian \cite{nesterov1994interior, nesterov2013introductory}. When the problem at hand is constrained to a convex set, one can use a constrained analog of these methods \cite{levitin1966constrained}, where a quadratic approximation to the function is minimized over $\cx$ at each iteration.

However, there are two shortcomings to these methods. First, computing second-order information about $f(\vx)$ is expensive. This has led to the development of \emph{Variable-Metric} algorithms, which use approximate second-order information. Secondly, in many cases solving the quadratic subproblem to optimality is too costly. This has resulted in numerous \emph{Inexact Variable-Metric} algorithms, which in many cases inherit many of the favorable properties of Newton methods \cite{scheinberg2016practical, lee2014proximal}.

The \emph{Conditional Gradients} (CG) algorithm \cite{levitin1966constrained} (also known as the \emph{Frank-Wolfe} algorithm \cite{frank1956algorithm}) instead builds a linear approximation to $f(\vx)$ using $\nabla f(\vx)$, and moves in the direction given by the point that minimizes this linear approximation over $\cx$. Instead of solving a constrained quadratic problem at each iteration, it solves a constrained linear problem, which is usually much cheaper. As the algorithm maintains its iterates as convex combinations of extremal points of \(\cx\) obtained from the linear optimization problem it is dubbed \emph{projection-free}. Conditional Gradients have become the method of choice in many applications where projecting onto $\mathcal{X}$ is computationally prohibitive, such as, e.g., in video co-localization \cite{joulin2014efficient} or greedy particle optimization in Bayesian inference \cite{futami2019bayesian}.

For constrained problems where the gradient of $f(\vx)$ is relatively hard to compute, using Projected Variable-Metric methods seems counter-intuitive, yet it allows the construction of a quadratic approximation whose gradients are much cheaper to compute. Minimizing a quadratic approximation at each iteration is often costly, but due to the substantial progress it provides per-iteration it can often become competitive in wall-clock time with using first-order algorithms to directly minimize $f(\vx)$ \cite{schmidt2009optimizing}. We consider the case where both the first-order oracle for $f(\vx)$ and the projection oracle onto $\cx$ are computationally expensive, but linear programming oracles over $\cx$ are relatively cheap. In this setting, we show how conditional gradient algorithms can be used to compute Inexact Projected Variable-Metric steps, in an approach that is similar in essence to \emph{Conditional Gradient Sliding} (CGS) \cite{lan2016conditional}, where the Euclidean projections onto $\cx$ in \emph{Nesterov's Accelerated Gradient Descent} are computed using the conditional gradient algorithm. We also show how coupling with an independent sequence of conditional gradient steps we can guarantee the global linear convergence in primal gap of the algorithm.
\subsection{Contributions and related work}

We provide a projection-free Inexact Variable-Metric algorithm, denoted as the \emph{Second-order Conditional Gradient Sliding} (SOCGS) algorithm which uses inexact second-order information. The algorithm has a stopping criterion that relies on a lower bound on the primal gap, e.g., via smoothness, and achieves global linear convergence and quadratic local convergence when close to the optimum. 

The use of a combination of second-order and projection-free methods was first pioneered in \citet{gonccalves2017newton}, who proposed an algorithm in which exact unconstrained Newton steps were performed, and were later projected onto $\cx$ using the Euclidean norm and the CG algorithm. This resulted in a method that showed local linear convergence in distance to the optimum for functions whose derivative satisfied a H\"{o}lder-like condition and also for a subclass of analytic functions. This was later extended in \citet{gonccalves2018inexact} to deal with inexact second-order information, using the inexactness criteria in \citet{morini1999convergence}, and resulting in the same local linear convergence. A variation of the former algorithm, was shown to converge globally (without an explicit convergence rate) using a non-monotone line search strategy \cite{gonccalves2019global}. Neither of these three algorithms included a complexity analysis on the number of linear minimization oracle calls needed to achieve a certain target accuracy. 

An approach that is similar in spirit is the recent \emph{Newton Conditional Gradient} (NCG) algorithm \cite{liu2020newton} which performs Inexact Newton steps using a conditional gradient algorithm to minimize a self-concordant function over $\cx$. This algorithm requires exact second-order information, as opposed to the approximate information used by the SOCGS algorithm, however it does not require the function to be smooth and strongly convex, or the feasible region to be a polytope. After a finite number of damped-steps the NCG algorithm reaches an $\varepsilon$-optimal solution with $\mathcal{O}(\log \frac{1}{\varepsilon})$ first order and exact Hessian oracle calls and $\mathcal{O}(\varepsilon^{-\nu})$ linear optimization oracle calls with $\nu = 1 + o(1)$. Note that \citet{ochs2019model}[Example 4.3] also proposed a conditional gradient-based Variable-Metric algorithm via their \emph{Model Function-Based Conditional Gradient} algorithm, however their approach is markedly different from ours: the steps performed in their algorithm can be seen as unconstrained Variable-Metric steps which are projected onto $\cx$ using the Euclidean norm while the SOCGS performs steps which can be interpreted as unconstrained Inexact Variable-Metric steps which are projected onto $\cx$ using a norm defined by the positive semi-definite matrix that approximates the Hessian. The same can be said regarding the algorithms in \cite{gonccalves2019global}, moreover, the SOCGS algorithm directly approximately minimizes a quadratic using a CG variant, in an operation that directly represents an Inexact Projected Variable-Metric step, whereas the algorithm in \cite{gonccalves2019global} proceeds in a sequential manner, first computing the Newton step, and afterwards projecting onto $\cx$ using the CG algorithm.

Further CG variants for the minimization of a self-concordant function have been developed in \citet{dvurechensky2020self}, in which an algorithm was developed that achieves an $\epsilon$-optimal solution in primal gap after $O(1/\epsilon)$ first-order, second-order and linear minimization oracle calls. A related second-order CG variant was later developed in \citet{zhao2022analysis} for the minimization of the sum of a logarithmically-homogeneous self-concordant barrier function and a non-smooth function with bounded domain. The aforementioned algorithm reaches an $\epsilon$-optimal solution in primal gap after $O( 1/\epsilon)$ first-order, second-order and linear minimization oracle calls. 

Later on, other variants in \citet{carderera2021simple} and \citet{dvurechensky2022generalized} were shown to achieve an $\epsilon$-optimal solution in primal gap when minimizing a generalized self-concordant function after $O(1/\epsilon)$ first-order, domain and linear minimization oracle calls (where the domain oracle call simply checks if a given point is in the domain of the function being minimized). These variants, therefore, do not require second-order information. When the feasible region under consideration is a polytope, a variant presented in \citet{carderera2021simple} of the Away-step Conditional Gradient (ACG) algorithm \cite{wolfe1970convergence} with the stepsize of \citet{pedregosa2020linearly} was shown to achieve an $\epsilon$-optimal solution in primal gap after $O(\log 1/\epsilon)$ first-order, domain and linear minimization oracle calls. Another CG variant in \citet{dvurechensky2022generalized} was shown to achieve an $\epsilon$-optimal solution in primal gap after $O(\log 1/\epsilon)$ first-order, second-order and linear minimization oracle calls when the feasible region is a polytope. 

Since our initial submission in 2020, several follow-up works and related preprints have been published in conferences and journals. We have updated the bibliography to reflect these developments.

\section{Preliminaries}       
 We denote the unique minimizer of Problem~\eqref{eq:OptProblem} by $\vx^*$. Let $\mathcal{S}^n_{++}$ and $I^n$ denote the set of symmetric positive definite matrices and the identity matrix in $\rr^{n\times n}$. We denote the largest eigenvalue of the matrix $H \in \rr^{n\times n}$ as $\lambda_{\max}\left(H\right)$. Let $\norm{\cdot}$ and $\norm{\cdot}_{H}$ denote the \emph{Euclidean norm} and the \emph{matrix norm} defined by $H \in \mathcal{S}^n_{++}$, respectively. We denote the \emph{diameter} of the polytope $\cx$ as $D = \max_{\vx,\vy \in \cx} \norm{\vx - \vy}$, and its \emph{vertices} by $\vertex\left(\mathcal{X}\right)\subseteq \cx$. Given a non-empty set $\mathcal{S}\subset \rr^n$  we refer to its \emph{convex hull} as $\co\left( \mathcal{S} \right)$. For any $\vx\in \cx$ we denote by $\mathcal{F}\left( \vx\right)$ the \emph{minimal face} of $\cx$ that contains $\vx$. Lastly, given a matrix $H \in \mathcal{S}^n_{++}$ we denote the \emph{$H$-scaled projection} of $\vy$ onto $\cx$ as:
\begin{align}
    \Pi^{H}_{\cx}(\vy)  & \defeq \argmin_{\vx \in \cx} \frac{1}{2} \norm{\vx - \vy}_{H}^2 \\
    & = \argmin_{\vx \in \cx} \frac{1}{2} \norm{H^{1/2} \left( \vx - \vy \right)}^2. \label{projectionMapping}
\end{align}
 
 \subsection{The Conditional Gradients algorithm} \label{Subsection:CGAlg}
  We define the linear approximation of the function $f(\vx)$ around the point $\vx_k$ as:
\begin{align}
    \hat{l}_{k}(\vx) \defeq f(\vx_k) + \innp{\nabla f(\vx_k), \vx - \vx_k}. \label{def:LinearApprox}
\end{align}
At each iteration the vanilla \emph{Conditional Gradients} (CG) algorithm \cite{levitin1966constrained, frank1956algorithm, jaggi2013revisiting} takes steps defined as $\vx_{k+1} = \vx_k + \gamma_k (\argmin_{\vx \in \cx} \hat{l}_{k}(\vx) - \vx_k )$ with $\gamma_k \in (0,1]$. As the iterates are formed as convex combinations of points in $\cx$ the algorithm is \emph{projection-free}. A useful quantity that can readily be computed in all steps is $\max_{\mathbf{v} \in \cx}\langle\nabla f(\vx_k),\vx_k-\mathbf{v}\rangle$, known as the \emph{Frank-Wolfe gap}, which provides an upper bound on the primal gap and is often used as a stopping criterion when running the CG algorithm.

However, the vanilla CG algorithm does not converge linearly in primal gap when applied to Problem~\eqref{eq:OptProblem} in general. This motivated the development of the \emph{Away-step Conditional Gradient} (ACG) algorithm \cite{wolfe1970convergence} (shown in Algorithm~\ref{Algo:Appx:ACGAlg} in Appendix~\ref{Section:Appx:ACG}), which uses \emph{Away-steps} (shown in Algorithm~\ref{algo:ACGStep}) and converges linearly when coupled with an exact line search \citep{lacoste2015global} or a step size strategy dependent on $L$ \cite{pedregosa2020linearly}. The ACG algorithm maintains what is called an \emph{active set} $\cs_k \subseteq \vertex\left(\mathcal{X}\right)$ which represents the potentially non-unique set of vertices of $\mathcal{X}$ such that $\vx_k \in \co\left(\mathcal{S}_k\right)$. Associated with this active set $\cs_k$ we have a set of barycentric coordinates $\vlambda_k$ such that if we denote by $\vlambda_k(\vu)\in [0,1]$ the element of $\vlambda_k$ associated with $\vu \in \cs_k$ we have that $\vx_{k} = \sum_{\vu \in \cs_k} \vlambda_k(\vu) \vu$, with $\sum_{\vu \in \cs_k} \vlambda_k(\vu) = 1$ and $\vlambda_k(\vu) \geq 0$ for all $\vu \in \cs_k$.

\begin{algorithm}
\SetKwInOut{Input}{Input}\SetKwInOut{Output}{Output}
\Input{Gradient $\nabla f(\vx_k)$, point $\vx_k \in \cx$, active set $\cs_k$ and barycentric coordinates $\vlambda_k$.}
\Output{Point $\vx_{k+1} \in \cx$, active set $\cs_{k+1}$ and barycentric coordinates $\vlambda_{k+1}$.}
\hrulealg
$\mathbf{v} \leftarrow \argmin_{\mathbf{v} \in  \mathcal{X}} \innp{\nabla f\left(\vx_k \right), \mathbf{v}}$, $\va\leftarrow \argmax_{\mathbf{v} \in  \mathcal{S}_k} \innp{\nabla f\left(\vx_k \right), \mathbf{v}}$\;
\uIf{$\innp{\nabla f(\vx_k),  \vx_k - \mathbf{v}} \geq \innp{\nabla f(\vx_k),\va - \vx_k}$}{
$\vd \leftarrow \vx_k - \mathbf{v}$, $\gamma_{\max} \leftarrow 1$\;}
\Else{ 
$\vd \leftarrow \va - \vx_k$, $\gamma_{\max} \leftarrow \vlambda(\va)/\left( 1 - \vlambda(\va)\right)$\;}
$\gamma_k \leftarrow \argmin_{\gamma\in [0,\gamma_{\max}]} f\left(\vx_k + \gamma \vd \right)$\;
$\vx_{k+1} \leftarrow \vx_k + \gamma_k \vd$\;
Update $\cs_k$ and $\vlambda_k$ (see full details in Algorithm~\ref{algo:Appx:ACGSteps} in Appendix~\ref{Section:Appx:ACG})\;
\caption{Away-step Conditional Gradients step ACG$\left( \nabla f(\vx_k), \vx_k , \cs_k, \vlambda_k \right)$} \label{algo:ACGStep}
\end{algorithm}

\subsubsection{Global convergence} \label{Subsection:CGAlgGlobalConv}

The first proof of asymptotic linear convergence of the ACG algorithm relied on the \emph{strict complementarity} of the problem in Equation~\eqref{eq:OptProblem} (shown in Assumption~\ref{assumption:strictComplementarity}), which we will also use in the convergence proof of the SOCGS algorithm. A mild assumption that rules out degeneracy.

\begin{assumption}[Strict Complementarity]\label{assumption:strictComplementarity}
We have that $\innp{\nabla f\left( \vx^*\right), \vx - \vx^*} = 0$ if and only if $\vx \in \mathcal{F} \left( \vx^*\right)$.
\end{assumption} 

If Assumption~\ref{assumption:strictComplementarity} is satisfied the iterates of the ACG algorithm reach $\mathcal{F} \left( \vx^*\right)$ in a finite number of steps, remaining in $\mathcal{F} \left( \vx^*\right)$ for all subsequent iterations \cite{guelat1986some}. When inside $\mathcal{F} \left( \vx^*\right)$, the iterates of the ACG algorithm contract the primal gap linearly. This analysis was later significantly extended to provide an explicit global linear convergence rate in primal gap (Theorem~\ref{theorem:ConvergenceACG}), by making use of the \emph{pyramidal width} of the polytope $\cx$ \cite{lacoste2015global}. With the pyramidal width one can derive a primal progress guarantee for all steps taken by the ACG algorithm except \lq{}bad\rq{} away-steps that reduce the cardinality of the active set $\mathcal{S}_k$, that is when $\innp{\nabla f(\vx_k),  \vx_k - \mathbf{v}} <\innp{\nabla f(\vx_k),\va - \vx_k}$ and the step size satisfies $\gamma_k = \gamma_{\max}$ in Algorithm~\ref{algo:ACGStep}. This cannot happen more than $\lfloor K/2 \rfloor$ times when running the ACG algorithm for $K$ iterations (as the algorithm cannot drop more vertices with away-steps than it has picked up with Frank-Wolfe steps). This is an important consideration to keep in mind, as it means that the ACG linear primal gap contraction does not hold on a per-iteration basis. 

\begin{theorem}[Primal gap convergence of the ACG algorithm]~\citep[Theorem 1]{lacoste2015global} \label{theorem:ConvergenceACG}
  Given an initial point $\vx_0 \in \cx$, the ACG algorithm applied to Problem~\eqref{eq:OptProblem} satisfies after $K \geq 0$ iterations:
   \begin{align*}
      f(\vx_{K}) - f(\vx^*) \leq \left( 1 - \frac{\mu}{4L} \left( \frac{\delta}{D}\right)^2\right)^{ K/2} \left(f(\vx_{0}) - f(\vx^*) \right),
  \end{align*}
  where $D$ denotes the diameter of the polytope $\cx$ and $\delta$ its pyramidal width.
\end{theorem}

The CG algorithm and its variants make heavy use of the linear approximation $\hat{l}_{k}(\vx)$ in Equation~\eqref{def:LinearApprox}. What if we consider a quadratic approximation of $f(\vx)$, as opposed to a linear approximation?

 \subsection{Projected Variable-Metric algorithms} \label{Section:ProjVarMet}
 We define the quadratic approximation of the function $f(\vx)$ around the point $\vx_k$ using a matrix $H_k \in \mathcal{S}^n_{++}$, denoted by $\hat{f}_{k}(\vx)$ as:
\begin{align}
    \hat{f}_{k}(\vx) \defeq f(\vx_k) + \innp{\nabla f(\vx_k), \vx - \vx_k} + \frac{1}{2} \norm{\vx - \vx_k}_{H_k}^2. \label{def:approx}
\end{align}
Intuitively, $\hat{f}_{k}(\vx)$ will be a good local approximation to $f(\vx)$ around $\vx_k$ if $H_k$ is a good approximation to $\nabla^2 f(\vx_k)$. In this case, the quadratic approximation to $\hat{f}_{k}(\vx)$ will contain more information about the local curvature of the function $f(\vx)$ than the linear approximation $\hat{l}_{k}(\vx)$. Methods that minimize quadratic approximations of the function $f\left( \vx\right)$ over $\cx$ to define iterates are commonly known as \emph{Projected Variable-Metric} (PVM) algorithms \cite{nesterov2018lectures, nemirovski2020OPTIII}. These methods could, for example, set $\vx_{k+1} = \vx_{k} + \gamma_k (\argmin_{\vx \in \cx}  \hat{f}_{k}(\vx) - \vx_k)$, with $\gamma_k \in (0,1]$.

Minimizing the approximation $\hat{f}_{k}(\vx)$ over $\cx$ can be interpreted as a scaled projection operation onto $\cx$, which is why these methods are considered projection-based, as opposed to the CG algorithm.

\begin{remark}
  \label{fact:equivNewton}
Minimizing $\hat{f}_{k}(\vx)$ over $\mathcal{X}$ can be viewed as the $H_k$-scaled projection of $\vx_k - H^{-1}\nabla f(\vx_k)$ onto $\mathcal{X}$, namely:
\begin{align}
    \argmin_{\vx \in \mathcal{X}} \hat{f}_k \left( \vx\right) & = \Pi^{H_k}_{\cx}\left(\vx_k - H_k^{-1} \nabla f(\vx_k) \right). \label{Eq:VarMetricStep}
\end{align}
\end{remark}
 
  We can recover many well-known algorithms from the PVM formulation, for example, if we set $H_k = \nabla^2 f\left( \vx_k \right)$ in Equation~\eqref{Eq:VarMetricStep} we recover the \emph{Projected Newton} algorithm. Alternatively, if we use $H_k = I^n$ we recover the \emph{Projected Gradient Descent} (PGD) algorithm. 

\subsubsection{Local convergence} \label{Subsection:VarMetricAlgLocalConv}

One of the most attractive features of the Projected Newton algorithm with $\gamma_k = 1$ when applied to Problem~\eqref{eq:OptProblem} is its local quadratic convergence in distance to $\vx^*$. This property also extends to PVM algorithms if $H_k$ approximates $\nabla^2 f \left( \vx_k\right)$ sufficiently well as $\vx_k$ approaches $\vx^*$. What do we mean by sufficiently well? As $f(\vx)$ is strongly convex we know that for any $H_k \in \mathcal{S}^n_{++}$ and $\vy \in \mathcal{X}$, then for $\vd = \vy - \vx_k$:
%\begin{align}
%    \frac{1}{\eta_k} \norm{\vy - \vx_k}^2_{H_k} &\leq \norm{\vy - \vx_k}^2_{\nabla^2 f(\vx_k)} \\
%    &\leq \eta_k \norm{\vy - \vx_k}^2_{H_k}, \label{Eq:ScalingHessian}
%\end{align}
\begin{align}
    \frac{1}{\eta_k} \norm{\vd}^2_{H_k} &\leq \norm{\vd}^2_{\nabla^2 f(\vx_k)}\leq \eta_k \norm{\vd}^2_{H_k}, \label{Eq:ScalingHessian}
\end{align}
for $\eta_k = \max\{ \lambda_{\max}( H_k^{-1} \nabla^2 f(\vx_k)),  \lambda_{\max}([\nabla^2 f(\vx_k) ]^{-1}H_k )\}$ and $\eta_k \geq 1$ (see Lemma~\ref{lemma:scalingCondition} in Appendix~\ref{Section:Appx:HessianApprox}). The parameter $\eta_k$ can be used to measure how well $H_k$ approximates $\nabla^2 f(\vx_k)$, and will serve as our accuracy parameter. The chain of inequalities shown in Equation~\eqref{Eq:ScalingHessian} is presented as Assumption C in \citet{karimireddy2018global}, where it is used to prove the global convergence of an Inexact Projected Variable-Metric variant. Using $H_k = \nabla^2 f(\vx_k)$ we recover $\eta_k = 1$. We assume that we have access to an oracle $\Omega: \cx \rightarrow \mathcal{S}^n_{++}$ that returns estimates of the Hessian that satisfy:
\begin{assumption}[Accuracy of Hessian oracle $\Omega$]\label{assumption:accuracyHessian}
The oracle $\Omega$ queried with a point $\vx$ returns a matrix $H$ with a parameter $\eta$ such that:
\begin{align} \label{Eq:accuracyParam}
    \frac{\eta - 1}{\norm{\vx - \vx^*}^2} \leq \omega.
\end{align}
Where $\eta = \max\{ \lambda_{\max}( H^{-1} \nabla^2 f(\vx)),  \lambda_{\max}([\nabla^2 f(\vx) ]^{-1}H )\}$ and $\omega \geq 0$ denotes a known constant. 
\end{assumption} 

Intuitively, the accuracy of the oracle improves as the oracle is queried with points closer to $\vx^*$. If the oracle returns $\Omega\left(\vx \right) = \nabla^2 f\left(\vx\right)$ for all $\vx \in \mathcal{X}$ then $\omega = 0$. This assumption allows us to obtain local quadratic convergence in distance to $\vx^*$ for the simplest PVM algorithm, i.e., $\vx_{k+1} = \vx_{k+1}^* = \argmin_{\vx \in \cx}  \hat{f}_{k}(\vx)$, as shown in Theorem~\ref{theorem:convervenceExactProj} (see Corollary~\ref{Corollary:Appx:ConvExact} in Appendix~\ref{appx:VarMetricAlgo}).

\begin{remark}
Note that finding a matrix $H$ satisfying Assumption~\ref{assumption:accuracyHessian} at $\vx$, given a fixed $\omega$ requires knowledge of a tight lower bound on $\norm{\vx - \vx^*}$.
\end{remark}

\begin{theorem}[Local quadratic convergence of vanilla PVM algorithm] \label{theorem:convervenceExactProj}
Given an $L$-smooth and $\mu$-strongly convex function with $L_2$-Lipschitz Hessian and a convex set $\mathcal{X}$ if Assumption~\ref{assumption:accuracyHessian} is satisfied, and we set $\vx_{k+1} = \vx_{k+1}^* = \argmin_{\vx \in \cx}  \hat{f}_{k}(\vx)$ we have for all $k \geq 0$:
\begin{align*}
   \norm{\vx_{k+1} - \vx^*} & \leq \sqrt{\frac{\eta_k^2 L_2^2}{4\mu^2} + \frac{2 L \eta_k\omega}{\mu}}\norm{\vx_{k} - \vx^*}^2,
\end{align*}
for $\eta_k= \max\{ \lambda_{\max}( H_k^{-1} \nabla^2 f(\vx_k)),  \lambda_{\max}([\nabla^2 f(\vx_k) ]^{-1}H_k )\}$.
\end{theorem}

\subsubsection{Global convergence} \label{Subsection:VarMetricAlgGlobalConv}

One of the key questions that remains to be answered in this section is how PVM algorithms behave globally. For Problem~\eqref{eq:OptProblem} the vanilla PVM algorithm with unit step size will converge globally, and if we use bounded step sizes, or a exact line search, we can show that the primal gap contracts linearly (Theorem~\ref{Theorem:GlobalRateExact}). The global convergence of these methods can be recast in terms of a notion related to the multiplicative stability of the Hessian, allowing for elegant proofs of convergence \cite{karimireddy2018global}.

\begin{theorem}[Primal gap convergence of vanilla PVM algorithm with line search]~\citep[Theorem 4]{karimireddy2018global} \label{Theorem:GlobalRateExact}
 Given an $L$-smooth and $\mu$-strongly convex function and a convex set $\mathcal{X}$ then the vanilla PVM algorithm with an exact line search or with a step size $\gamma_k = \frac{\mu}{L\eta_k}$ guarantees for all $k \geq 0$:
   \begin{align*}
      f(\vx_{k+1}) - f(\vx^*) \leq \left( 1 - \frac{\mu^3}{L^3\eta_k^3}\right) \left(f(\vx_{k}) - f(\vx^*) \right),
  \end{align*} 
  for $\eta_k= \max\{ \lambda_{\max}( H_k^{-1} \nabla^2 f(\vx_k)),  \lambda_{\max}([\nabla^2 f(\vx_k) ]^{-1}H_k )\}$.
 \end{theorem}
 
 Note that as $\eta_k\geq 1$ for all $H_k \in \mathcal{S}^n_{++}$ the primal gap will always contract regardless of how badly chosen $H_k$ is. Moreover, the best we can do is to choose $H_k = \nabla^2 f \left( \vx_k\right)$, which results in $\eta_k = 1$.

\section{Second-order Conditional Gradient Sliding Algorithm}

 The discussion of PVM algorithms in Section~\ref{Section:ProjVarMet} did not address two important concerns:

 %\begin{enumerate}[wide, labelwidth=!, labelindent=10pt]
 \begin{enumerate}
      
      \item The PVM algorithm requires computing a scaled projection at every iteration. These projections are usually too expensive to compute to optimality. Ideally we would want to solve these scaled projection problems to a certain accuracy, but can we maintain the local quadratic convergence in distance to the optimum shown in Theorem~\ref{theorem:convervenceExactProj} when computing approximate scaled projections? 
       
      \item The global convergence rate of the PVM algorithm with exact line search and perfect Hessian information (Theorem~\ref{Theorem:GlobalRateExact} with $\eta_k = 1$) has a worse dependence on the condition number $L/\mu$ than the convergence rate of the PGD and the ACG algorithm (see Theorem~\ref{theorem:ConvergenceACG} for the latter). Can we couple Inexact PVM steps with ACG steps and improve the global convergence rate in Theorem~\ref{Theorem:GlobalRateExact}? 
       
 \end{enumerate}

The Second-order Conditional Gradient Sliding (SOCGS) algorithm (Algorithm~\ref{algo:proj-Newton}) is designed with these considerations in mind, providing global linear convergence in primal gap and local quadratic convergence in primal gap and distance to $\vx^*$. The algorithm couples an independent ACG step with line search (Line~\ref{algLine:ACG}) with an Inexact PVM step with unit step size (Lines~\ref{algLine:PNStep1}-\ref{algLine:PNStep3}). At the end of each iteration we choose the step that provides the greatest primal progress (Lines~\ref{algLine:MonotonicityFirst}-\ref{algLine:ACGStep}). The ACG steps in Line~\ref{algLine:ACG} will ensure global linear convergence in primal gap, and the Inexact PVM steps in Lines~\ref{algLine:MonotonicityFirst}-\ref{algLine:ACGStep} will provide quadratic convergence. Note that the ACG iterates in Line~\ref{algLine:ACG} do not depend on the Inexact PVM steps in Lines~\ref{algLine:PNStep1}-\ref{algLine:PNStep3}. This is because the ACG steps do not contract the primal gap on a per-iteration basis (see discussion in Section~\ref{Subsection:CGAlgGlobalConv}).

We compute the scaled projection in the Inexact PVM step (Lines~\ref{algLine:PNStep1}-\ref{algLine:PNStep3}) using the ACG algorithm with exact line search, thereby making the SOCGS algorithm (Algorithm~\ref{algo:proj-Newton}) projection-free. As the function being minimized in the Inexact PVM steps is quadratic there is a closed-form expression for the optimal step size for $\hat{f}_k\left( \vx\right)$ in Line~\ref{algLine:PNStep2}. The scaled projection problem is solved to an accuracy $\varepsilon_k$ such that $\hat{f}_{k}(\tilde{\vx}_{k+1}) - \min_{\vx \in \cx} \hat{f}_k\left( \vx\right) \leq \varepsilon_k$, using the Frank-Wolfe gap as a stopping criterion, as in the CGS algorithm \cite{lan2016conditional}. The accuracy parameter $\varepsilon_k$ in the SOCGS algorithm depends on a lower bound on the primal gap of Problem~\ref{eq:OptProblem} which we denote by $lb\left( \vx_k \right)$ that satisfies $lb\left( \vx_k \right) \leq f\left(\vx_k \right) - f\left(\vx^* \right)$.

\begin{remark}[Removing line search]
The line search in the independent ACG step (Line~\ref{algLine:ACGStep}) can be substituted with a step size strategy that requires knowledge of the $L$-smoothness parameter of $f(\vx)$ \cite{pedregosa2020linearly}, while maintaining the convergence rate shown in Theorem~\ref{theorem:ConvergenceACG}, and avoiding the costly line search for $f(\vx)$.
\end{remark}

%\begin{remark}
%  For any $\vx_k\in \cx$ we can compute a lower bound on the primal gap of Problem~\eqref{eq:OptProblem} bounded away from zero using any CG variant that monotonically decreases the primal gap. It suffices to run an arbitrary number of steps $n$ of the aforementioned variant to minimize $f(\vx)$ starting from $\vx_k$, resulting in $\vx_k^n\in \cx$. Simply noting that $f(\vx_k^n)\geq f(\vx^*)$ allows us to conclude that $f(\vx_k) - f(\vx^*)  \geq f(\vx_k) - f(\vx_k^n)$, and therefore a valid lower bound is $lb\left( \vx_k \right)  = f(\vx_k) - f(\vx^n_k)$. The higher the number of CG steps performed from $\vx_k$, the tighter the resulting lower bound will be.
%\end{remark}

\begin{remark}[Estimating $lb\left( \vx_k \right)$]
  For any $\vx_k\in \cx$ we can compute a lower bound on the primal gap of Problem~\eqref{eq:OptProblem} bounded away from zero using the primal gap progress guarantees from the vanilla CG algorithm that follow from smoothness. Namely, taking a step $\gamma_k = \min \{1, \max_{\mathbf{v} \in \cx}\langle\nabla f(\vx_k),\vx_k-\mathbf{v}\rangle/(LD^2) \}$ and using the $L$-smoothness of $f$, we have that:
  \begin{align*}
      f(\vx_k) - f(\vx_{k+1}) & \geq \gamma_k \max_{\mathbf{v} \in \cx}\langle\nabla f(\vx_k),\vx_k-\mathbf{v}\rangle - \gamma_k^2 \frac{L}{2} D^2.
  \end{align*}
  Considering the case where $1 \leq \max_{\mathbf{v} \in \cx}\langle\nabla f(\vx_k),\vx_k-\mathbf{v}\rangle/(LD^2)$ allows us to conclude that $f(\vx_k) - f(\vx_{k+1}) \geq  LD^2/2$. Conversely, the case where $1 \geq \max_{\mathbf{v} \in \cx}\langle\nabla f(\vx_k),\vx_k-\mathbf{v}\rangle/(LD^2)$ allows us to conclude that $f(\vx_k) - f(\vx_{k+1}) \geq  (\max_{\mathbf{v} \in \cx}\langle\nabla f(\vx_k),\vx_k-\mathbf{v}\rangle)^2/(LD^2)$. This allows us to bound:
  \begin{align*}
      f(\vx_k) - f(\vx^*) & \geq f(\vx_k) - f(\vx_{k+1}) \\
      & \geq \min\{ LD^2/2, (\max_{\mathbf{v} \in \cx}\langle\nabla f(\vx_k),\vx_k-\mathbf{v}\rangle)^2/(LD^2)\}.
  \end{align*}
  Note that this guarantee also holds if we use a line search instead of the step size described above, as the line search is guaranteed to make at least as much progress. Computing the aforementioned quantity comes at no extra cost if $L$ and $D$ are known, as the Frank-Wolfe vertex from Line~\ref{algLine:ACG} of Algorithm~\ref{algo:proj-Newton} can be reused. Alternatively one could use any CG variant that monotonically decreases the primal gap. It suffices to run an arbitrary number of steps $n$ of the aforementioned variant to minimize $f(\vx)$ starting from $\vx_k$, resulting in $\vx_k^n\in \cx$. Simply noting that $f(\vx_k^n)\geq f(\vx^*)$ allows us to conclude that $f(\vx_k) - f(\vx^*)  \geq f(\vx_k) - f(\vx_k^n)$, and therefore a valid lower bound is $lb\left( \vx_k \right)  = f(\vx_k) - f(\vx^n_k)$. The higher the number of CG steps performed from $\vx_k$, the tighter the resulting lower bound will be.
\end{remark}

\begin{remark}[Assuming knowledge of a lower bound]
  In several machine learning applications the value of $f(\vx^*)$ is known a priori, such is the case of the approximate Carath\'{e}odory problem \cite{mirrokni2017tight, combettes2019revisiting} where $f(\vx^*) = 0$. In other applications, estimating $f(\vx^*)$ is easier than estimating the strong convexity parameter (see \cite{barre2020complexity, barre2019polyak,asi2019importance, hazan2019revisiting} for an in-depth discussion). This allows for tight lower bounds on the primal gap.
\end{remark}

\begin{algorithm}
\SetKwInOut{Input}{Input}\SetKwInOut{Output}{Output}
%\SetKwComment{Comment}{$\triangleright$\ }{}
\SetKwComment{Comment}{//}{}
\Input{Point $\vx \in \cx$}
\Output{Point $\vx_{K } \in \cx$}
\hrulealg
$\vx_0 \leftarrow \argmin_{\mathbf{v} \in  \mathcal{X}} \innp{\nabla f\left(\vx \right), \mathbf{v}}$, $\cs_0 \leftarrow \{ \vx_0 \}$, $\vlambda_0(\vx_0) \leftarrow 1$\;
$\vx_0^{\text{ACG}} \leftarrow \vx_0$, $\cs_0^{\text{ACG}} \leftarrow \cs_0$, $\vlambda_0^{\text{ACG}}(\vx_0) \leftarrow 1$\;
\For{$k = 0$ \textbf{to} $K - 1$}{
%\tikzmk{A}
$\vx^{\text{ACG}}_{k + 1}, \cs^{\text{ACG}}_{k + 1}, \vlambda^{\text{ACG}}_{k + 1} \leftarrow $ ACG$\left( \nabla f(\vx_k), \vx^{\text{ACG}}_{k} , \cs^{\text{ACG}}_{k}, \vlambda^{\text{ACG}}_{k} \right)$ \label{algLine:ACG} \Comment*[r]{ACG step}
%\tikzmk{B} \boxit{pink}
%\tikzmk{A}
$H_{k} \leftarrow \Omega \left(\vx_{k}  \right)$ \label{algLine:updateHessian} \Comment*[r]{Call Hessian oracle}
$\hat{f}_k\left( \vx\right) \leftarrow \innp{\nabla f\left( \vx_k\right), \vx - \vx_k} + \frac{1}{2} \norm{\vx - \vx_k}^2_{H_k}$ \label{algLine:buildQuadratic} \Comment*[r]{Build quadratic approximation}
$\varepsilon_k \leftarrow \left(\frac{lb\left( \vx_k\right)}{\norm{\nabla f\left( \vx_k\right)}}\right)^4$ \label{AlgLine:AccuracyParameter}\;
$\tilde{\vx}^0_{k + 1} \leftarrow \vx_k$, $\tilde{\cs}^0_{k + 1} \leftarrow \cs_k$, $\tilde{\vlambda}^0_{k + 1} \leftarrow \vlambda_k$, $t \leftarrow 0$ \label{algLine:TransferPoint} \; 
\While(\tcp*[f]{Compute Inexact PVM step}){$\max\limits_{\mathbf{v} \in \cx}   \langle \nabla \hat{f}_k ( \tilde{\vx}^t_{k + 1} ), \tilde{\vx}^t_{k + 1} - \mathbf{v} \rangle \geq \varepsilon_k$ \label{algLine:PNStep1}}{
$\tilde{\vx}^{t+1}_{k + 1}, \tilde{\cs}^{t+1}_{k + 1}, \tilde{\vlambda}^{t+1}_{k + 1} \leftarrow $ ACG$\left(\nabla  \hat{f}_k(\tilde{\vx}^{t}_{k + 1}), \tilde{\vx}^{t}_{k + 1}, \tilde{\cs}^{t}_{k + 1}, \tilde{\vlambda}^{t}_{k + 1} \right)$ \label{algLine:PNStep2} \;
$t \leftarrow t+1$\;
} \label{algLine:PNStep3} 
$\tilde{\vx}_{k + 1} \leftarrow \tilde{\vx}^t_{k + 1} $, $\tilde{\cs}_{k + 1} \leftarrow \tilde{\cs}^t_{k + 1}$, $\tilde{\vlambda}_{k + 1} \leftarrow \tilde{\vlambda}^t_{k + 1}$ \label{algLine:OutputPVM} \;
%\tikzmk{B}
% \boxit{lightblue}
\uIf{$f\left(\tilde{\vx}_{k + 1} \right) \leq f(\vx^{\textup{ACG}}_{k + 1} )$ \label{algLine:MonotonicityFirst}}{
$\vx_{k+1} \leftarrow \tilde{\vx}_{k + 1} $, $\cs_{k+1} \leftarrow \tilde{\cs}_{k + 1}$, $\vlambda_{k+1} \leftarrow \tilde{\vlambda}_{k + 1}$\label{algLine:PVMStep} \Comment*[r]{Choose Inexact PVM step}}
\Else{
$\vx_{k+1} \leftarrow \vx^{\text{ACG}}_{k + 1} $, $\cs_{k+1} \leftarrow \cs^{\text{ACG}}_{k + 1}$, $\vlambda_{k+1} \leftarrow \vlambda^{\text{ACG}}_{k + 1}$\Comment*[r]{Choose ACG step}} \label{algLine:ACGStep}
} \label{algLine:MonotonicityLast}
\caption{Second-order Conditional Gradient Sliding (SOCGS) Algorithm}\label{algo:proj-Newton}
\end{algorithm}

\subsection{Global convergence}

The global convergence rate in primal gap of the SOCGS algorithm (Algorithm~\ref{algo:proj-Newton}) is driven by the ACG steps in Line~\ref{algLine:ACG}, as such:

\begin{theorem}
  \label{theorem:GlobalConvergence}
  %Given an $L$-smooth and $\mu$-strongly convex function and a polytope $\mathcal{X}$ the SOCGS algorithm (Algorithm~\ref{algo:proj-Newton}) with $\vx_0 \in \cx$ satisfies:
  Given $\vx_0 \in \cx$, then the SOCGS algorithm applied to Problem~\eqref{eq:OptProblem} satisfies:
   \begin{align}
      f(\vx_{k}) - f(\vx^*) \leq \left( 1 - \frac{\mu}{4L} \left( \frac{\delta}{D}\right)^2\right)^{ k/2 } \left(f(\vx_{0}) - f(\vx^*) \right), \label{Eq:LinearRate}
  \end{align}
  where $D$ denotes the diameter of the polytope $\cx$ and $\delta$ its pyramidal width.
  \begin{proof}
  As at each step the SOCGS algorithm (Algorithm~\ref{algo:proj-Newton}) chooses between the independent ACG step (Line~\ref{algLine:ACG}) and the Inexact PVM step (Lines~\ref{algLine:PNStep1}-\ref{algLine:PNStep3}) according to which one provides the greatest primal progress, the primal gap convergence in Theorem~\ref{theorem:ConvergenceACG} applies.
  \end{proof}
\end{theorem}

\subsection{Local convergence}

Despite computing inexact scaled projections in Lines~\ref{algLine:PNStep1}-\ref{algLine:PNStep3} of Algorithm~\ref{algo:proj-Newton}, the Inexact PVM steps contract the distance to optimum quadratically when close enough to the optimal solution. %This is shown in Theorem~\ref{Lemma:InexactConvergenceDistance}.

\begin{lemma} \label{Lemma:InexactConvergenceDistance}
Given a $\mu$-strongly convex and $L$-smooth function $f(\vx)$ with $L_2$-Lipschitz Hessian and a convex set $\mathcal{X}$, if Assumption~\ref{assumption:accuracyHessian} is satisfied then the Inexact PVM steps in Lines~\ref{algLine:PNStep1}-\ref{algLine:PNStep3} of Algorithm~\ref{algo:proj-Newton} satisfy for all $k\geq 0$:

 \begin{align}
    \norm{\tilde{\vx}_{k+1} - \vx^*} &\leq  \frac{\sqrt{\eta_k}}{2\mu}\left(   \sqrt{8\mu}\left(1  + \sqrt{L\omega}\right)  + \sqrt{\eta_k}L_2 \right) \norm{\vx_k - \vx^*}^2, \label{Eq:ConvDistOpt} 
\end{align}
 where $\eta_k = \max\{ \lambda_{\max}( H_k^{-1} \nabla^2 f(\vx_k)),  \lambda_{\max}([\nabla^2 f(\vx_k) ]^{-1}H_k )\}$ and $\omega \geq 0$ denotes a constant.
\end{lemma}

In order to take advantage of the quadratic convergence in distance to the optimum shown in Lemma~\ref{Lemma:InexactConvergenceDistance}, we need to show that at some point the SOCGS algorithm will always choose in Lines~\ref{algLine:MonotonicityFirst}-\ref{algLine:MonotonicityLast} the Inexact PVM step defined in Lines~\ref{algLine:PNStep1}-\ref{algLine:PNStep3}. To be more specific, we show that the convergence in primal gap for the Inexact PVM step will also be quadratic. We do this by first showing that there is an iteration $K\geq 0$ such that for all $k \geq K$  we have $\vx_k \in \mathcal{F}\left( \vx^*\right)$ (Lemma~\ref{Corollary:Appx:StuckToFace} in Appendix~\ref{Section:Appx:InexactScaledProjections}).

 \begin{lemma}\label{Theorem:ConvOptFace2}
Given a $\mu$-strongly convex and $L$-smooth function $f(\vx)$ with $L_2$-Lipschitz continuous Hessian and a polytope $\cx$, if Assumption~\ref{assumption:strictComplementarity} and \ref{assumption:accuracyHessian} are satisfied, then there is an index $K\geq 0$ such that for $k \geq K$ we have that $\vx_k\in \mathcal{F}(\vx^*)$, that is, both the Inexact PVM steps (Lines~\ref{algLine:PNStep1}-\ref{algLine:PNStep3} of Algorithm~\ref{algo:proj-Newton}) and the ACG step  (Line~\ref{algLine:ACG} of Algorithm~\ref{algo:proj-Newton}) are in $\mathcal{F}(\vx^*)$.
\end{lemma}

 We can upper bound the right-hand side of Equation~\eqref{Eq:ConvDistOpt} using strong convexity, and the left-hand side using smoothness, Lemma~\ref{Theorem:ConvOptFace2} and strict-complementarity (Assumption~\ref{assumption:strictComplementarity}). This allows us to show that there exists an iteration after which the primal progress of the Inexact PMV steps in Lines~\ref{algLine:PNStep1}-\ref{algLine:PNStep3} will be quadratic, which ensures the local quadratic convergence of the SOCGS algorithm.

\begin{theorem}[Quadratic convergence in primal gap of the SOCGS algorithm] \label{Theorem:InexactConvergencePrimalGap}
Given a $\mu$-strongly convex and $L$-smooth function $f(\vx)$ with $L_2$-Lipschitz Hessian and a polytope $\mathcal{X}$, if Assumption~\ref{assumption:strictComplementarity} and Assumption~\ref{assumption:accuracyHessian} are satisfied, then there is a $K\geq 0$ such that for $k\geq K$ the iterates of the SOCGS algorithm (Algorithm~\ref{algo:proj-Newton}) satisfy:
  \begin{align*}
  f(\vx_{k+1}) - f(\vx^*) \leq  \frac{L \eta_k}{2\mu^4}  \left(\sqrt{8\mu} (1+ \sqrt{L\omega}) + \sqrt{\eta_k} L_2\right)^2\left(f(\vx_{k}) - f(\vx^*)\right)^{2}.
 \end{align*}
 where $\eta_k = \max\{ \lambda_{\max}( H_k^{-1} \nabla^2 f(\vx_k)),  \lambda_{\max}([\nabla^2 f(\vx_k) ]^{-1}H_k )\}$ and $\omega \geq 0$ denotes a constant.

 \end{theorem}

\subsection{Complexity analysis}

We defer the full details of the complexity analysis to Section~\ref{Appx:ComplexityAnalysis} in Appendix~\ref{appx:SOCG}. Throughout this section we make the simplifying assumption that we have at our disposal the tightest possible lower bound $lb(\vx_k)$ on the primal gap, that is, $lb(\vx_k) = f(\vx_k) - f(\vx^*)$ (in Remark~\ref{primal_gap_not_available} we address a strategy that can be used when the primal gap is not known). Let $r = \min\{r^\text{ACG}, r^{\text{PVM}}\}>0$ (where $r^\text{ACG}$ is described in Theorem~\ref{Theorem:Appx:ConvOptFace} and $r^{\text{PVM}}$ in Corollary~\ref{Corollary:Appx:StuckToFaceRadius}),  $G = \max_{\vx \in \cx} \norm{\nabla f(\vx)}$ and $\beta = \max\{ (2DG)^{1/4}, (2L(1 + \omega D^2) D^3G)^{1/8} \}$. With these considerations in mind the different oracle complexities are listed in Table~\ref{Table}. As in the classical analysis of PVM algorithms, the SOCGS algorithm shows local quadratic convergence after a number of iterations that is independent of $\varepsilon$ (but dependent on $f(\vx)$ and $\cx$).

\begin{remark} \label{primal_gap_not_available}
Providing a looser lower bound $lb(\vx_k)$ on the primal gap does not affect the number of first-order or Hessian oracle calls, however it can significantly increase the number of linear optimization oracle calls used to compute the Inexact PVM steps in Lines~\ref{algLine:PNStep1}-\ref{algLine:PNStep3}. Note that the progress guarantee from a single ACG step that is not an away-step that drops a vertex is $f(\vx_k) - f(\vx_{k+1}) \geq \frac{\mu}{4L}(\frac{\delta}{D})^2(f(x_k) - f(x^*))$ (see Theorem 1 in \citet{lacoste2015global}). If we use as $lb(\vx_k)$ 
 the progress obtained from such a step (note that $f(\vx_k) - f(\vx^*) \geq f(\vx_k) - f(\vx_{k+1})$) in the complexity analysis, one can obtain after a finite number of iterations, a $\log\log 1/\varepsilon$ complexity in terms of FO and Hessian oracle calls and $\log (1/\varepsilon)\log(\log 1/\varepsilon)$ LO calls, but with worse constants then the ones in Table~\ref{Table}. 
\end{remark}

\begin{table*}[th!]
 \begin{center} 
  {\renewcommand{\arraystretch}{2.0}
\begin{tabular}{ccc}
\toprule
\textbf{Phase} & \textbf{FO and Hessian Oracle Calls} & \textbf{LO Oracle Calls} \\ \hline
\midrule
 Initial Phase          &    $\mathcal{O}\left(  \left(\frac{L(1 + \omega D^2)}{\mu}\right)^2\left(\frac{D}{\delta} \right)^4 \log\left(\frac{1}{\mu r^2}\right)\log \left( \frac{\beta G}{\mu r^2}\right)\right)$  & $\mathcal{O}\left(  \left(\frac{L(1 + \omega D^2)}{\mu}\right)^2\left(\frac{D}{\delta} \right)^4 \log\left(\frac{1}{\mu r^2}\right)\log \left( \frac{\beta G}{\mu r^2}\right)\right)$  \\ \hline
 Final Phase  & $\mathcal{O}\left(\log\log \left(\frac{1}{\varepsilon}\right)\right)$  & $\mathcal{O}\left(\frac{L(1 + \omega D^2)^2 }{\mu}\left( \frac{D}{\delta}\right)^2 \log \left( \frac{\beta G}{\varepsilon}\right)\log\log \left(\frac{1}{\varepsilon}\right)\right)$  \\ \hline
\end{tabular}
}
\end{center}
 \caption{Complexity to reach an $\varepsilon$-optimal solution to Problem~\eqref{eq:OptProblem} for the SOCGS algorithm.} \label{Table}
\end{table*}

\section{Computations}

We compare the performance of the SOCGS algorithm with that of other projection-free algorithms, and that of Projected-Gradient Descent (PGD). In all experiments we compare against the vanilla CG algorithm, the ACG algorithm, the \emph{Pairwise-Step Conditional Gradients} algorithms (PCG) and the \emph{Lazy ACG} algorithm \cite{braun2017lazifying} (ACG (L)).  In the first experiment we also compare against the \emph{Decomposition Invariant Conditional Gradient} (DICG) algorithm \cite{garber2016linear}, the CGS algorithm \cite{lan2016conditional} and the \emph{Stochastic Variance-Reduced Conditional Gradients} (SVRCG) algorithm \cite{hazan2016variance}. We were not able to achieve acceptable performance with the CGS algorithm in the second and third experiment and with the SVRFW algorithm in the third experiment. Lastly we also compare against the \emph{Newton Conditional Gradients} (NCG) algorithm \cite{liu2020newton} which is similar in spirit to the SOCGS algorithm, in the second and third experiment. One of the key features of the NCG algorithm is that it does not require an exact line search strategy, as it provides a specific step size strategy (however it requires selecting five hyperparameters and using an exact Hessian).

In the first problem the Hessian oracle will be inexact, but will satisfy Assumption~\ref{assumption:accuracyHessian} with $\omega = 0.1$, moreover we will also assume knowledge of the primal gap, by first computing a solution to high accuracy. In the remaining problems the Hessian oracle will be exact, and we will assume that we do not have knowledge of the primal gap, and will use the strategy outlined in Remark~\ref{primal_gap_not_available}. In the second experiment, in addition to using the exact Hessian, we will also implement SOCGS with an LBFGS  Hessian update (SOCGS LBFGS) (note that this does not satisfy Assumption~\ref{assumption:accuracyHessian}). All the line searches that do not have a closed form solution are computed using a golden-section bounded line search between $0$ and $1$. The full details of the implementation can be found in Appendix~\ref{Appx:Computations}. In the second and third experiment we will also cap the maximum number of inner iterations to $1000$ for the SOCGS and NCG algorithms, as is done in the computational experiments of NCG and SVRCG. 

The code used can be found in \href{https://github.com/alejandro-carderera/SOCGS}{\texttt{https://github.com/alejandro-carderera/SOCGS}}.

\begin{remark}[Hyperparameter search for the NCG algorithm]
We tested 27 hyperparameter combinations for the NCG algorithm, and the one that provided the best performance was selected (see Appendix~\ref{Appx:Computations} for the full details).
\end{remark}

\paragraph{Sparse coding over the Birkhoff polytope}

In this example (Figure~\ref{fig:Birkhoff}) we minimize the objective function $f(X) = \sum_{i = 1}^m \norm{\vy_i - X\vz_i}^2$, with $X \in \rr^{n\times n}$, over the Birkhoff polytope. This objective function is strongly convex if the vectors $\vz_i$, with $m\in [1,m]$ form a basis for $\mathbb{R}^n$ (See discussion in Appendix~\ref{Appx:SparseCoding}). We generate synthetic data by creating a matrix $B\in \rr^{n \times n}$ with $n = 80$ entries sampled from a standard normal distribution, and $m$ vectors $\vx\in \rr^n$ (with $m = 10000$ in the first experiment and $m = 100000$ in the second), with entries sampled from a standard normal distribution, in order to form $Z = \left\{\vz_1, \cdots, \vz_m \right\}$. For both the experiments we verified numerically that the resulting objective function is strongly convex. The set of vectors $Y = \left\{\vy_1, \cdots, \vy_m \right\}$ is generated by computing $\vy_i = B\vz_i$ for all $i\in \llbracket1, m\rrbracket$. The starting point for all the algorithms is $I^n$. To implement the projection operation used in PGD we use the interior point solver implemented in CVXOPT \cite{andersen2011interior}, which we have found to be computationally faster than the Douglas-Rachford approach described in \citet{combettes2021complexity}. Note that the use of this implementation only impacts the performance with respect to time, and not with respect to iteration count.

\paragraph{Structured logistic regression over $\ell_1$ unit ball}
In this last experiment (Figure~\ref{fig:LogReg}) we minimize a function of the form $f(x) = 1/m \sum_{i = 1}^{m}\log\left( 1+e^{-y_i \innp{\vx, \vz_i}}\right) + \lambda/2 \norm{\vx}^2$ over the $\ell_1$ unit ball with $\lambda = 0.05$. The labels and samples used are taken from the training set of the \texttt{gissette} \cite{guyon2007competitive} and the \texttt{real-sim} \cite{chang2011libsvm} dataset, where $n = 5000$ and $m = 6000$ and $n = 72309$ and $m = 20958$ respectively. The starting point for all the algorithms is the vector $(1, 0, \cdots, 0)$.

\paragraph{Inverse covariance estimation over spectrahedron}
In the second experiment (Figure~\ref{fig:GLasso}) we minimize the function $f(X) = -\log \det (X+\delta I^n) + \trace\left( S X\right) + \frac{\lambda}{2} \norm{X}^2_F$ with $X \in \rr^{n\times n}$ over the space of positive semidefinite matrices of unit trace, with $\delta = 10^{-5}$ and $\lambda =0.05$. This feasible region is not a polytope, and so the guarantees shown in the paper do not apply as they crucially rely on Theorem~\ref{theorem:ConvergenceACG}, and the pyramidal width of the spectrahedron is zero. However, we include the results to show the promising numerical performance of the method. The matrix $S$ is generated by computing a random orthonormal basis $\mathcal{B} = \{ \mathbf{v}_1, \cdots,  \mathbf{v}_m \}$ in $\rr^m$ and computing $S = \sum_{i=1} \sigma_i \mathbf{v}_i\mathbf{v}_i^T$, where $\sigma_i$ is uniformly distributed between $0.5$ and $1$ for $i\in \llbracket1, m\rrbracket$.  The starting point for all the algorithms is the matrix $1/n I^n$.

\begin{figure*}[th!]
    \centering
    \vspace{-10pt}
    \hspace{\fill}
    \subfloat[Iterations]{{\includegraphics[width=3.9cm]{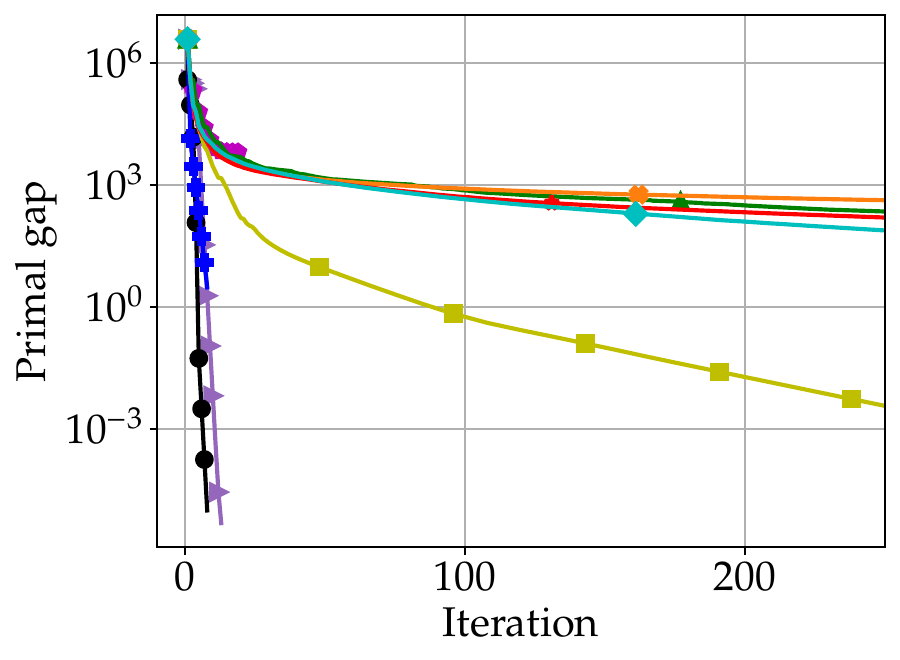} }\label{fig:BirkhoffPGIt}}%
    %\qquad
    \hspace{\fill}
    \subfloat[Seconds]{{\includegraphics[width=3.9cm]{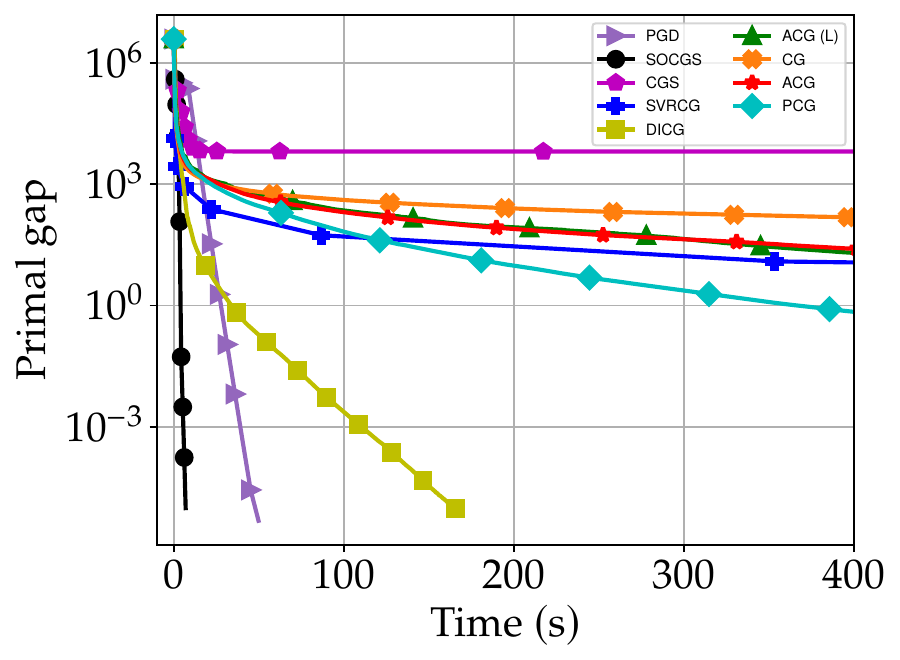}}\label{fig:BirkhoffPGTime} }%
    \hspace{\fill}
    \subfloat[Iterations]{{\includegraphics[width=3.8cm]{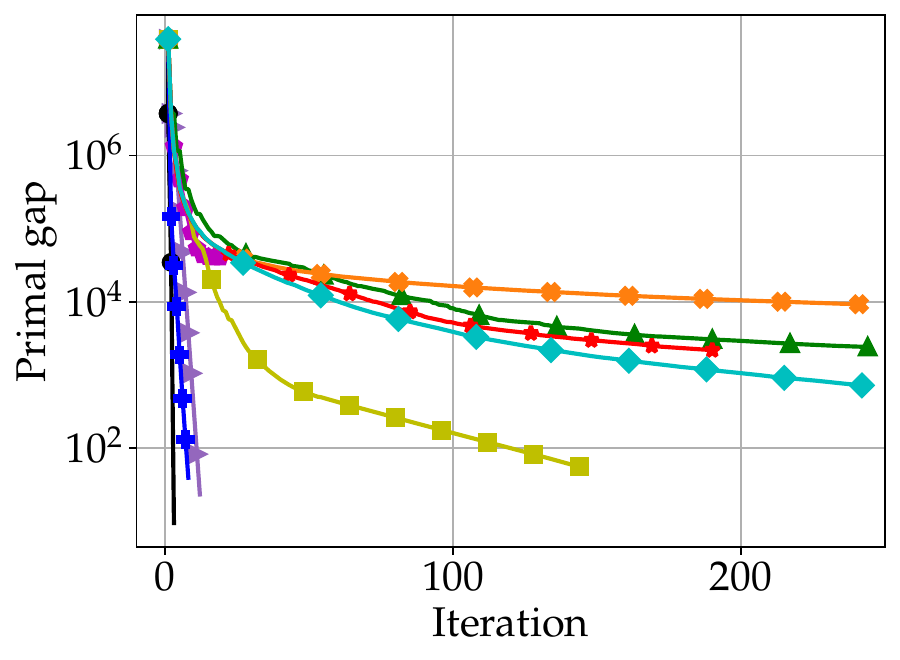} }\label{fig:BirkhoffDistanceIt}}%
    %\qquad
    \hspace{\fill}
    \subfloat[Seconds]{{\includegraphics[width=3.9cm]{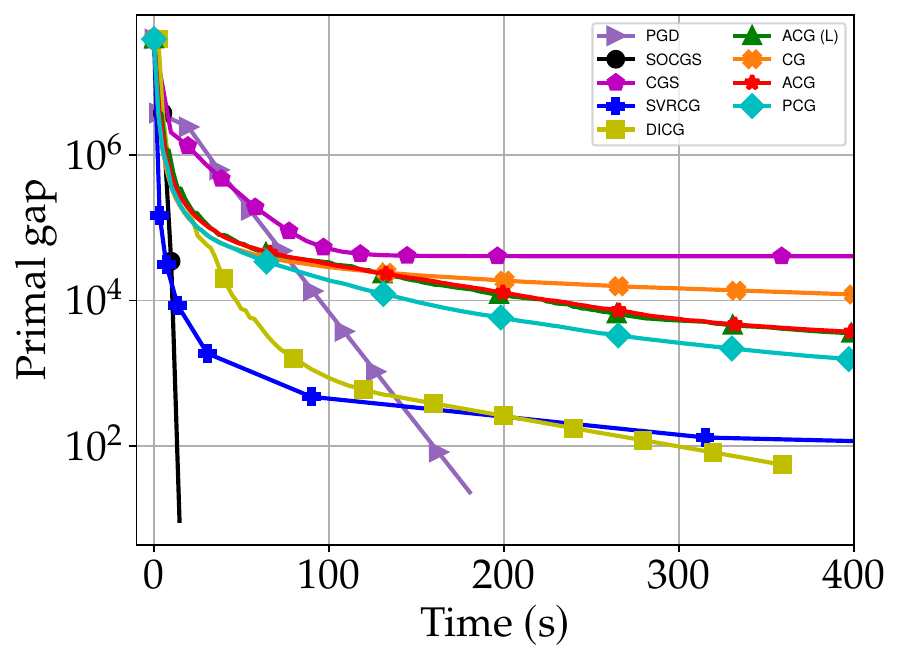} }\label{fig:BirkhoffDistanceTime}}%
    \hspace*{\fill}
    \caption{Birkhoff polytope: Primal gap comparison for $m = 10000$ \protect\subref{fig:BirkhoffPGIt},\protect\subref{fig:BirkhoffPGTime} and $m = 100000$ \protect\subref{fig:BirkhoffDistanceIt},\protect\subref{fig:BirkhoffDistanceTime}.}%
    \label{fig:Birkhoff}%
\end{figure*}

\begin{figure*}[th!]
    \centering
    \vspace{-10pt}
    \hspace{\fill}
    \subfloat[Iterations]{{\includegraphics[width=3.9cm]{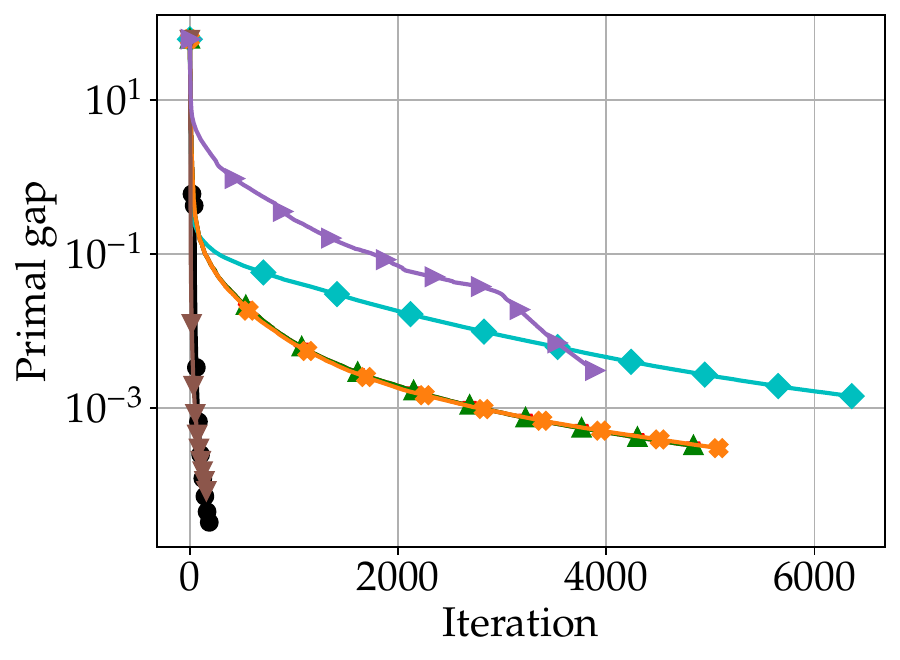} }\label{fig:LogisticRegPGIt}}%
    %\qquad
    \hspace{\fill}
    \subfloat[Seconds]{{\includegraphics[width=3.9cm]{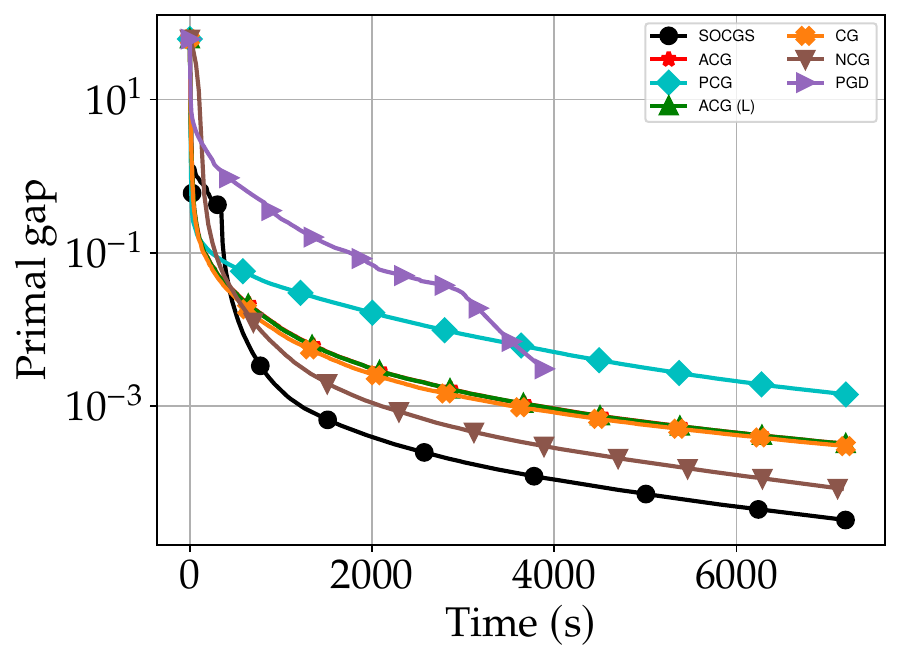} }\label{fig:LogisticRegPGTime}}%
    \hspace{\fill}
    \subfloat[Iterations]{{\includegraphics[width=3.9cm]{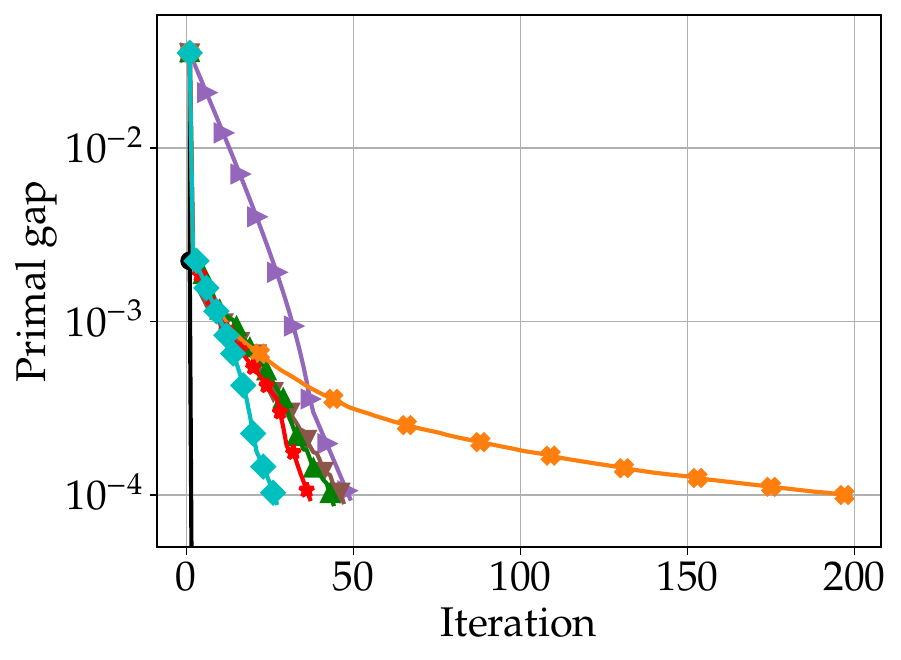} }\label{fig:LogisticRegDistanceIt}}%
    %\qquad
    \hspace{\fill}
    \subfloat[Seconds]{{\includegraphics[width=3.9cm]{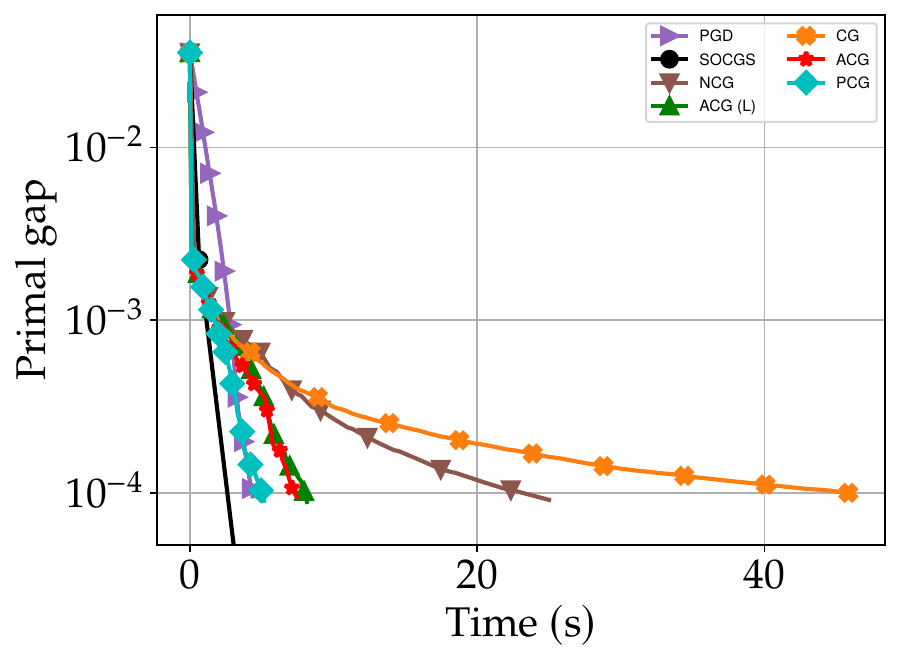} }\label{fig:LogisticRegDistanceTime}}%
    \hspace*{\fill}
    \caption{$\ell_1$-ball: Comparison in terms of primal gap for the \texttt{gissette} \protect\subref{fig:LogisticRegPGIt},\protect\subref{fig:LogisticRegPGTime} and the \texttt{real-sim} \protect\subref{fig:LogisticRegDistanceIt},\protect\subref{fig:LogisticRegDistanceTime} datasets.}%
    \label{fig:LogReg}%
\end{figure*}

\begin{figure*}[th!]
    \centering
    \vspace{-10pt}
    \hspace{\fill}
    \subfloat[Iterations]{{\includegraphics[width=3.9cm]{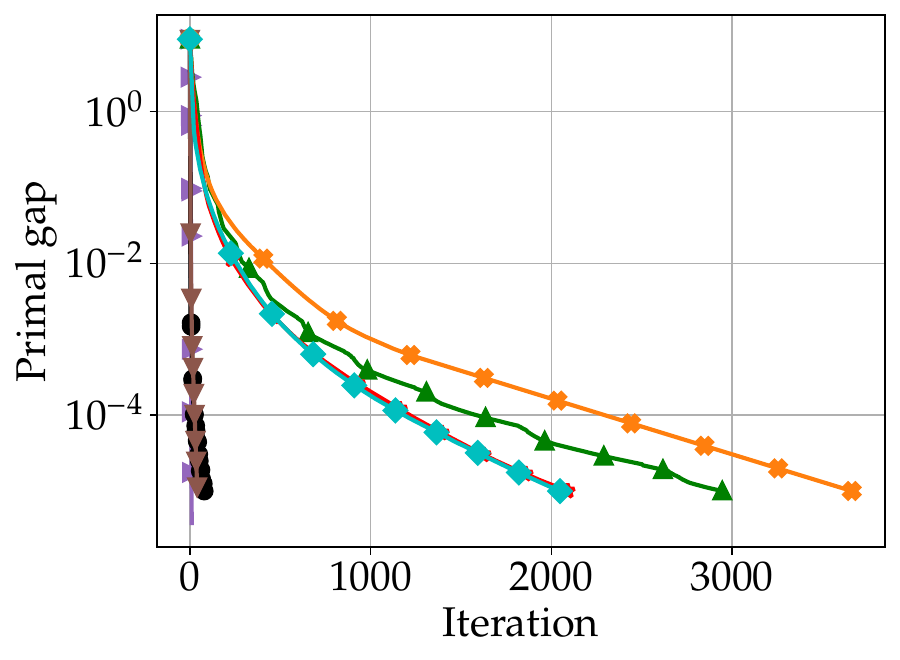} }\label{fig:GLassoPGIt}}%
    %\qquad
    \hspace{\fill}
    \subfloat[Seconds]{{\includegraphics[width=3.9cm]{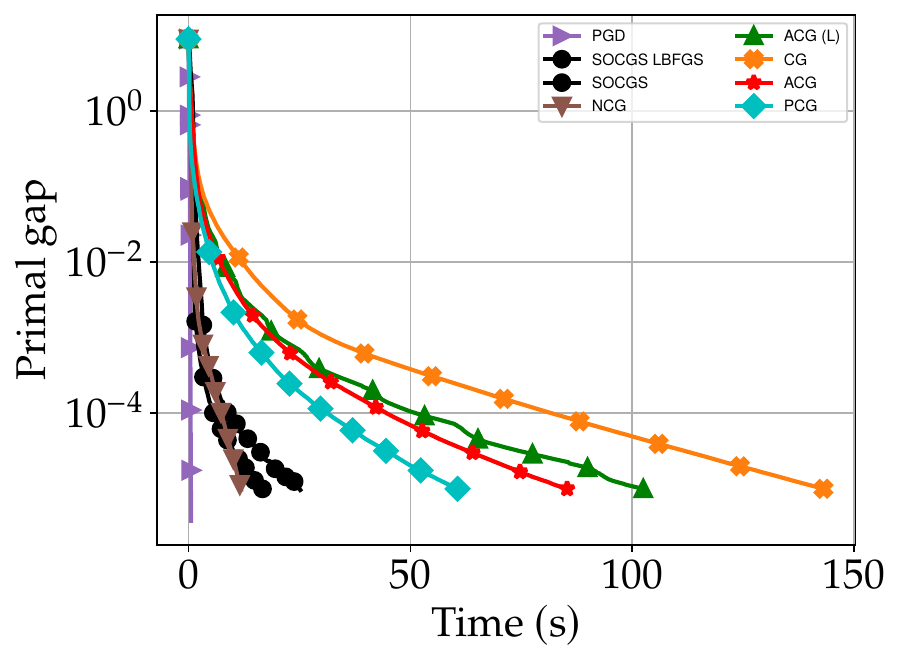} }\label{fig:GLassoPGTime}}%
    \hspace{\fill}
    \subfloat[Iterations]{{\includegraphics[width=3.9cm]{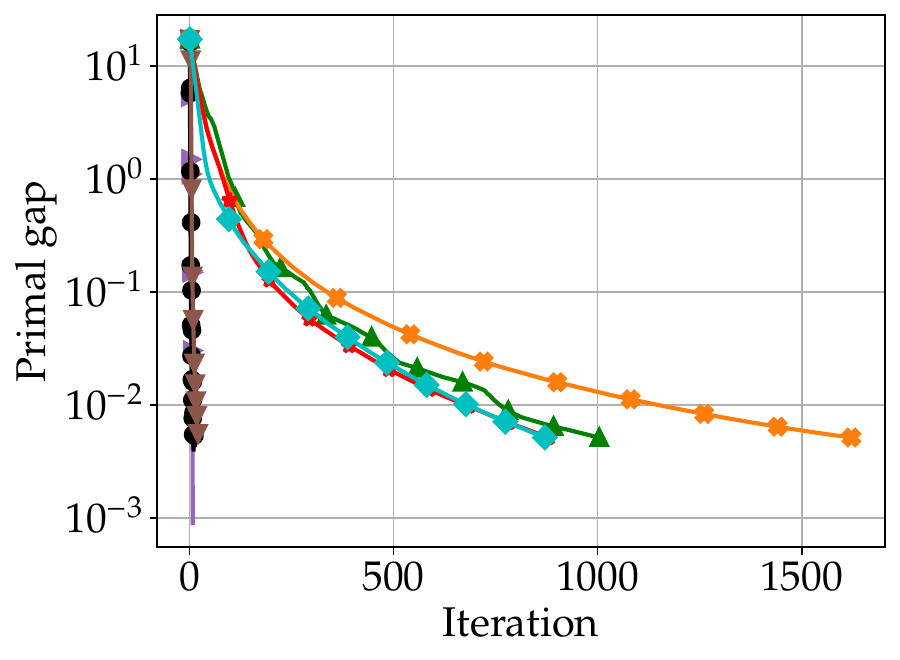} }\label{fig:GLassoDistanceIt}}%
    %\qquad
    \hspace{\fill}
    \subfloat[Seconds]{{\includegraphics[width=3.9cm]{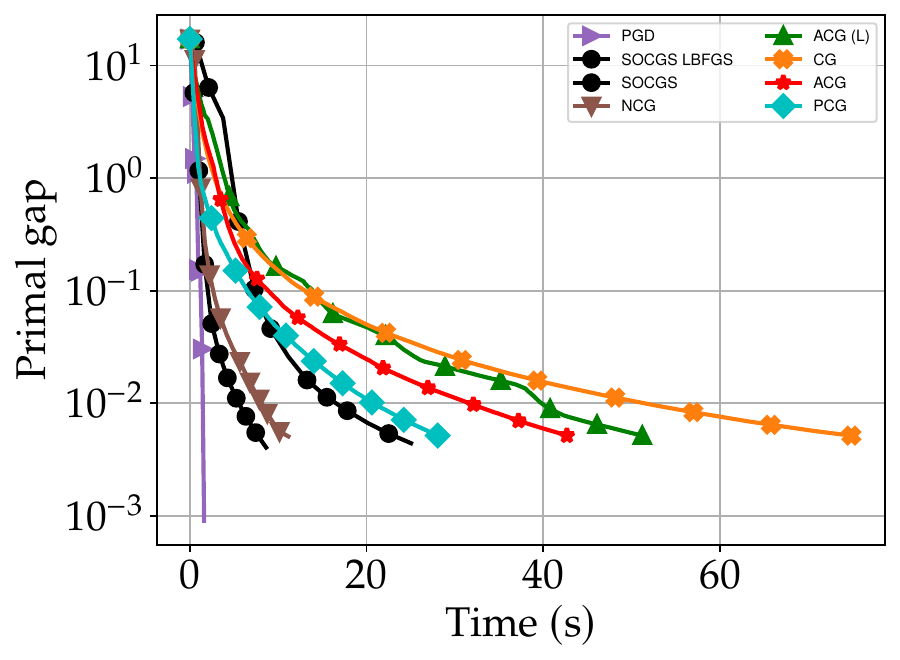} }\label{fig:GLassoDistanceTime}}%
    \hspace*{\fill}
    \caption{Spectrahedron: Comparison in terms of primal gap for $n = 100$ \protect\subref{fig:GLassoPGIt},\protect\subref{fig:GLassoPGTime} and for $n = 50$ \protect\subref{fig:GLassoDistanceIt},\protect\subref{fig:GLassoDistanceTime}.}%
    \label{fig:GLasso}%
\end{figure*}

\section*{Conclusion}
\label{sec:conclusion}

This paper focuses on the minimization of a smooth and strongly convex function over a polytope in the setting where efficient access to the feasible region is limited to a linear optimization oracle and first-order information about the objective function is expensive to compute. We also assume inexact second-order information subject to an accuracy requirement.

Given these challenges, we present the Second-order Conditional Gradient Sliding (SOCGS) algorithm, which at each iteration computes an Inexact Projected Variable-Metric (PVM) step with unit step size (using the Away-step Conditional Gradient (ACG) algorithm and an accuracy criterion that depends on a lower bound on the primal gap), and an independent ACG step with line search, and chooses the step that provides the greatest primal progress. As the algorithm relies on a linear minimization oracle, as opposed to a projection oracle, it is projection-free. The algorithm can be seen as the second-order analog of the Conditional Gradient Sliding algorithm \cite{lan2016conditional}, which uses Conditional Gradient steps to compute inexact Euclidean projections in Nesterov's Accelerated Gradient Descent algorithm. After a finite number (independent of the target accuracy $\varepsilon$) of linearly convergent iterations, the convergence rate of the SOCGS algorithm is quadratic in primal gap. Once inside this phase the SOCGS algorithm reaches an $\varepsilon$-optimal solution after $\mathcal{O}\left(\log(\log 1/\varepsilon)\right)$ Hessian and first-order oracle calls and $\mathcal{O}(\log (1/\varepsilon) \log(\log1/\varepsilon))$ linear minimization oracle calls.

The Newton Conditional Gradient (NCG) (or Newton Frank-Wolfe) algorithm \cite{liu2020newton} uses an approach that is similar in spirit to the one used in the SOCGS algorithm, however with a very different analysis and set of assumptions. The aforementioned algorithm minimizes a self-concordant function over a convex set by performing Inexact Newton steps using a Conditional Gradient algorithm to solve the constrained quadratic subproblems. This algorithm requires exact Hessian information, and after a finite number of iterations (independent of the target accuracy $\varepsilon$), the convergence rate of the NCG algorithm is linear in primal gap. Once inside this phase a $\varepsilon$-optimal solution is reached after $\mathcal{O}\left(\log 1/\varepsilon\right)$ exact Hessian and first-order oracle calls and $\mathcal{O}( 1/\varepsilon^{\nu})$ linear minimization oracle calls, where $\nu$ is a constant greater than one.

The computational results show that the SOCGS algorithm outperforms other first-order projection-free algorithms and the NCG algorithm in applications where first-order information is costly to compute. The improved performance with respect to other first-order projection-free algorithms is due to the substantial progress per iteration provided by the Inexact PVM steps, which makes up for their higher computational cost, resulting in faster convergence with respect to time. The better performance of the SOCGS algorithm with respect to the NCG algorithm is due to the better global convergence of the SOCGS algorithm, and the use of the Away-step Conditional Gradient algorithm as a subproblem solver in the SOCGS algorithm, as opposed to the vanilla Conditional Gradient algorithm used by the NCG algorithm.

\section*{Acknowledgments}
\label{sec:acknowledgments}
Research reported in this paper was partially supported by NSF CAREER
Award CMMI-1452463. We would like to thank Gábor Braun for the helpful discussions, and the anonymous reviewers for their suggestions and comments.

%\clearpage

%\balance
\bibliography{references}

\bibliographystyle{icml2021.bst}

\newpage

\appendix

{\centering{\LARGE\bfseries Second-order Conditional Gradient Sliding}

  \vspace{1em}
  \centering{{\LARGE\bfseries Supplementary material}}

}
\vspace{2em}

\paragraph{Outline.} The appendix of the paper is organized as follows:
\begin{itemize}[leftmargin=*]
  \item Section~\ref{Appx:Notation} presents the notation and definitions used throughout the appendix, as well as useful material pertaining to the Hessian approximation.
  \item Section~\ref{Section:Appx:ACG} contains background information about the Conditional Gradients algorithm, pseudocode for the vanilla Conditional Gradients algorithm and the Away-step Conditional Gradients algorithm and theoretical information about the convergence of the Away-step Conditional Gradients algorithm.
  \item Section~\ref{appx:VarMetricAlgo} presents information about the vanilla Projected Variable-Metric algorithm, its global linear convergence with exact line search or a bounded stepsize, and its quadratic local convergence in distance to the optimum with unit step size. 
  \item Section~\ref{appx:SOCG} contains the proof of global linear and local quadratic convergence in primal gap of the Second-order Conditional Gradient Sliding algorithm, as well as an oracle complexity analysis.
  \item Section~\ref{Appx:Computations} presents a detailed description of the numerical experiments performed.
\end{itemize}

\section{Notation and Preliminaries} \label{Appx:Notation}

We denote the norm of a vector $\mathbf{v}$ as $\norm{\mathbf{v}} = \sqrt{\innp{\mathbf{v}, \mathbf{v}}}$, and the norm of a matrix $A$ as $\norm{A} = \max_{\mathbf{v} \neq 0} \norm{A \mathbf{v}} /\norm{\mathbf{v}}$. Let $\mathcal{S}^n_{++}$ denote the set of symmetric positive definite matrices in $\rr^{n\times n}$ and let $\norm{\cdot}_H$ denote the matrix norm defined by $H \in \mathcal{S}^n_{++}$, that is, for a given vector $\mathbf{v}$ the norm defined by $H$ is $\norm{\mathbf{v}}_H = \sqrt{\innp{\mathbf{v}, H\mathbf{v}}}$.  We use $\mathbf{v}_{\min}\left(H\right)$ and $\mathbf{v}_{\max}\left(H\right)$ to refer to the eigenvectors of unit norm associated with the minimum and maximum eigenvalues, denoted by $\lambda_{\min}\left( H \right)$ and $\lambda_{\max}\left( H \right)$ respectively, of the matrix $H\in \mathcal{S}^n_{++}$. Similarly, we use $\lambda_{i}\left( H \right)$ with $i \in [1,n]$ to refer to the $i$-th largest eigenvalue of the matrix $H\in \mathcal{S}^n_{++}$. Let $\sigma_{\min}(H)$ and $\sigma_{\max}(H)$ denote the minimum and maximum singular values of the matrix $H$. We denote the open ball of radius $r>0$ centered at $\vx$ as $\mathcal{B}(\vx, r)$. Let $\interior\left( \cx\right)$ and $\relinterior\left( \cx\right)$ represent the interior and the relative interior of the set $\cx$, respectively. Given a function $f (\vx): \rr^n \rightarrow \rr$, we say that the function is:
 \begin{definition}[$\mu$-strongly convex function]
 The function is $\mu$-strongly convex over $\mathcal{X}$ if there exists a $\mu >0$ such that:
 \begin{align*}
     f\left(\vx\right) - f\left(\vy\right) \geq \innp{\nabla f\left(\vy\right), \vx - \vy} + \frac{\mu}{2} \norm{\vx - \vy}^2,
 \end{align*}
 for all $\vx, \vy \in \mathcal{X}$.
 \end{definition}
  \begin{definition}[$L$-smooth function]
 The function is $L$-smooth over $\mathcal{X}$ if there exists a $L >0$ such that:
 \begin{align*}
     f\left(\vx\right) - f\left(\vy\right) \leq \innp{\nabla f\left(\vy\right), \vx - \vy} + \frac{L}{2} \norm{\vx - \vy}^2,
 \end{align*}
 for all $\vx, \vy \in \mathcal{X}$.
 \end{definition}
 
 A simple schematic representation of the bounds provided by convexity, $\mu$-strong convexity and $L$-smoothness can be seen in Figure~\ref{fig:definitions}.
 
 \begin{figure}[th!]
\centering
  \includegraphics[width=0.9\linewidth]{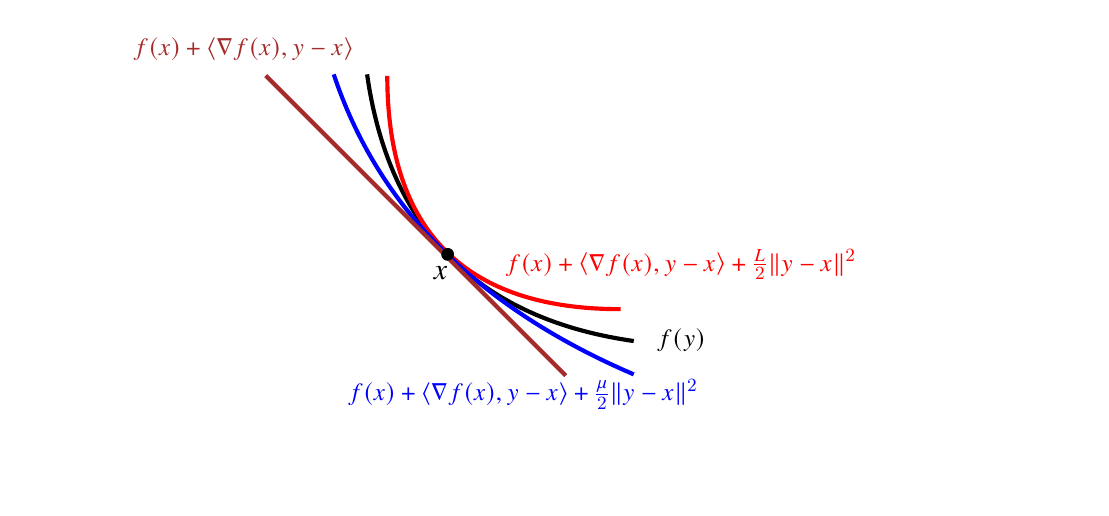}
 \caption{The red line depict the quadratic upper bound from $L$-smoothness, the blue line depicts the quadratic lower bound provided by $\mu$-strong convexity, the green line depicts the linear lower bound provided by convexity.}
  \label{fig:definitions}
\end{figure}

  \begin{definition}[$L_2$-Lipschitz continuous Hessian]
 The function has a $L_2$-Lipschitz continuous Hessian over $\mathcal{X}$ if there exists a $L_2 > 0$  such that:
 \begin{align*}
     \norm{\nabla^2 f\left( \vx\right) - \nabla^2 f\left( \vy\right)} \leq L_2\norm{\vx - \vy}
 \end{align*}
 for all $\vx, \vy \in \mathcal{X}$.
 \end{definition}
 
    \begin{definition}[Normal cone of $\cx$] \label{Def:Appx:NormalCone}
  We define the normal cone of the set $\cx$ at point $\vx$, denoted by $N_{\cx} \left( \vx\right)$, as:
 \begin{align*}
     N_{\cx}\left( \vx \right) = \begin{cases}
\left\{ \vd \in \rr^n \mid \innp{\vd, \vy - \vx} \leq 0, \forall \vy \in \cx \right\} &\text{if $\vx \in \cx$}\\
\emptyset &\text{if $\vx \notin \cx$}
\end{cases}
 \end{align*}
   \end{definition}

\subsection{Hessian Approximation Accuracy} \label{Section:Appx:HessianApprox}

\begin{lemma}
  \label{lemma:fracProg}
  Let $P, Q \in \mathcal{S}^n_{++}$. The solution to the fractional quadratic program $\max_{\vu \in \rr^n} \norm{\vu}_Q^2/\norm{\vu}_P^2$ is given by the the largest eigenvalue of the symmetric positive definite matrix $P^{-1/2} Q P^{-1/2}$, that is, $\lambda_{\max} \left( P^{-1/2} Q P^{-1/2} \right)$, which in turn is equal to $\lambda_{\max} \left( P^{-1} Q  \right)$.  Moreover, the solution to the fractional quadratic program $\min_{\vu \in \rr^n} \norm{\vu}_Q^2/\norm{\vu}_P^2$ is given by the smallest eigenvalue of the symmetric positive matrix $P^{-1/2} Q P^{-1/2}$, that is $\lambda_{\min} \left( P^{-1/2} Q P^{-1/2}\right)$, which in turn is equal to $\lambda_{\min} \left( P^{-1} Q \right)$. 
\begin{proof}
Writing out the expression for the quadratic program we have that:
  \begin{align*}
      \max_{\vu \in \rr^n} \frac{\norm{\vu}_Q^2}{\norm{\vu}_P^2}  &= \max_{\vu \in \rr^n} \frac{\vu^T Q\vu}{\vu^TP\vu} \\
      & = \max_{\vu \in \rr^n} \frac{\vu^T Q \vu}{(P^{1/2}\vu)^T P^{1/2} \vu} \\
      & = \max_{\vw \in \rr^n} \frac{(P^{-1/2} \vw)^T Q P^{-1/2} \vw}{\norm{\vw}^2} \\
      & = \max_{\vw \in \rr^n} \frac{ \vw^T P^{-1/2} Q P^{-1/2} \vw}{\norm{\vw}^2} \\
      & = \lambda_{\max} \left( P^{-1/2} Q P^{-1/2} \right).
  \end{align*}
Moreover, note that as $P$ and $Q$ are positive definite. $\lambda_{\max} \left( P^{-1/2} Q P^{-1/2} \right) = \lambda_{\max} \left( P^{-1} Q  \right)$. The second claim follows using a very similar reasoning.
\end{proof}
\end{lemma}

\begin{lemma}
  \label{lemma:scalingCondition}
  Given two matrices $P, Q \in \mathcal{S}^n_{++}$, then for all $\mathbf{v} \in \rr^n$:
  \begin{align}
  \frac{1}{\eta} \norm{\mathbf{v}}^2_P  \leq \norm{\mathbf{v}}^2_Q \leq \eta \norm{\mathbf{v}}^2_P, \label{Eq:scalingEq}
  \end{align}
  with $\eta = \max \left\{\lambda_{\max} \left( P^{-1} Q\right), \lambda_{\max} \left( Q^{-1} P\right) \right\} \geq 1$.
  \begin{proof}
  Let $\lambda_{i}\left(P\right)$ denote the $i$-th eigenvalue of matrix $P$. Note that as $P$ and $Q$ are positive definite $P^{-1}$ and $Q^{-1}$ are well-defined, furthermore $P^{-1}Q$ and $Q^{-1} P$ are also positive definite, as the eigenvalues of $P^{-1} Q$ are the same as those of the symmetric positive definite matrix $P^{-1/2}QP^{-1/2}$, and the eigenvalues of $Q^{-1} P$ are the same as those of the symmetric positive definite matrix $Q^{-1/2}PQ^{-1/2}$. In order to show that $\eta \geq 1$ note that if $\lambda_{\max} \left( Q^{-1/2} P Q^{-1/2}\right) = \lambda_{\max} \left( Q^{-1} P \right) \leq 1$, then $\lambda_{i}\left(Q^{-1/2} P Q^{-1/2}\right) = \lambda_{i}\left(  Q^{-1} P\right) \in (0,1]$ for all $i \in \llbracket 1, n\rrbracket$, and therefore the eigenvalues of its inverse satisfy $\lambda_{i}\left((Q^{-1} P)^{-1}\right) = \lambda_{i}\left(P^{-1}Q\right) \geq 1$ for all $i \in \llbracket 1, n\rrbracket$. Conversely, if $\lambda_{\max} \left( P^{-1}Q\right) \leq 1$, the same reasoning applies, and $\lambda_{i} \left( Q^{-1}P\right) \geq 1$ for all $i\in \llbracket 1, n\rrbracket$. Note that the definition of $\eta$ together with Lemma~\ref{lemma:fracProg} implies that $\frac{1}{\eta} = \min \left\{\lambda_{\min} \left( P^{-1}Q\right), \lambda_{min} \left( Q^{-1}P\right) \right\}\leq  \lambda_{\min} \left( P^{-1}Q \right) = \lambda_{\max} \left( Q^{-1}P \right)$. Focusing on the first inequality on Equation~\eqref{Eq:scalingEq} and plugging in the value of $\eta$ leads to:
    \begin{align*}
      \frac{1}{\eta} \norm{\mathbf{v}}^2_P & \leq  \norm{\mathbf{v}}^2_P  \lambda_{\min} \left( P^{-1}Q \right)\\
      & =  \norm{\mathbf{v}}^2_P \min\limits_{\vu \in \rr^n} \frac{\norm{\vu}^2_Q}{\norm{\vu}^2_P} \\
      & \leq  \norm{\mathbf{v}}^2_P \frac{\norm{\mathbf{v}}^2_Q}{\norm{\mathbf{v}}^2_P}\\
      & = \norm{\mathbf{v}}^2_Q.
  \end{align*}
  Focusing on the second inequality of Equation~\eqref{Eq:scalingEq} and noting that $\eta = \max \left\{\lambda_{\max} \left( P^{-1}Q\right), \lambda_{\max} \left( Q^{-1}P\right) \right\} \geq  \lambda_{\max} \left( P^{-1}Q\right)$ we have that:
  \begin{align*}
      \eta \norm{\mathbf{v}}^2_P & \geq \norm{\mathbf{v}}^2_P \lambda_{\max} \left( P^{-1}Q\right)\\
       & =  \norm{\mathbf{v}}^2_P \max\limits_{\vu \in \rr^n} \frac{\norm{\vu}^2_Q}{\norm{\vu}^2_P}\\
      & \geq  \norm{\mathbf{v}}^2_P\frac{\norm{\mathbf{v}}^2_Q}{\norm{\mathbf{v}}^2_P} \\
      & = \norm{\mathbf{v}}^2_Q.
  \end{align*}
  Which completes the proof.
  \end{proof}
\end{lemma}

\begin{remark}
  \label{lemma:scalingConditionInv}
 Given two matrices $P, Q \in \mathcal{S}^n_{++}$, then for all $\mathbf{v} \in \rr^n$:
  \begin{align}
  \frac{1}{\eta} \norm{\mathbf{v}}^2_{P^{-1}}  \leq \norm{\mathbf{v}}^2_{Q^{-1}} \leq \eta \norm{\mathbf{v}}^2_{P^{-1}}, \label{Eq:scalingEqInv}
  \end{align}
  with $\eta = \max \left\{\lambda_{\max} \left( P^{-1}Q\right), \lambda_{\max} \left( Q^{-1}P\right) \right\} \geq 1$.
  \begin{proof}
  As $P^{-1}, Q^{-1} \in \mathcal{S}^n_{++}$, we can apply Lemma~\ref{lemma:scalingCondition}. The proof then follows from the fact that $\lambda_{\max} \left( PQ^{-1}\right) = \lambda_{\max} \left( Q^{-1}P\right)$ and $ \lambda_{\max} \left( QP^{-1}\right) = \lambda_{\max} \left( P^{-1}Q\right) $ as $P$ and $Q$ are symmetric positive definite.
  \end{proof}
\end{remark}

If we define the ellipsoid $\mathcal{E}_{P} = \left\{\mathbf{v} \in \rr^n \mid \mathbf{v}^TP\mathbf{v} \leq 1 \right\}$ for $P\in \mathcal{S}^n_{++}$, we can interpret the value of $\eta$ as being the smallest value that ensures that $\mathcal{E}_{
P/\eta} \subseteq \mathcal{E}_{Q} \subseteq \mathcal{E}_{\eta P}$ for $Q\in \mathcal{S}^n_{++}$ (see Figure~\ref{fig:scalingHessian}).

\begin{figure}[h]
    \centering
    \includegraphics[width=0.6\textwidth]{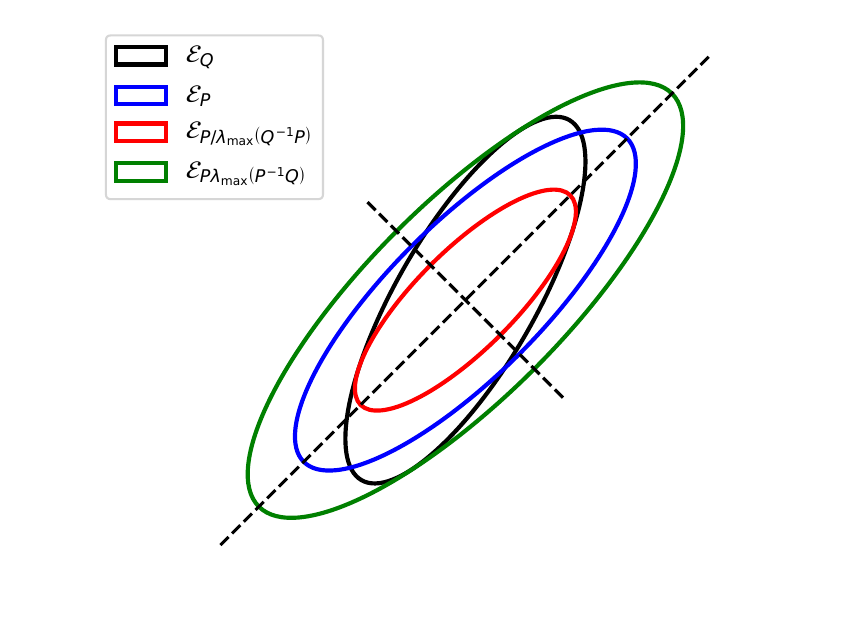}
    \caption{Given $P, Q \in \mathcal{S}^n_{++}$, we can always find an $\eta$ such that $\mathcal{E}_{
P/\eta} \subseteq \mathcal{E}_{Q} \subseteq \mathcal{E}_{\eta P}$.}
    \label{fig:scalingHessian}
\end{figure}

The following corollary will allow us to bound the maximum and minimum eigenvalue of the approximation $H_k$ in terms of the maximum and minimum eigenvalue of $\nabla^2 f(\vx_k)$ and $\eta_k$ for all $k\geq 0$, which will be useful in the proofs to follow.

\begin{corollary}
  \label{corollary:Eigenvalues}
  Given two matrices $P, Q \in \mathcal{S}^n_{++}$, we have that:
  \begin{align*}
      \frac{\lambda_{\min}\left(P \right)}{\eta} \leq \lambda_{\min}\left(Q \right) \leq  \eta \lambda_{\min}\left(P \right) \\
       \frac{\lambda_{\max}\left(P \right)}{\eta}  \leq \lambda_{\max}\left(Q \right) \leq \eta\lambda_{\max}\left(P \right),
  \end{align*}
  where $\eta = \max \left\{\lambda_{\max} \left( P^{-1}Q\right), \lambda_{\max} \left( Q^{-1}P\right) \right\} \geq 1$. This allows us to conclude that $\frac{\lambda_{\min}\left(P \right)}{\eta} I^n \preceq Q \preceq \eta\lambda_{\max}\left(P \right) I^n$.
  \begin{proof}
  Let $\mathbf{v}_{\min}\left(Q \right)$ and $\mathbf{v}_{\max}\left(Q \right)$ denote the eigenvectors of unit length associated with the minimum and maximum eigenvalue of $Q$, denoted by $\lambda_{\min}\left(Q \right)$ and $\lambda_{\max}\left(Q \right)$ respectively. As $P, Q \in \mathcal{S}^n_{++}$ from Lemma~\ref{lemma:scalingCondition} we have that:
\begin{align*}
  \lambda_{\min}\left( Q\right)  & = \norm{\mathbf{v}_{\min}\left(Q \right)}^2_Q \geq \frac{1}{\eta} \norm{\mathbf{v}_{\min}\left(Q \right)}^2_P \geq \frac{1}{\eta} \norm{\mathbf{v}_{\min}\left(P \right)}^2_P = \frac{\lambda_{\min}\left( P\right)}{\eta}.
  \end{align*}
  On the other hand, using similar arguments we have:
    \begin{align*}
  \lambda_{\min}\left( Q\right)  & = \norm{\mathbf{v}_{\min}\left(Q \right)}^2_Q \leq \norm{\mathbf{v}_{\min}\left(P \right)}^2_Q \leq \eta \norm{\mathbf{v}_{\min}\left(P \right)}^2_P = \eta \lambda_{\min}\left( P\right).
  \end{align*}  
  Moving on to the bound for $\lambda_{\max}\left(Q \right)$ we have:
  \begin{align*}
  \lambda_{\max}\left( P\right)  & = \norm{\mathbf{v}_{\max}\left(P \right)}^2_P \geq \norm{\mathbf{v}_{\max}\left(Q \right)}^2_P  \geq\frac{1}{\eta} \norm{\mathbf{v}_{\max}\left(Q \right)}^2_Q  = \frac{\lambda_{\max}\left( Q\right)}{\eta}.
 \end{align*}
  Similarly, we have that:
  \begin{align*}
  \lambda_{\max}\left( P\right)  & = \norm{\mathbf{v}_{\max}\left(P \right)}^2_P \leq \eta \norm{\mathbf{v}_{\max}\left(P \right)}^2_Q  \leq \eta \norm{\mathbf{v}_{\max}\left(Q \right)}^2_Q  = \eta \lambda_{\max}\left( Q\right).
 \end{align*}  
  Combining these bounds completes the proof.
  \end{proof}
\end{corollary}

Particularizing Corollary~\ref{corollary:Eigenvalues} with $Q = H_k$ and $P = \nabla^2 f(\vx_k)$ allows us to conclude that $\mu/\eta_k I^n \preceq H_k \preceq \eta_k L I^n$, and so the quadratic approximation $\hat{f}_k(\vx)$ in  Equation~\eqref{def:approx} will be $\mu/\eta_k$-strongly convex and $\eta_k L$-smooth.

\newpage

\section{The Conditional Gradients algorithm} \label{Section:Appx:ACG}

  We define the linear approximation of the function $f(\vx)$ around the point $\vx_k$ as:
\begin{align}
    \hat{l}_{k}(\vx) \defeq f(\vx_k) + \innp{\nabla f(\vx_k), \vx - \vx_k}. \label{def:Appx:LinearApprox}
\end{align}
At each iteration the vanilla \emph{Conditional Gradients} (CG) algorithm \cite{levitin1966constrained, frank1956algorithm, jaggi2013revisiting} (Algorithm~\ref{algo:ConditionalGradients}) takes steps defined as $\vx_{k+1} = \vx_k + \gamma_k (\argmin_{\vx \in \cx} \hat{l}_{k}(\vx) - \vx_k )$ with $\gamma_k \in (0,1]$. As the iterates are formed as convex combinations of points in $\cx$ there is no need for projections onto $\cx$, making the algorithm \emph{projection-free}.

\begin{algorithm}
\SetKwInOut{Input}{Input}\SetKwInOut{Output}{Output}
%\SetKwComment{Comment}{$\triangleright$\ }{}
\SetKwComment{Comment}{//}{}
\Input{Point $\vx_0 \in \cx$, step sizes $\{\gamma_0, \cdots ,\gamma_k \}$}
\Output{Point $\vx_{K } \in \cx$}
\hrulealg
\For{$k = 0$ \textbf{to} $K - 1$}{
  $\mathbf{v}_k\leftarrow  \argmin\limits_{\vx\in\mathcal{X}}\hat{l}_{k}(\vx) =  \argmin\limits_{\vx\in\mathcal{X}} \left( f \left( \vx_k \right) + \innp{\nabla f(\vx_k),\vx - \vx_k}\right) $\label{fw_extract} \;
  $\vx_{k+1}\leftarrow \vx_k+\gamma_k(\mathbf{v}_k-\vx_k)$\label{fw_new}  \;
} 
\caption{Conditional Gradients algorithm}\label{algo:ConditionalGradients}
\end{algorithm}

A useful quantity that can readily be computed in all CG steps is $\langle\nabla f(\vx_k),\vx_k-\mathbf{v}_k\rangle$, known as the \emph{Frank-Wolfe gap}, which provides an upper bound on the primal gap. If $\vx^*\in\argmin\limits_{\vx\in\mathcal{X}}f(\vx)$, then:
\begin{align*}
    \innp{\nabla f(\vx_k), \vx_k - \mathbf{v}_k} = \max_{\mathbf{v} \in \mathcal{X}}\innp{\nabla f(\vx_k), \vx_k - \mathbf{v}} \geq \innp{\nabla f(\vx_k), \vx_k - \vx^*}  \geq f(\vx_k) - f(\vx^*)
\end{align*}
where the last inequality follows from the convexity of $f(\vx)$. This quantity is often used as a stopping criterion when running the CG algorithm. The CG algorithm has seen a renewed interest from the Machine Learning community, as several machine learning problems can be phrased as constrained optimization problems with feasible regions onto which it is hard to project on \cite{joulin2014efficient, futami2019bayesian, garber2018fast}.

\subsection{Global Convergence}

The CG algorithm with exact line search converges linearly in primal gap when applied to Problem~\eqref{eq:OptProblem} when $\vx^* \in \interior\left( \cx \right)$ \citep{guelat1986some}. However, when $\vx^* \in \cx \setminus \interior\left( \cx \right)$ the algorithm suffers from a \emph{zig-zagging} phenomenon - as the iterates get closer to $\vx^*$ the directions provided by the algorithm starts to become close to perpendicular to the gradient (Figure~\ref{fig:zigzag}). This is remedied by using \emph{Away-steps} (Algorithm~\ref{algo:Appx:ACGSteps}), which result in the \emph{Away-step Conditional Gradient} (ACG) algorithm (Algorithm~\ref{Algo:Appx:ACGAlg}, Figure~\ref{fig:ACG}) \cite{wolfe1970convergence}, which converges linearly in primal gap regardless of the location of $\vx^*$ when using exact line search \citep{lacoste2015global} or a step size strategy dependent on $L$ \cite{pedregosa2020linearly}.

\begin{figure*}[ht!]
    \centering
    \hspace{\fill}
    \subfloat[CG (Algorithm~\ref{algo:ConditionalGradients}).]{{\includegraphics[width= .4\textwidth]{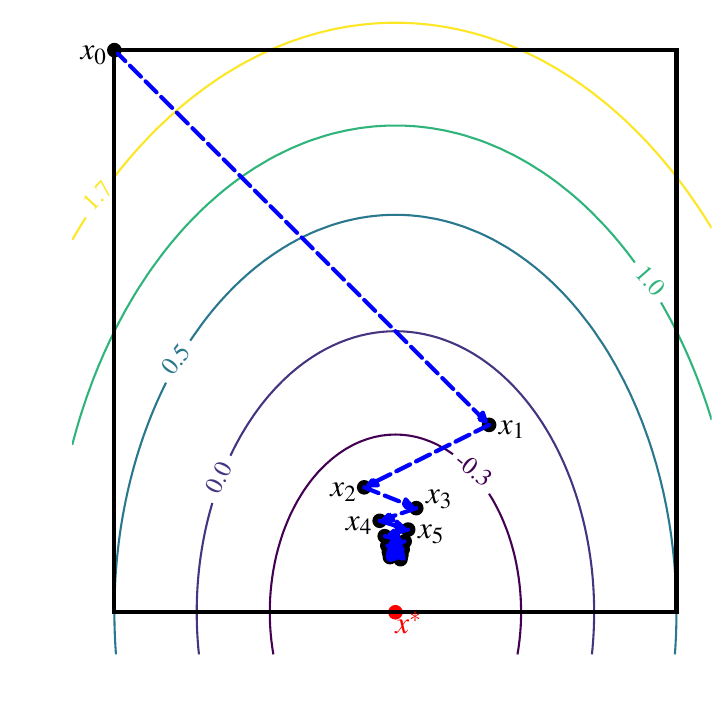} }\label{fig:zigzag}}%
    %\qquad
    \hspace{\fill}
    \subfloat[ACG (Algorithm~\ref{Algo:Appx:ACGAlg}).]{{\includegraphics[width = .4\textwidth]{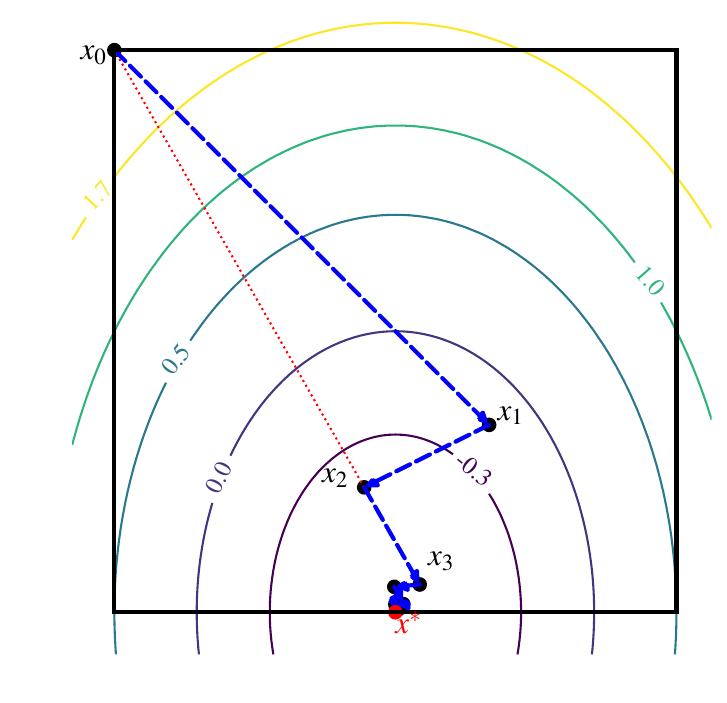} }\label{fig:ACG}}%
    \hspace*{\fill}
    \caption{Qualitative performance comparison of the CG and the ACG algorithm.}%
    \label{fig:comparison}%
\end{figure*}

\begin{algorithm}
\SetKwInOut{Input}{Input}\SetKwInOut{Output}{Output}
%\SetKwComment{Comment}{$\triangleright$\ }{}
\SetKwComment{Comment}{//}{}
\Input{Point $\vx_0 \in \cx$}
\Output{Point $\vx_{K } \in \cx$}
\hrulealg
$\vx_0 \leftarrow \argmin_{\vx \in  \mathcal{X}} \innp{\nabla f\left(\vx \right), \vx}$, $\cs_0 \leftarrow \{ \vx_0 \}$, $\vlambda_0(\vx_0) \leftarrow  1 $\;
\For{$k = 0$ \textbf{to} $K - 1$}{
$\vx_{k+1}, \mathcal{S}_{k+1}, \vlambda_{k+1} \leftarrow $ACG$(f(\vx), \vx_{k}, \mathcal{S}_{k}, \vlambda_{k})$\;
} 
\caption{Away-step Conditional Gradients (ACG) algorithm with exact line search.}\label{Algo:Appx:ACGAlg}
\end{algorithm}

The ACG algorithm maintains what is called an \emph{active set} $\cs_k \subseteq \vertex\left(\mathcal{X}\right)$ which represents the potentially non-unique set of vertices of $\mathcal{X}$ such that $\vx_k \in \co\left(\mathcal{S}_k\right)$. Associated with this active set $\cs_k$ we have a set of barycentric coordinates $\vlambda_k$ such that if we denote by $\vlambda_k(\vu)\in [0,1]$ the element of $\vlambda_k$ associated with $\vu \in \cs_k$ we have that $\vx_{k} = \sum_{\vu \in \cs_k} \vlambda_k(\vu) \vu$, with $\sum_{\vu \in \cs_k} \vlambda_k(\vu) = 1$ and $\vlambda_k(\vu) \geq 0$ for all $\vu \in \cs_k$. 

\begin{algorithm}[t!]
\SetKwInOut{Input}{Input}\SetKwInOut{Output}{Output}
\Input{Function $f:\cx \rightarrow \rr$, point $\vx \in \cx$, active set $\cs$ and barycentric coordinates $\vlambda$.}
\Output{Point $\vx' \in \cx$, active set $\cs'$ and barycentric coordinates $\vlambda'$.}
\hrulealg
$\mathbf{v} \leftarrow \argmin_{\mathbf{v} \in  \mathcal{X}} \innp{\nabla f\left(\vx \right), \mathbf{v}}$\;
$\va \leftarrow \argmax_{\mathbf{v} \in  \mathcal{S}} \innp{\nabla f\left(\vx \right), \mathbf{v}}$\;
\uIf{$\innp{\nabla f(\vx),  \vx - \mathbf{v}} \geq \innp{\nabla f(\vx),\va - \vx}$ \label{algLine:ACG:Appx:ChooseStep} }{
$\vd \leftarrow \vx - \mathbf{v}$, $\gamma_{\max} \leftarrow 1$\;}
\Else{ 
$\vd \leftarrow \va - \vx$, $\gamma_{\max} \leftarrow \vlambda(\va)/\left( 1 - \vlambda(\va)\right)$\;}
$\gamma \leftarrow \argmin_{\gamma\in [0,\gamma_{\max}]} f\left(\vx + \gamma \vd \right)$ \label{algLine:ACG:Appx:LineSearch}\;
$\vx' \leftarrow \vx + \gamma \vd$\;
\uIf{$\innp{\nabla f(\vx),  \vx - \mathbf{v}} \geq \innp{\nabla f(\vx),\va - \vx}$}{
\uIf{$\gamma = 1$}{
$\cs' \leftarrow \left\{ \mathbf{v} \right\}$\;}
\Else{
$\cs' \leftarrow \cs \cup \left\{ \mathbf{v} \right\}$\;
}
$\vlambda'(\vu) \leftarrow \left( 1 - \gamma \right) \vlambda(\vu)$ if $\vu \in \cs \setminus \mathbf{v}$\;
$\vlambda'(\mathbf{v}) \leftarrow \left( 1 - \gamma \right) \vlambda(\mathbf{v}) + \gamma$\;
}
\Else{
\uIf{$\gamma = \gamma_{\max}$}{
$\cs' \leftarrow \cs \setminus \{ \va\}$\;}
\Else{
$\cs' \leftarrow \cs$\;
}
$\vlambda'(\vu) \leftarrow \left( 1 + \gamma \right) \vlambda(\vu)$ if $\vu \in \cs \setminus \va$\;
$\vlambda'(\va) \leftarrow \left( 1 + \gamma \right) \vlambda(\va) - \gamma$\;
}
\caption{Away-step Conditional Gradients step ACG$\left( f, \vx , \cs, \vlambda \right)$} \label{algo:Appx:ACGSteps}
\end{algorithm}

In general one of the easiest ways to maintain the active set is to build a list of previously used vertices and a list of associated barycentric coordinates. If the Frank-Wolfe step adds a new vertex $\mathbf{v}$ that is not already in $\cs$ it is added to the list of vertices and its associated barycentric coordinate is added to the list of barycentric coordinates. If the vertex $\mathbf{v}$ is already contained in the list that maintains $\cs$, its existing barycentric coordinate is updated in the appropiate list. Note that the barycentric coordinates of the points $\cs\setminus\{\mathbf{v}\}$ are also updated at each iteration. The away-steps in Algorithm~\ref{algo:Appx:ACGSteps} cannot add new vertices, only remove them from the active set. This type of step also requires updating the barycentric coordinates of the points $\cs \setminus \{\va\}$. For both Frank-Wolfe and away-steps a vertex is removed from the list of vertices and the associated barycentric coordinate removed from the list of coordinates if the value of the barycentric coordinate is zero.

The first proof of asymptotic linear convergence of the ACG algorithm relied on the \emph{strict complementarity} of the problem in Equation~\eqref{eq:OptProblem} (shown in Assumption~\ref{assumption:strictComplementarity}), which we will also use in the convergence proof of the SOCGS algorithm.

\setcounter{assumption}{0}
\begin{assumption}[Strict Complementarity]\label{assumption:Appx:strictComplementarity}
We have that $\innp{\nabla f\left( \vx^*\right), \vx - \vx^*} = 0$ if and only if $\vx \in \mathcal{F} \left( \vx^*\right)$.
\end{assumption}

\begin{remark}
 Assumption~\ref{assumption:Appx:strictComplementarity} automatically holds if $\vx^* \in \interior \left(\cx \right)$, that is, if $\vx^*$ is in the strict interior of $\cx$. In this case the polytope is fully-dimensional and itself the optimal face, so no off-optimal-face vertices exist. 
\end{remark}

If Assumption~\ref{assumption:Appx:strictComplementarity} is satisfied the iterates of the ACG algorithm reach $\mathcal{F} \left( \vx^*\right)$ in a finite number of steps, remaining in $\mathcal{F} \left( \vx^*\right)$ for all subsequent iterations \cite{guelat1986some}. When inside $\mathcal{F} \left( \vx^*\right)$, the iterates of the ACG algorithm contract the primal gap linearly. This analysis was later significantly extended to provide an explicit global linear convergence rate in primal gap (Theorem~\ref{theorem:ConvergenceACG}), by making use of the \emph{pyramidal width} of the polytope $\cx$ \cite{lacoste2015global}. With the pyramidal width one can derive a primal progress guarantee for all steps taken by the ACG algorithm except "bad" away-steps that reduce the cardinality of the active set $\mathcal{S}_k$, that is when $\innp{\nabla f(\vx_k),  \vx_k - \mathbf{v}} <\innp{\nabla f(\vx_k),\va - \vx_k}$ and the step size satisfies $\gamma_k = \gamma_{\max}$ in Algorithm~\ref{algo:Appx:ACGSteps}. This cannot happen more than $\lfloor K/2 \rfloor$ times when running the ACG algorithm for $K$ iterations (as the algorithm cannot drop more vertices with away-steps than it has picked up with Frank-Wolfe steps). This is an important consideration to keep in mind, as it means that the ACG primal gap contraction does not hold on a per-iteration basis. 

\begin{theorem}[Primal gap convergence of the ACG algorithm (Algorithm~\ref{Algo:Appx:ACGAlg})]~\citep[Theorem 1]{lacoste2015global} \label{theorem:Appx:ConvergenceACG}
  Given an $L$-smooth and $\mu$-strongly convex function $f(\vx)$, a polytope $\mathcal{X}$ and an initial point $\vx_0 \in \cx$, the ACG algorithm satisfies after $K \geq 0$ iterations:
   \begin{align*}
      f(\vx_{K}) - f(\vx^*) \leq \left( 1 - \frac{\mu}{4L} \left( \frac{\delta}{D}\right)^2\right)^{ K/2} \left(f(\vx_{0}) - f(\vx^*) \right),
  \end{align*}
  where $D$ denotes the diameter of the polytope $\cx$ and $\delta$ its pyramidal width.
\end{theorem}

See also \cite{garber2016linearly, diakonikolas2020locally} for work on linearly convergent CG algorithms, and \cite{jaggi2013revisiting, lan2013complexity} for strong lower bounds that limit the linear convergence that can be achieved with algorithms that only access the feasible region through a linear optimization oracle.

\newpage

\section{Projected Variable-Metric algorithms} \label{appx:VarMetricAlgo}

In this section we provide theoretical context for the Projected Variable-Metric (PVM) algorithm (Algorithm~\ref{algo:Appx:varMetricExact}), and we present several well-known results that will be helpful in motivating the SOCGS algorithm.
\begin{algorithm}
\SetKwInOut{Input}{Input}\SetKwInOut{Output}{Output}
%\SetKwComment{Comment}{$\triangleright$\ }{}
\SetKwComment{Comment}{//}{}
\Input{Point $\vx_0 \in \cx$, step sizes $\{\gamma_0, \cdots ,\gamma_k \}$}
\Output{Point $\vx_{K } \in \cx$.}
\For{$k = 0$ \textbf{to} $K - 1$}{
  $\tilde{\vx}^*_{k+1} \leftarrow \argmin\limits_{\vx \in \cx} \hat{f}_{k}(\vx) = \argmin\limits_{\vx \in \cx} \left( f\left(\vx_k \right) + \innp{\nabla f\left( \vx_k\right), \vx - \vx_k} + \frac{1}{2} \norm{\vx - \vx_k}^2_{H_k} \right)$ \label{alg:Appx:BuildQuadApprox}  \;
    $\vx_{k+1} \leftarrow \vx_{k} + \gamma_k \left( \tilde{\vx}_{k+1}^* - \vx_{k} \right)$ \label{alg:Appx:MinimizationVarMetric}  \;
} 
\caption{Projected Variable-Metric (PVM) algorithm}\label{algo:Appx:varMetricExact}
\end{algorithm}

At each iteration the PVM algorithm builds a quadratic approximation of the original function $f(\vx)$, and moves towards the point that minimizes this approximation over $\cx$. Formally, we denote the quadratic approximation of $f(\vx)$ at $\vx_k$ using $H_k \in \mathcal{S}^n_{++}$ as:
 \begin{align}
 \hat{f}_{k}(\vx) \defeq f(\vx_k) + \innp{\nabla f(\vx_k), \vx - \vx_k} + \frac{1}{2} \norm{\vx - \vx_k}_{H_k}^2, \label{Eq:appx:approxfunc}
 \end{align}
 where $H_k$ is an approximation to the Hessian $\nabla^2f\left( \vx_k\right)$. In order to measure how well $H_k$ approximates $\nabla^2f\left( \vx_k\right)$ we note that for any $H_k \in \mathcal{S}^n_{++}$ and all $\vy \in \mathcal{X}$ that:
\begin{align}
    \frac{1}{\eta_k} \norm{\vy - \vx_k}^2_{H_k} \leq \norm{\vy - \vx_k}^2_{\nabla^2 f(\vx_k)} \leq \eta_k \norm{\vy - \vx_k}^2_{H_k}, \label{Eq:Appx:ScalingHessian}
\end{align}
where $\eta_k = \max\{ \lambda_{\max}( H_k^{-1} \nabla^2 f(\vx_k)),  \lambda_{\max}([\nabla^2 f(\vx_k) ]^{-1}H_k )\} \geq 1$ (see Lemma~\ref{lemma:scalingCondition} in Appendix~\ref{Section:Appx:HessianApprox}). We will use the value of $\eta_k$ to measure the accuracy of how well $H_k$ approximates $\nabla^2 f(\vx_k)$. For example, an $\eta_k = 1$ means that $H_k = \nabla^2 f(\vx_k)$. If we were to use $H_k = I^n$ we would have that $\eta_k = \max\left\{L, 1/\mu\right\}$.

Just as the steps taken by the Projected Gradient Descent (PGD) algorithm can be interpreted in terms of Euclidean projection operators, the steps taken by the PVM algorithm in Line~\ref{alg:Appx:BuildQuadApprox} of Algorithm~\ref{algo:Appx:varMetricExact} can be interpreted in terms of scaled projection operators, where the norm of the projection operator is defined by $H_k \in \mathcal{S}^n_{++}$. Let $\Pi^H_{\cx}(\vx): \rr^n \rightarrow \cx$ denote the \emph{scaled projection} of $\vx$ onto $\mathcal{X}$ using the matrix norm $\norm{\cdot}_H$, more concretely $\Pi^H_{\cx}(\vx) \defeq \argmin_{\vy \in \cx} \frac{1}{2} \norm{\vy - \vx}_H^2$. We have that:
 \begin{align}
     \tilde{\vx}_{k+1}^* \defeq \argmin\limits_{\vx \in \mathcal{X}} \hat{f}_k\left( \vx\right) = \Pi^{H_k}_{\cx}\left(\vx_k - H_k^{-1} \nabla f(\vx_k) \right). \label{scaledProjAppx}
 \end{align}
 \begin{remark}[First-order optimality condition for PVM subproblems] \label{Remark:Appx:Optimality}
  The solution to the problem in Line~\ref{alg:Appx:BuildQuadApprox} of Algorithm~\ref{algo:Appx:varMetricExact} (also shown in Equation~\eqref{scaledProjAppx}), that is, $\tilde{\vx}_{k+1}^* = \argmin_{\vx \in \mathcal{X}} \hat{f}_k\left( \vx\right)$ satisfies for all $\vz \in \cx$:
  \begin{align*}
      \innp{\nabla f(\vx_k) + H_k(\tilde{\vx}_{k+1}^* - \vx_k), \vz - \tilde{\vx}_{k+1}^*} \geq 0.
  \end{align*}
\end{remark}
In both the PGD and the PVM algorithm the only point that is invariant under the steps taken by the algorithms is $\vx^*$. That is, in the case of the PGD algorithm we have that $\Pi^{I^n}_{\cx}\left(\vx - \nabla f(\vx) \right) = \vx^*$ if and only if $\vx = \vx^*$. Similarly, in the case of the PVM algorithm we have that $\Pi^{H_k}_{\cx}\left(\vx_k - H_k^{-1} \nabla f(\vx_k) \right) = \vx^*$ if and only if $\vx = \vx^*$ for $H_k \in \mathcal{S}^n_{++}$ (this is shown in Lemma~\ref{lemma:optCond} with the help of Lemma~\ref{lemma:optCondProj}).
\begin{lemma}
  \label{lemma:optCondProj}
  Given a matrix $H \in \mathcal{S}^n_{++}$, for any $\vx \in \cx$ and $\vd \in N_{\cx}\left( \vx\right)$ (where $N_{\cx}\left( \vx\right)$ represents the normal cone of $\cx$ at $\vx$, see Definition~\ref{Def:Appx:NormalCone}) we have that:
  \begin{align}
  \vx   = \Pi^{H}_{\cx}\left(\vx + H^{-1} \vd \right) = \Pi_{\cx}\left(\vx + \vd \right) \label{eq:OptProblem4}.
  \end{align}
  \begin{proof}
  From the definition of the normal cone, given a $\vx \in \cx$ and $\vd \in N_{\cx}\left( \vx\right)$ we know that for all $\vy \in \cx$
  \begin{align}
      0 & \geq \innp{\vd, \vy -\vx}  \\
      & = \innp{\vd, \vy - \left(\vx + H^{-1}\vd \right)} + \innp{\vd, H^{-1}\vd} \\
      & = \innp{H^{-1}\vd, H\left(\vy - \left(\vx + H^{-1}\vd\right)\right)} + \norm{H^{-1}\vd}^2_{H}.\label{Equation:optCondProj}
  \end{align}
  Reordering the previous expression leads to:
    \begin{align*}
      \norm{H^{-1}\vd}^2_{H} & \leq \innp{H^{-1}\vd, H\left(\left(\vx + H^{-1}\vd\right) - \vy\right)} \\
      & \leq \norm{H^{-1}\vd}_{H} \norm{H\left(\left(\vx + H^{-1}\vd\right) - \vy\right)}_{H^{-1}} \\
      & \leq \norm{H^{-1}\vd}_{H} \norm{\left(\vx + H^{-1}\vd\right) - \vy}_{H},
  \end{align*}
  which is true for all $\vy \in \cx$. This leads to:
    \begin{align*}
      \norm{\left(\vx + H^{-1}\vd\right) - \vx}_{H} \leq \norm{\left(\vx + H^{-1}\vd\right) - \vy}_{H},
  \end{align*}
  for all $\vy \in \cx$. This means that the closest point to $\vx + H^{-1}\vd$ that is in $\cx$, when we measure the distance in the $H$ norm, is given by $\vx$ itself, i.e., $\Pi^{H}_{\cx} \left( \vx + H^{-1}\vd\right) = \vx$. This holds for any $H \in \mathcal{S}^n_{++}$, and in particular it also holds for $H = I^n$.
  \end{proof}
\end{lemma}

\begin{lemma}
  \label{lemma:optCond}
  Given a matrix $H \in \mathcal{S}^n_{++}$,
  an \(\vx \in \cx\) satisfies:
  \begin{align}
  \vx = \Pi^{H}_{\cx}\left(\vx - H^{-1} \nabla f(\vx) \right), \label{eq:OptProblem3}
  \end{align}
  if and only if $\vx = \vx^*$\, where $\vx^* = \argmin_{\vx \in \cx}f(\vx)$.
  \begin{proof}
    ($\Rightarrow$) Using the first-order optimality conditions for the scaled
    projection problem, shown in Remark~\ref{Remark:Appx:Optimality}, and particularizing for $\tilde{\vx}_{k+1}^* = \vx_k = \vx$ we have that for all $\vz\in\mathcal{X}$:
    \begin{align}
        \innp{H\left( \vx - \vx\right) + \nabla f(\vx), \vz - \vx} = \innp{\nabla f(\vx), \vz - \vx} \geq 0, \label{eq:OptCond2}
    \end{align}
    which hold true if and only if $\vx = \vx^*$, as Equation~\eqref{eq:OptCond2} represents the first-order optimality conditions for Problem~\ref{eq:OptProblem}, of which $\vx^*$ is the unique optimal solution. \\
    ($\Leftarrow$) Assume that $\vx = \vx^*$, then $-\nabla f\left( \vx^*\right) \in N_{\cx}\left( \vx^*\right)$. By the application of Lemma~\ref{lemma:optCondProj} we have that for any $H \in \mathcal{S}^n_{++}$ then it holds that $\vx = \Pi^{H}_{\cx}\left(\vx - H^{-1} \nabla f(\vx) \right)$.
  \end{proof}
\end{lemma}

Another interesting property of the PVM algorithm is the fact that the direction $\tilde{\vx}_{k+1}^* - \vx_{k}$ in Line~\ref{alg:Appx:MinimizationVarMetric} of Algorithm~\ref{algo:Appx:varMetricExact} is a descent direction regardless of how well $H_k\in \mathcal{S}^n_{++}$ approximates the Hessian $\nabla^2 f(\vx_k)$, this is formalized in Lemma~\ref{Lemma:Appx:DescentDir}. Note that despite this, we cannot guarantee that $f(\tilde{\vx}_{k+1}^*) \leq f(\vx_k)$, which is why to ensure primal progress at each iteration a line search or a bounded step size is often used in Line~\ref{alg:Appx:MinimizationVarMetric} of Algorithm~\ref{algo:Appx:varMetricExact}.

\begin{lemma}[Descent property of Projected Variable-Metric directions]~\citep[Section 7.2.1]{nemirovski2020OPTIII} \label{Lemma:Appx:DescentDir}
If $H_k\in \mathcal{S}^n_{++}$ and $\vx_k \neq \vx^*$, then the directions given by $\tilde{\vx}^*_{k+1} - \vx_k$, where $\tilde{\vx}^*_{k+1} = \argmin_{\vx \in \mathcal{X}} \hat{f}_{k} (\vx)$ are descent directions at point $\vx_{k}$, i.e., they satisfy $\innp{-\nabla f(\vx_k), \tilde{\vx}^*_{k+1} - \vx_k} > 0$.
  \begin{proof}
  Using the first-order optimality conditions shown in Remark~\ref{Remark:Appx:Optimality} for the scaled projection subproblem and particularizing for $\vz = \vx_{k}$:
  \begin{align*}
    \innp{-\nabla f(\vx_k) , \tilde{\vx}^*_{k+1}- \vx_{k}} \geq \norm{ \tilde{\vx}^*_{k+1} - \vx_{k}}^2_{H_k} > 0.
  \end{align*}
  Where the last strict inequality follows from the fact that we have assumed that $\vx_k \neq \vx^*$, and consequently $ \tilde{\vx}^*_{k+1} \neq \vx_k$ by application of Lemma~\ref{lemma:optCond}, and the assumption that $H_k\in \mathcal{S}^n_{++}$, thus \( \norm{ \tilde{\vx}^*_{k+1} - \vx_{k}}^2_{H_k}>0\). 
  \end{proof}
\end{lemma}

\subsection{Global Convergence} \label{appx:VarMetricAlgoGlobal}

The global primal gap convergence of the PVM algorithm (Algorithm~\ref{algo:Appx:varMetricExact}) with bounded step sizes is a well-known result that we reproduce here for completeness, as we will compare this global convergence rate with that of other first-order optimization algorithms. In order to prove it, we review Lemma~\ref{lemma:QuadraticSubproblems} which will be used in the global convergence proof.

\begin{lemma}~\citep[Lemma 9]{karimireddy2018adaptive} \label{lemma:QuadraticSubproblems}
  Given a convex domain $\mathcal{X}$ and $H_k \in \mathcal{S}^n_{++}$ then for constants $\alpha>0$ and $\nu>0$ such that $\alpha \nu \geq 1$ we have that:
  \begin{align}
      \min\limits_{\vx\in\cx} \left(\innp{\nabla f(\vx_k),\vx - \vx_k} + \frac{\alpha}{2} \norm{\vx - \vx_k}^2_{H_k}\right) \leq \frac{1}{\alpha\nu} \min\limits_{\vx\in\cx} \left(\innp{\nabla f(\vx_k),\vx - \vx_k} + \frac{1}{2\nu} \norm{\vx - \vx_k}^2_{H_k} \right).
  \end{align}
\end{lemma}

With the previous Lemma at hand, we can prove the global linear convergence in primal gap of the PVM algorithm with bounded step size when minimizing a $\mu$-strongly convex and $L$-smooth function over a convex set $\cx$.

\begin{theorem}[Global convergence of Projected Variable-Metric algorithm with bounded step size.]~\citep[Theorem 4]{karimireddy2018global}   \label{Appx:Theorem:GlobalRateExact}
   Given an $L$-smooth and $\mu$-strongly convex function and a convex set $\mathcal{X}$ then the Projected Variable-Metric algorithm (Algorithm~\ref{algo:Appx:varMetricExact}) with a step size $\gamma_k \leq \frac{\mu}{L \eta_k}$ guarantees for all $k \geq 0$:
   \begin{align*}
      f(\vx_{k+1}) - f(\vx^*) \leq \left( 1 - \frac{\mu\gamma_k^2}{L\eta_k}\right) \left(f(\vx_{k}) - f(\vx^*) \right),
  \end{align*} 
  where the parameter $\eta_k$ measures how well $H_k$ approximates $\nabla^2 f(\vx_k)$.
  \begin{proof}
  The iterate $\vx_{k+1}$ can be rewritten as:
  \begin{align}
      \vx_{k+1} = \argmin\limits_{\vx \in (1 - \gamma_k)\vx_k + \gamma_k \mathcal{X}} \innp{\nabla f(\vx_k), \vx - \vx_k} + \frac{1}{2\gamma_k}\norm{\vx - \vx_k}^2_{H_k} \label{Eq:DefinitionOptStepSize}
  \end{align}
  Using $L$-smoothness and the $\mu$-strong convexity of the function $f$ we can write:
  \begin{align}
      f(\vx_{k+1}) - f(\vx_k) &\leq \innp{\nabla f(\vx_k), \vx_{k+1} - \vx_k} + \frac{L}{2 \mu}\norm{\vx_{k+1} - \vx_k}^2_{\nabla^2 f(\vx_k)} \\
      &\leq \innp{\nabla f(\vx_k), \vx_{k+1} - \vx_k} + \frac{L\eta_k}{2 \mu}\norm{\vx_{k+1} - \vx_k}^2_{H_k} \\
      &\leq \innp{\nabla f(\vx_k), \vx_{k+1} - \vx_k} + \frac{1}{2\gamma_k}\norm{\vx_{k+1} - \vx_k}^2_{H_k} \\
      &= \min\limits_{\vx \in (1 - \gamma_k)\vx_k + \gamma_k \mathcal{X}}\left( \innp{\nabla f(\vx_k), \vx - \vx_k} + \frac{1}{2\gamma_k}\norm{\vx - \vx_k}^2_{H_k} \right) \label{Eq:GlobalConvergenceExact}.
  \end{align}
  Where the second inequality follows from Equation~\eqref{Eq:ScalingHessian} (which in turn is a consequence of $H_k \in \mathcal{S}^n_{++}$) and the third inequality follows from the fact that $\gamma_k \leq \frac{\mu}{L\eta_k}$. Applying Lemma~\ref{lemma:QuadraticSubproblems} to Equation~\eqref{Eq:GlobalConvergenceExact} and noting that as $H_k \in \mathcal{S}^n_{++}$ we can apply Equation~\eqref{Eq:ScalingHessian} and transform the minimization problem involving $\norm{\vx - \vx_k}_{H_k}$ to one that involves $\norm{\vx - \vx_k}_{\nabla^2 f(\vx_k)}$. Continuing with the chain of inequalities: 
    \begin{align}
      f(\vx_{k+1}) - f(\vx_k) &\leq\min\limits_{\vx \in (1 - \gamma_k)\vx_k + \gamma_k \mathcal{X}}\left( \innp{\nabla f(\vx_k), \vx - \vx_k} + \frac{1}{2\gamma_k}\norm{\vx - \vx_k}^2_{H_k} \right)  \\
      & \leq \frac{\mu \gamma_k}{L\eta_k}  \min\limits_{\vx \in (1 - \gamma_k)\vx_k + \gamma_k \mathcal{X}}\left( \innp{\nabla f(\vx_k), \vx - \vx_k} + \frac{\mu}{2L\eta_k}\norm{\vx - \vx_k}^2_{H_k}\right) \label{eq:plugging_lemma} \\
      &\leq\frac{\mu \gamma_k}{L\eta_k}  \min\limits_{\vx \in (1 - \gamma_k)\vx_k + \gamma_k \mathcal{X}} \left(\innp{\nabla f(\vx_k), \vx - \vx_k} + \frac{\mu}{2L}\norm{\vx - \vx_k}^2_{\nabla^2 f(\vx_k)}\right) \label{eq:before_plugging_optimum}\\
      &\leq\frac{\mu \gamma_k^2}{L\eta_k} \left(  \innp{\nabla f(\vx_k), \vx^* - \vx_k} + \frac{\mu\gamma_k}{2L}\norm{\vx^*- \vx_k}^2_{\nabla^2 f(\vx_k)} \right) \label{eq:plugging_optimum}\\ 
      &\leq\frac{\mu \gamma_k^2}{L\eta_k} \left(  \innp{\nabla f(\vx_k), \vx^* - \vx_k} + \frac{\mu}{2L}\norm{\vx^*- \vx_k}^2_{\nabla^2 f(\vx_k)} \right) \label{eq:particularize_stepsize} \\
      & \leq \frac{\mu \gamma_k^2}{L\eta_k} \left(  f(\vx^*) - f(\vx_k) \right). \label{eq:strong_convexity_deriv} 
  \end{align}
  We obtain Equation~\eqref{eq:plugging_lemma} by applying Lemma~\ref{lemma:QuadraticSubproblems}, and  Equation~\eqref{eq:before_plugging_optimum} from applying  Lemma~\ref{lemma:scalingCondition} to the norm term in Equation~\eqref{eq:plugging_lemma}, which allows us to use that $1/ \eta_k\norm{\vx - \vx_k}^2_{H_k} \leq \norm{\vx - \vx_k}^2_{\nabla^2 f(\vx_k)}$. Equation~\eqref{eq:plugging_optimum} follows from plugging in $\vx =(1 - \gamma_k) \vx_k + \gamma_k \vx^*$ into Equation~\eqref{eq:before_plugging_optimum} (as of course $\vx^* \in \mathcal{X}$). We obtain Equation~\eqref{eq:particularize_stepsize} by considering that $\gamma_k \leq 1$, and Equation~\eqref{eq:strong_convexity_deriv} from the  $\mu$-strong convexity and $L$-smoothness of the function $f(\vx)$. Reordering the previous expression leads to:
  \begin{align*}
      f(\vx_{k+1}) - f(\vx^*) \leq \left( 1 - \frac{\mu\gamma_k^2}{L\eta_k}\right) \left(f(\vx_{k}) - f(\vx^*) \right).
  \end{align*}
  \end{proof}
\end{theorem}

  As the exact line search strategy makes at least as much progress as choosing any $\gamma_k \leq \frac{\mu}{L \eta_k}$, the bound in Theorem~\ref{Appx:Theorem:GlobalRateExact} also holds for the Projected Variable-Metric algorithm (Algorithm~\ref{algo:Appx:varMetricExact}) with exact line search.

\begin{corollary}[Global convergence of Projected Variable-Metric algorithm with exact line search or $\gamma_k = \frac{\mu}{L\eta_k}$]  \label{Appx:corollary:GlobalRateExactLineSearch}
   Given an $L$-smooth and $\mu$-strongly convex function and a convex set $\mathcal{X}$ then the Projected Variable-Metric algorithm (Algorithm~\ref{algo:Appx:varMetricExact}) with an exact line search or with a step size $\gamma_k = \frac{\mu}{L \eta_k}$ guarantees for all $k \geq 0$:
   \begin{align*}
      f(\vx_{k+1}) - f(\vx^*) \leq \left( 1 - \frac{\mu^3}{L^3\eta_k^3}\right) \left(f(\vx_{k}) - f(\vx^*) \right),
  \end{align*} 
  where the parameter $\eta_k$ measures how well $H_k$ approximates $\nabla^2 f(\vx_k)$.
\end{corollary}

As was mentioned in Lemma~\ref{Lemma:Appx:DescentDir} the direction $\tilde{\vx}_{k+1}^* - \vx_{k}$ in Line~\ref{alg:Appx:MinimizationVarMetric} of Algorithm~\ref{algo:Appx:varMetricExact} is a descent direction regardless of how well $H_k\in \mathcal{S}^n_{++}$ approximates the Hessian $\nabla^2 f(\vx_k)$. However, as we can see in Theorem~\ref{Appx:Theorem:GlobalRateExact} and Corollary~\ref{Appx:corollary:GlobalRateExactLineSearch}, if we pick a matrix $H_k\in \mathcal{S}^n_{++}$ that approximates the Hessian $\nabla^2 f(\vx_k)$ well, that is, we have an $\eta_k$ close to 1, we will be able to guarantee more primal progress per step when using an exact line search or bounded step sizes.

One of the key consequences of Corollary~\ref{Appx:corollary:GlobalRateExactLineSearch} is that even if we run the PVM algorithm with an exact line search and we use $H_k = \nabla^2 f(\vx_k)$ (which is equivalent to $\eta_k = 1$), we need $\mathcal{O}(L^3/\mu^3 \log 1/\varepsilon)$ iterations to reach an $\varepsilon$-optimal solution to Problem~\eqref{eq:OptProblem}. This stands in contrast to the PGD algorithm, which requires $\mathcal{O}(L/\mu \log 1/\varepsilon)$ iterations, or \emph{Nesterov's Projected Gradient Descent} (NPGD) algorithm, which requires $\mathcal{O}(\sqrt{L/\mu} \log 1/\varepsilon)$ iterations to reach an $\varepsilon$-optimal solution. Note that with a small modification of the proof in Theorem~\ref{Appx:Theorem:GlobalRateExact} we can recover the same rate for the PGD algorithm and the PVM algorithm with $H_k = I^n$. This is expected, as in this case the algorithms are equivalent, except for the bounded step size strategy.

\begin{theorem}[Global convergence of Projected Variable-Metric algorithm with bounded step size and $H_k = I^n$]   \label{Appx:Theorem:GlobalRateExactEuclidean}
   Given an $L$-smooth and $\mu$-strongly convex function and a convex set $\mathcal{X}$ then the Projected Variable-Metric algorithm (Algorithm~\ref{algo:Appx:varMetricExact}) with a step size $\gamma_k \leq \min\{ 1, \frac{1}{L } \}$ and $H_k = I^n$ guarantees for all $k \geq 0$:
   \begin{align}
      f(\vx_{k+1}) - f(\vx^*) \leq \left( 1 - \frac{\mu}{L}\right) \left(f(\vx_{k}) - f(\vx^*) \right).
  \end{align} 
  \begin{proof}
  The proof mirrors that of Theorem~\ref{Appx:Theorem:GlobalRateExact}, and so we only give a brief outline.   The iterate $\vx_{k+1}$ can be rewritten as:
  \begin{align}
      \vx_{k+1} = \argmin\limits_{\vx \in (1 - \gamma_k)\vx_k + \gamma_k \mathcal{X}} \innp{\nabla f(\vx_k), \vx - \vx_k} + \frac{1}{2\gamma_k}\norm{\vx - \vx_k}^2. \label{Eq:Appx:ReqriteIterate}
  \end{align} Using $L$-smoothness we can write:
  \begin{align}
      f(\vx_{k+1}) - f(\vx_k) &\leq \innp{\nabla f(\vx_k), \vx_{k+1} - \vx_k} + \frac{L}{2}\norm{\vx_{k+1} - \vx_k}^2 \\
      &\leq \innp{\nabla f(\vx_k), \vx_{k+1} - \vx_k} + \frac{1}{2\gamma_k}\norm{\vx_{k+1} - \vx_k}^2 \label{Eq:Appx:StepsizeGlobalConvergenceExact} \\
      &= \min\limits_{\vx \in (1 - \gamma_k)\vx_k + \gamma_k \mathcal{X}}\left( \innp{\nabla f(\vx_k), \vx - \vx_k} + \frac{1}{2\gamma_k}\norm{\vx - \vx_k}^2 \right) \label{Eq:GlobalConvergenceExact2} \\
     &\leq \gamma_k\mu \min\limits_{\vx \in (1 - \gamma_k)\vx_k + \gamma_k \mathcal{X}}\left( \innp{\nabla f(\vx_k), \vx - \vx_k} + \frac{\mu}{2}\norm{\vx - \vx_k}^2 \right) \label{Eq:GlobalConvergenceExact3} \\
      &\leq \gamma_k\mu \left( \innp{\nabla f(\vx_k), \vx^* - \vx_k} + \frac{\mu}{2}\norm{\vx^* - \vx_k}^2 \right) \label{Eq:GlobalConvergenceExact4} \\
        &\leq \frac{\mu}{L} \left( f(\vx^*) - f(\vx_k)) \right) \label{Eq:GlobalConvergenceExact5}.
  \end{align}
  Where Equation~\eqref{Eq:Appx:StepsizeGlobalConvergenceExact} follows from $\gamma_k \leq \min\{ 1, \frac{1}{L } \}$ and Equation~\eqref{Eq:GlobalConvergenceExact2} follows from Equation~\eqref{Eq:Appx:ReqriteIterate}. Applying Lemma~\ref{lemma:QuadraticSubproblems} to Equation~\eqref{Eq:GlobalConvergenceExact2} leads to Equation~\eqref{Eq:GlobalConvergenceExact3}. Equation~\eqref{Eq:GlobalConvergenceExact4} follows from plugging in $\vx =(1 - \gamma_k) \vx_k + \gamma_k \vx^*$ into Equation~\eqref{Eq:GlobalConvergenceExact3} (as of course $\vx^* \in \mathcal{X}$)
  Lastly, in Equation~\eqref{Eq:GlobalConvergenceExact5} we have used $\mu$-strong convexity and the fact that $\gamma_k \leq \min\{ 1, \frac{1}{L } \}$. Reordering the terms previous inequality completes the proof.
  \end{proof}
\end{theorem}

\subsection{Local Convergence} \label{appx:VarMetricAlgoLocal}

Despite the lackluster convergence rate in primal gap shown in Theorem~\ref{Appx:Theorem:GlobalRateExact}, the PVM algorithm can achieve quadratic convergence in distance to the optimum when the iterates are close enough to the optimum and the Hessian approximations are accurate enough. We first review a series of results that will allow us to prove the local quadratic convergence of the PVM algorithm. One of the key properties that is often used in the convergence proof of the PGD algorithm is the non-expansiveness of the Euclidean projection operator onto a convex set $\cx$, denoted by  $\Pi^{I^n}_{\mathcal{X}}$. In the local convergence proof of the PVM algorithm we use a generalization of the aforementioned fact, that is, the scaled projection operator onto a convex set $\cx$, denoted by $\Pi^{H_k}_{\mathcal{X}}$ where $H  \in \mathcal{S}^n_{++}$, is also non-expansive (see Lemma~\ref{lemma:scaledProj}).

\begin{lemma}
  \label{lemma:scaledProj}~\cite{beck2017first}[Theorem 6.42]
  Given a $H  \in \mathcal{S}^n_{++}$ and a convex set $\mathcal{X}$, the
  scaled projection is a contraction mapping (it is firmly-nonexpansive) in the $H$-norm:
  \begin{align*}
  (\vx - \vy)^T H (\Pi^{H}_{\mathcal{X}}\left(\vx\right) - \Pi^{H}_{\mathcal{X}}\left(\vy\right)) \geq \norm{\Pi^{H}_{\mathcal{X}}\left(\vx\right) - \Pi^{H}_{\mathcal{X}}\left(\vy\right)}^2_{H}
  \end{align*}
  Using the Cauchy-Schwarz inequality this leads to $\norm{\vx - \vy}_H \geq \norm{\Pi^{H}_{\mathcal{X}}\left(\vx\right) - \Pi^{H}_{\mathcal{X}}\left(\vy\right)}_{H}$.
\end{lemma}

The following Lemma, which is intimately linked with the $L_2$-Lipschitzness of the Hessian $\nabla^2 f(\vx)$, will also be key in the proof of quadratic local convergence.

\begin{lemma}
  \label{lemma:LipschitzHessian}~\cite{nesterov2018lectures}[Lemma 4.1.1]
  If a twice differentiable function $f$ has $L_2$-Lipschitz continuous Hessian over $\mathcal{X}$ then for all $\vx, \vy \in \mathcal{X}$:
  \begin{align*}
  \norm{\nabla f(\vy) - \nabla f(\vx) - \nabla^2 f(\vx)\left( \vy - \vx\right)} \leq \frac{L_2}{2} \norm{\vy - \vx}^2.
  \end{align*}
\end{lemma}

With the results from Lemma~\ref{lemma:scaledProj} and Lemma~\ref{lemma:LipschitzHessian} we can formalize the local convergence of the PVM algorithm.

\begin{lemma}[Local convergence of Projected Variable-Metric algorithm] \label{lemma:convervenceExactProjAppx}
Given an $L$-smooth and $\mu$-strongly convex function with $L_2$-Lipschitz Hessian and a compact convex set $\mathcal{X}$, if $\tilde{\vx}^*_{k+1} = \argmin\limits_{\vx \in \mathcal{X}} \hat{f}_{k} (\vx)$
then for all $k \geq 0$:
\begin{align*}
   \norm{\tilde{\vx}^*_{k+1} - \vx^*}^2 & \leq \frac{\eta_k^2 L_2^2}{4\mu^2}\norm{\vx_{k} - \vx^*}^4 + \frac{2 L \eta_k\left( \eta_k - 1 \right)}{\mu} \norm{\vx_{k} - \vx^*}^2.
\end{align*}
where the parameter $\eta_k$ measures how well $H_k$ approximates $\nabla^2 f(\vx_k)$.
  \begin{proof}
 Using the definition of $\tilde{\vx}^*_{k+1}$ (see Remark \ref{fact:equivNewton}) and Lemma~\ref{lemma:optCond} we have:
        \begin{align}
       \norm{\tilde{\vx}^*_{k+1} - \vx^*}_{H_k}^2  =& \norm{\Pi^{H_k}_{\cx}\left(\vx_{k} - H_k^{-1} \nabla f(\vx_{k}) \right) - \Pi^{H_k}_{\cx}\left(\vx^* - H_k^{-1} \nabla f(\vx^*) \right) }_{H_k}^2 \\
         \leq & \norm{\left(\vx_{k} - \vx^* \right) - H_k^{-1} \left( \nabla f(\vx_{k}) - \nabla f(\vx^*) \right) }_{H_k}^2 \\
         = &\norm{H_k \left(\vx_{k} - \vx^* \right) -  \left(\nabla f(\vx_{k}) - \nabla f(\vx^*) \right)  }_{H_k^{-1}}^2 \\
         = &\norm{\vx_{k} - \vx^*}_{H_k}^2 + \norm{\nabla f(\vx_{k}) - \nabla f(\vx^*)}^2_{H^{-1}_k}\\
        & - 2 \innp{\vx_{k} - \vx^*, \nabla f(\vx_{k}) - \nabla f(\vx^*)}. \label{Eq:ConvIntermediate}
    \end{align}
        Where the first inequality is a consequence of Lemma~\ref{lemma:scaledProj}. We can apply Lemma~\ref{lemma:scalingCondition} and Remark~\ref{lemma:scalingConditionInv} to bound $\norm{\vx_{k} - \vx^*}_{H_k}^2$ and $\norm{\nabla f(\vx_{k}) - \nabla f(\vx^*)}^2_{H^{-1}_k}$ in Equation~\eqref{Eq:ConvIntermediate}, this allows us to write: 
    \begin{align*}
       \norm{\tilde{\vx}^*_{k+1} - \vx^*}_{H_k}^2  & \leq \eta_k\norm{\nabla f(\vx_{k}) - \nabla f(\vx^*)}^2_{\nabla^2 f(\vx_{k})^{-1}} - 2\eta_k \innp{\vx_{k} - \vx^*, \nabla f(\vx_{k}) - \nabla f(\vx^*)}\\
        & \quad +\eta_k  \norm{\vx_{k} - \vx^*}_{\nabla^2 f(\vx_{k})}^2     + 2 \left( \eta_k - 1 \right) \innp{\vx_{k} - \vx^*, \nabla f(\vx_{k}) - \nabla f(\vx^*)} \\
        & = \eta_k \norm{\nabla^2 f(\vx_{k})  \left(\vx_{k} - \vx^* \right) -  \left(\nabla f(\vx_{k}) - \nabla f(\vx^*) \right)  }_{\nabla^2 f(\vx_{k})^{-1}}^2  \\ 
        & \quad + 2 \left( \eta_k - 1 \right) \innp{\vx_{k} - \vx^*, \nabla f(\vx_{k}) - \nabla f(\vx^*)}  \\
        & \leq \frac{\eta_k}{\mu} \norm{\nabla^2 f(\vx_{k})  \left(\vx_{k} - \vx^* \right) -  \left(\nabla f(\vx_{k}) - \nabla f(\vx^*) \right)  }^2  \\
        & \quad + 2 \left( \eta_k - 1 \right) \innp{\vx_{k} - \vx^*, \nabla f(\vx_{k}) - \nabla f(\vx^*)}.
    \end{align*}
 The last inequality is a consequence of the $\mu$-strong convexity of $f$ which ensures that $ \nabla^2 f(\vx_{k})^{-1}\preceq \mu^{-1} I^n$. Using the fact that the Hessian is  $L_2$-Lipschitz and applying Lemma~\ref{lemma:LipschitzHessian} and using the $L$-smoothness of $f$ leads to:
    \begin{align}
       \norm{\tilde{\vx}^*_{k+1} - \vx^*}_{H_k}^2 & \leq \frac{\eta_k L_2^2}{4\mu}\norm{\vx_{k} - \vx^*}^4 + 2 L\left( \eta_k - 1 \right) \norm{\vx_{k} - \vx^*}^2.
       \label{Eq:ExactHessianIntermediate3}
    \end{align}
     Using Lemma~\ref{lemma:scalingCondition} along with the $\mu$-strong convexity of $f$ and reordering the expression shown in Equation~\eqref{Eq:ExactHessianIntermediate3} completes the proof.
  \end{proof}
\end{lemma}

As we can see, even if the scaled projection subproblems are solved to optimality  we arrive at a convergence rate for $\norm{\tilde{\vx}^*_{k+1}- \vx^*}$ that is linear-quadratic in terms of $\norm{\vx_{k} -\vx^*}$, and we do not obtain local quadratic convergence without additional assumptions on how well $H_k$ approximates $\nabla^2 f(\vx_k)$, due to $\eta_k - 1$ in the second term in Equation~\eqref{Eq:ExactHessianIntermediate3}. This can be remedied with Assumption~\ref{assumption:accuracyHessian}:

\begin{corollary} \label{Corollary:Appx:ConvExact}
  If in addition to the conditions described in Lemma~\ref{lemma:convervenceExactProjAppx} we also assume that Assumption~\ref{assumption:accuracyHessian} is satisfied, we have:
  \begin{align}
   \norm{\tilde{\vx}^*_{k+1} - \vx^*} & \leq  \sqrt{\frac{\eta_k}{\mu}\left( \frac{\eta_k L_2^2}{4\mu} + 2 L\omega \right) }\norm{\vx_{k} - \vx^*}^2, \label{Eq:Appx:ConvExact}
\end{align}
where $\omega \geq 0$ is described in Equation~\eqref{Eq:accuracyParam}.
\end{corollary}

Even though $\norm{\tilde{\vx}^*_{k+1} - \vx^*}$ may converge quadratically, what we are interested in is in the quadratic convergence of $\norm{\vx_{k+1} - \vx^*}$, formed as $\vx_{k+1} = \vx_k + \gamma_k (\tilde{\vx}^*_{k+1} - \vx_k)$, that is:
\begin{align}
 \norm{\vx_{k+1} - \vx^*} &=   \norm{\vx_{k} + \gamma_k \left(\tilde{\vx}_{k+1}^* - \vx_{k} \right)  - \vx^* } \\
 &=   \norm{\left(1 - \gamma_k \right)\left(\vx_{k} - \vx^*\right) + \gamma_k \left(\tilde{\vx}^*_{k+1} - \vx^* \right) } \\
& \leq  \left(1 - \gamma_k \right)\norm{\vx_{k} - \vx^*} + \gamma_k\norm{\tilde{\vx}^*_{k+1} - \vx^*}. \label{eq:inexactProjections}
\end{align}
We can see from Equation~\eqref{eq:inexactProjections} that we will only have the desired convergence rate if $\left(1 - \gamma_k \right) \leq\beta \norm{\vx_{k} - \vx^*}$ for some $\beta \geq 0$, that is, we either need to set $\gamma_k=1$, or select a step size strategy that makes $\gamma_k$ converge to $1$ fast enough. 

\newpage

\section{Second-order Conditional Gradient Sliding} \label{appx:SOCG}

In Section~\ref{Section:Appx:InexactScaledProjections} we prove that the Inexact PVM steps (Lines~\ref{algLine:PNStep1}-\ref{algLine:PNStep3} in Algorithm~\ref{algo:proj-Newton}) that the SOCGS algorithm computes contract the distance to the optimum and the primal gap quadratically when close enough to $\vx^*$, by carefully choosing the $\varepsilon_k$-parameter at each iteration. First, we review the SOCGS from the main body of the text (shown in Algorithm~\ref{algo:Appx:proj-Newton}), and then we review a key result in Lemma~\ref{Lemma:Appx:AccuracyHessian} that measures the accuracy of the Hessian matrix approximation $H$ as we approach $\vx^*$, which will be used in the convergence proofs.

\begin{algorithm}
\SetKwInOut{Input}{Input}\SetKwInOut{Output}{Output}
%\SetKwComment{Comment}{$\triangleright$\ }{}
\SetKwComment{Comment}{//}{}
\Input{Point $\vx \in \cx$}
\Output{Point $\vx_{K } \in \cx$}
\hrulealg
$\vx_0 \leftarrow \argmin_{\mathbf{v} \in  \mathcal{X}} \innp{\nabla f\left(\vx \right), \mathbf{v}}$, $\cs_0 \leftarrow \{ \vx_0 \}$, $\vlambda_0(\vx_0) \leftarrow 1$\;
$\vx_0^{\text{ACG}} \leftarrow \vx_0$, $\cs_0^{\text{ACG}} \leftarrow \cs_0$, $\vlambda_0^{\text{ACG}}(\vx_0) \leftarrow 1$\;
\For{$k = 0$ \textbf{to} $K - 1$}{
%\tikzmk{A}
$\vx^{\text{ACG}}_{k + 1}, \cs^{\text{ACG}}_{k + 1}, \vlambda^{\text{ACG}}_{k + 1} \leftarrow $ ACG$\left( \nabla f(\vx_k), \vx^{\text{ACG}}_{k} , \cs^{\text{ACG}}_{k}, \vlambda^{\text{ACG}}_{k} \right)$ \label{algLine:Appx:ACG} \Comment*[r]{ACG step}
%\tikzmk{B} \boxit{pink}
%\tikzmk{A}
$H_{k} \leftarrow \Omega \left(\vx_{k}  \right)$ \label{algLine:Appx:updateHessian} \Comment*[r]{Call Hessian oracle}
$\hat{f}_k\left( \vx\right) \leftarrow \innp{\nabla f\left( \vx_k\right), \vx - \vx_k} + \frac{1}{2} \norm{\vx - \vx_k}^2_{H_k}$ \label{algLine:Appx:buildQuadratic} \Comment*[r]{Build quadratic approximation}
$\varepsilon_k \leftarrow \left(\frac{lb\left( \vx_k\right)}{\norm{\nabla f\left( \vx_k\right)}}\right)^4$ \label{AlgLine:Appx:AccuracyParameter}\;
$\tilde{\vx}^0_{k + 1} \leftarrow \vx_k$, $\tilde{\cs}^0_{k + 1} \leftarrow \cs_k$, $\tilde{\vlambda}^0_{k + 1} \leftarrow \vlambda_k$, $t \leftarrow 0$ \label{algLine:Appx:TransferPoint} \; 
\While(\tcp*[f]{Compute Inexact PVM step}){$\max\limits_{\mathbf{v} \in \cx}   \langle \nabla \hat{f}_k ( \tilde{\vx}^t_{k + 1} ), \tilde{\vx}^t_{k + 1} - \mathbf{v} \rangle \geq \varepsilon_k$ \label{algLine:Appx:PNStep1}}{
$\tilde{\vx}^{t+1}_{k + 1}, \tilde{\cs}^{t+1}_{k + 1}, \tilde{\vlambda}^{t+1}_{k + 1} \leftarrow $ ACG$\left(\nabla  \hat{f}_k(\tilde{\vx}^{t}_{k + 1}), \tilde{\vx}^{t}_{k + 1}, \tilde{\cs}^{t}_{k + 1}, \tilde{\vlambda}^{t}_{k + 1} \right)$ \label{algLine:Appx:PNStep2} \;
$t \leftarrow t+1$\;
} \label{algLine:Appx:PNStep3} 
$\tilde{\vx}_{k + 1} \leftarrow \tilde{\vx}^t_{k + 1} $, $\tilde{\cs}_{k + 1} \leftarrow \tilde{\cs}^t_{k + 1}$, $\tilde{\vlambda}_{k + 1} \leftarrow \tilde{\vlambda}^t_{k + 1}$ \label{algLine:Appx:OutputPVM} \;
%\tikzmk{B}
% \boxit{lightblue}
\uIf{$f\left(\tilde{\vx}_{k + 1} \right) \leq f(\vx^{\textup{ACG}}_{k + 1} )$ \label{algLine:Appx:MonotonicityFirst}}{
$\vx_{k+1} \leftarrow \tilde{\vx}_{k + 1} $, $\cs_{k+1} \leftarrow \tilde{\cs}_{k + 1}$, $\vlambda_{k+1} \leftarrow \tilde{\vlambda}_{k + 1}$\label{algLine:Appx:PVMStep} \Comment*[r]{Choose Inexact PVM step}}
\Else{
$\vx_{k+1} \leftarrow \vx^{\text{ACG}}_{k + 1} $, $\cs_{k+1} \leftarrow \cs^{\text{ACG}}_{k + 1}$, $\vlambda_{k+1} \leftarrow \vlambda^{\text{ACG}}_{k + 1}$\Comment*[r]{Choose ACG step}} \label{algLine:Appx:ACGStep}
} \label{algLine:Appx:MonotonicityLast}
\caption{Second-order Conditional Gradient Sliding (SOCGS) Algorithm}\label{algo:Appx:proj-Newton}
\end{algorithm}

The algorithm couples an independent ACG step with line search (Line~\ref{algLine:Appx:ACG}) with an Inexact PVM step with unit step size (Lines~\ref{algLine:Appx:PNStep1}-\ref{algLine:Appx:PNStep3}). At the end of each iteration we choose the step that provides the greatest primal progress (Lines~\ref{algLine:Appx:MonotonicityFirst}-\ref{algLine:Appx:ACGStep}). The ACG steps in Line~\ref{algLine:Appx:ACG} will ensure global linear convergence in primal gap, and the Inexact PVM steps in Lines~\ref{algLine:Appx:MonotonicityFirst}-\ref{algLine:Appx:ACGStep} will provide quadratic convergence.

Note that the ACG iterates in Line~\ref{algLine:Appx:ACG} do not depend on the Inexact PVM steps in Lines~Lines~\ref{algLine:Appx:PNStep1}-\ref{algLine:Appx:PNStep3}. This is because the ACG steps do not contract the primal gap on a per-iteration basis, and if the active sets of the ACG steps in Line~\ref{algLine:Appx:ACG} were to be modified using the active set of the PVM steps in Lines~\ref{algLine:Appx:PNStep1}-\ref{algLine:Appx:PNStep3}, this would break the proof of linear convergence in Theorem~\ref{theorem:Appx:ConvergenceACG} for the ACG algorithm. The proof in Theorem~\ref{theorem:Appx:ConvergenceACG} crucially relies on the fact that at each iteration of the ACG algorithm we can pick up or drop at most one vertex from the active set, whereas a PVM step may have dropped or picked up multiple vertices from the active set. The line search in the ACG step (Line~\ref{algLine:Appx:ACG}) can be substituted with a step size strategy that requires knowledge of the $L$-smoothness parameter of $f(\vx)$ \cite{pedregosa2020linearly}.

We compute the scaled projection in the Inexact PVM step (Lines~\ref{algLine:Appx:MonotonicityFirst}-\ref{algLine:Appx:ACGStep}) using the ACG algorithm with exact line search, as the objective function is quadratic, thereby making the SOCGS algorithm (Algorithm~\ref{algo:Appx:proj-Newton}) projection-free. As the function being minimized in the Inexact PVM steps is quadratic there is a closed-form expression for the optimal step size in Line~\ref{algLine:Appx:PNStep2}. The scaled projection problem is solved to an accuracy $\varepsilon_k$ such that $\hat{f}_{k}(\tilde{\vx}_{k+1}) - \min_{\vx \in \cx} \hat{f}_k\left( \vx\right) \leq \varepsilon_k$, using the Frank-Wolfe gap as a stopping criterion, as in the CGS algorithm \cite{lan2016conditional}. The accuracy parameter $\varepsilon_k$ in the SOCGS algorithm depends on a lower bound on the primal gap of Problem~\ref{eq:OptProblem} which we denote by $lb\left( \vx_k \right)$ that satisfies $lb\left( \vx_k \right) \leq f\left(\vx_k \right) - f\left(\vx^* \right)$.

\begin{lemma} \label{Lemma:Appx:AccuracyHessian}
Given a $\mu$-strongly convex and $L$-smooth function $f(\vx)$ and a convex set $\cx$, then for any $\vx \in \cx$ and any matrix $H \in \mathcal{S}^n_{++}$ that satisfies Assumption~\ref{assumption:accuracyHessian} at $\vx$ we have that:
\begin{align}
\norm{ H^{-1} -  [ \nabla^2 f\left(\vx \right) ]^{-1} } & \leq \frac{\eta\omega}{\mu} \norm{\vx - \vx^*}^2. \label{Eq:Appx:LemmaInvHess}
\end{align}
Similarly, we also have that:
\begin{align}
\norm{ H -  \nabla^2 f\left(\vx \right)  } & \leq \eta\omega L \norm{\vx - \vx^*}^2. \label{Eq:Appx:LemmaHess}
\end{align}
\begin{proof}
We can bound the term on the left-hand side of Equation~\eqref{Eq:Appx:LemmaInvHess} as:
\begin{align}
    \norm{H^{-1}- [ \nabla^2 f\left(\vx\right) ]^{-1}} & = \norm{H^{-1}\left(H[\nabla^2 f\left(\vx \right)]^{-1}- I^n\right)} \\
    & \leq \norm{H^{-1}}\norm{H[\nabla^2 f\left(\vx \right)]^{-1}- I^n} \label{eq:submultiplicative}\\
    & = \lambda_{\max} \left(H^{-1} \right) \norm{H[\nabla^2 f\left(\vx \right)]^{-1}- I^n} \\
    & \leq \eta/\mu \norm{H[\nabla^2 f\left(\vx \right)]^{-1}- I^n}. \label{Eq:Appx:StrCvxty}
\end{align}
We obtain Equation~\eqref{eq:submultiplicative} from the fact that the spectral norm of a matrix is submultiplicative, and both matrices are square. The inequality shown in Equation~\eqref{Eq:Appx:StrCvxty} follows from $H\in \mathcal{S}^n_{++}$ and Corollary~\ref{corollary:Eigenvalues}. Proceeding similarly, we can also bound the previous quantity as:
\begin{align}
    \norm{H^{-1}- [ \nabla^2 f\left(\vx \right) ]^{-1} } & \leq 1/\mu\norm{\nabla^2 f\left(\vx \right)H^{-1}- I^n} \\
    & \leq \eta/\mu \norm{\nabla^2 f\left(\vx \right)H^{-1}- I^n}. \label{Eq:Appx:StrCvxtBound}
\end{align}
Where the inequality in Equation~\eqref{Eq:Appx:StrCvxtBound} follows from fact that $\eta \geq 1$. Putting together these bounds, we have that:
\begin{align*}
    \norm{H^{-1}- [ \nabla^2 f\left(\vx \right) ]^{-1} } &  \leq\frac{\eta}{\mu} \max\left\{\norm{H[\nabla^2 f\left(\vx \right)]^{-1}- I^n} ,\norm{\nabla^2 f\left(\vx \right)H^{-1}- I^n} \right\}.
\end{align*}
Each of the terms in the maximization operator in the previous equation can be written as:
\begin{align}
    \norm{\nabla^2 f\left(\vx \right)H^{-1}- I^n} & =  \sigma_{\max} \left(\nabla^2 f\left(\vx \right)H^{-1}- I^n\right)\\
    & = \max\limits_{1 \leq i \leq n} \abs{\lambda_{i} \left(\nabla^2 f\left(\vx \right)H^{-1}- I^n\right)}  \label{Eq:Appx:SingVal}\\
    & =  \max\limits_{1 \leq i \leq n} \abs{\lambda_{i} \left(\nabla^2 f\left(\vx \right)H^{-1}\right)- 1}. \label{Eq:Appx:EigenvalIdentity}
\end{align}
Where the equality in Equation~\eqref{Eq:Appx:SingVal} follows from the fact that the maximum singular value of a square matrix is equal to the maximum absolute value of the eigenvalues of the matrix. This allows us to write:
\begin{equation}
\begin{aligned}
 \max\left\{\norm{H[\nabla^2 f\left(\vx \right)]^{-1}- I^n} ,\norm{\nabla^2 f\left(\vx \right)H^{-1}- I^n} \right\}= \max\bigg\{  & \max\limits_{1 \leq i \leq n}\abs{\lambda_{i} \left(\nabla^2 f\left(\vx \right)H^{-1}\right)- 1} ,\\
      & \max\limits_{1 \leq i \leq n}\abs{\lambda_{i} \left(H[\nabla^2 f\left(\vx \right)]^{-1}\right)- 1} \bigg\} \\
= \max\limits_{1 \leq i \leq n}&\bigg\{   \max\big\{\abs{\lambda_{i} \left(\nabla^2 f\left(\vx \right)H^{-1}\right)- 1} ,\\
      & \abs{\lambda_{n + 1 -i} \left(H[\nabla^2 f\left(\vx \right)]^{-1}\right)- 1} \big\} \bigg\}. \label{Eq:Appx:BoundMatrix} 
\end{aligned}
\end{equation}
We can get rid of the absolute values in the previous expression using the fact that if $0 < z \leq 1$, where $z\in \rr$, then $\abs{z - 1} \leq 1 / z - 1$. Note that as $H, \nabla^2 f\left(\vx \right) \in \mathcal{S}^n_{++}$ we have that $\lambda_{i} \left(H[\nabla^2 f\left(\vx \right)]^{-1}\right) = \lambda_{i} \left([\nabla^2 f\left(\vx \right)]^{-1/2}H[\nabla^2 f\left(\vx \right)]^{-1/2}\right)>0$ and $\lambda_{i} \left(\nabla^2 f\left(\vx \right)H^{-1}\right)=\lambda_{i} \left(H^{-1/2} \nabla^2 f\left(\vx \right)H^{-1/2}\right)>0$, moreover $\lambda_{i} \left(H[\nabla^2 f\left(\vx \right)]^{-1}\right) = 1/ \lambda_{n+1-i} \left(\nabla^2 f\left(\vx \right)H^{-1}\right)$ as $(H[\nabla^2 f\left(\vx \right)]^{-1})^{-1} = \nabla^2 f\left(\vx \right)H^{-1}$. This means that for any $1\leq i\leq n$ such that $\lambda_{i} \left(\nabla^2 f\left(\vx \right)H^{-1}\right) \leq 1$ we have that:
\begin{align*}
\big\lvert \lambda_{i} \left(\nabla^2 f\left(\vx \right)H^{-1}\right) - 1\big\rvert &\leq 1/ \lambda_{i} \left(\nabla^2 f\left(\vx \right)H^{-1}\right) - 1 = \lambda_{n+1-i} \left(H[\nabla^2 f\left(\vx \right)]^{-1}\right) - 1.
\end{align*}
Similarly, for any $1\leq i\leq n$ such that $\lambda_{i} \left(H[\nabla^2 f\left(\vx \right)]^{-1}\right) \leq 1$ we have that:
\begin{align*}
\big\lvert \lambda_{i} \left(H[\nabla^2 f\left(\vx \right)]^{-1}\right) - 1\big\rvert &\leq 1/ \lambda_{i} \left(H[\nabla^2 f\left(\vx \right)]^{-1}\right) - 1 = \lambda_{n+1-i} \left(\nabla^2 f\left(\vx \right)H^{-1}\right) - 1.
\end{align*}
This means that for all $1\leq i\leq n$ we have that:
\begin{align*}
\max\bigg\{\abs{\lambda_{i} \left(\nabla^2 f\left(\vx \right)H^{-1}\right)- 1} ,\abs{\lambda_{n + 1 -i} \left(H[\nabla^2 f\left(\vx \right)]^{-1}\right)- 1} \bigg\} \leq \max\bigg\{\lambda_{i} \left(\nabla^2 f\left(\vx \right)H^{-1}\right)- 1 ,\lambda_{n + 1 -i} \left(H[\nabla^2 f\left(\vx \right)]^{-1}\right)- 1 \bigg\}.
\end{align*}
Which allows us to write Equation~\eqref{Eq:Appx:BoundMatrix} as:
\begin{equation*}
\begin{aligned}
 \max\left\{\norm{H[\nabla^2 f\left(\vx \right)]^{-1}- I^n} ,\norm{\nabla^2 f\left(\vx \right)H^{-1}- I^n} \right\}= \max\bigg\{  & \max\limits_{1 \leq i \leq n}\left(\lambda_{i} \left(\nabla^2 f\left(\vx \right)H^{-1}\right)- 1 \right),\\
      & \max\limits_{1 \leq i \leq n}\left(\lambda_{i} \left(H[\nabla^2 f\left(\vx \right)]^{-1}\right)- 1\right) \bigg\}. 
\end{aligned}
\end{equation*}
Which immediately leads to:
\begin{align}
 \max\bigg\{ & \max\limits_{1 \leq i \leq n}\left(\lambda_{i} \left(\nabla^2 f\left(\vx \right)H^{-1}\right)- 1 \right),      \max\limits_{1 \leq i \leq n}\left(\lambda_{i} \left(H[\nabla^2 f\left(\vx \right)]^{-1}\right)- 1\right) \bigg\}  \\
     & = \max\left\{ \lambda_{\max} \left(\nabla^2 f\left(\vx \right)H^{-1}\right) , \lambda_{\max} \left(H[\nabla^2 f\left(\vx \right)]^{-1}\right) \right\} - 1 \\
    & = \eta - 1 \label{Eq:Appx:DefinitionHessAcc}\\ 
    & \leq \omega \norm{\vx - \vx^*}^2.  \label{Eq:Appx:HessAssump}
\end{align}
Where Equation~\eqref{Eq:Appx:DefinitionHessAcc} follows from the definition of $\eta$ and Equation~\eqref{Eq:Appx:HessAssump} follows from Assumption~\ref{assumption:accuracyHessian}. Putting this all together allows us to write:
\begin{align*}
\norm{ H^{-1}- [ \nabla^2 f\left(\vx \right) ]^{-1} } & \leq \frac{\eta\omega}{\mu} \norm{\vx - \vx^*}^2.
\end{align*}
The claim shown in Equation~\eqref{Eq:Appx:LemmaHess} follows from a very similar reasoning. With the only difference that:
\begin{align}
    \norm{H- \nabla^2 f\left(\vx \right)  } &  \leq \eta L \max\left\{\norm{[\nabla^2 f\left(\vx \right)]^{-1}H- I^n} ,\norm{H^{-1}\nabla^2 f\left(\vx \right)- I^n} \right\}. \label{Eq:Appx:LemmaAccHessian}
\end{align}
The maximization term on the right-hand side of Equation~\eqref{Eq:Appx:LemmaAccHessian} can be bound exactly like in the first claim.
\end{proof}
\end{lemma}

\subsection{Inexact Projected Variable-Metric steps} \label{Section:Appx:InexactScaledProjections}

We first begin by showing that if the PVM steps are computed inexactly using the error criterion shown in the SOCGS algorithm (Line~\ref{AlgLine:AccuracyParameter} of Algorithm~\ref{algo:proj-Newton}) they still achieve local quadratic convergence in distance to the optimum.

\begin{lemma} \label{Lemma:Appx:InexactConvergenceOptimums}
Given a $\mu$-strongly convex function $f(\vx)$ and a compact convex set $\mathcal{X}$, if $\tilde{\vx}_{k+1}$ denotes an $\varepsilon_k$-optimal solution to $\tilde{\vx}^*_{k+1} = \argmin_{\vx \in \mathcal{X}} \hat{f}_k\left( \vx \right)$ where $\varepsilon_k =  (lb(\vx_k)/\norm{\nabla f(\vx_k)})^4$ and $lb(\vx_k)$ denotes a lower bound on the primal gap such that $lb(\vx_k) \leq f(\vx_k) - f(\vx^*)$ then:
 \begin{align*}
      \norm{\tilde{\vx}_{k+1} - \tilde{\vx}_{k+1}^*} &\leq  \sqrt{\frac{2\eta_k}{\mu}}  \norm{\vx_k - \vx^*}^{2}.
\end{align*}
    where the parameter $\eta_k = \max\{ \lambda_{\max}( H_k^{-1} \nabla^2 f(\vx_k)),  \lambda_{\max}([\nabla^2 f(\vx_k) ]^{-1}H_k )\} \geq 1$ measures how well $H_k$ approximates $\nabla^2 f(\vx_k)$.
  \begin{proof}
 By the strong convexity of $\hat{f}_{k}$ (as $H_k \in \mathcal{S}^n_{++}$) we have that:
    \begin{align}
         \varepsilon_k & \geq \hat{f}_{k}(\tilde{\vx}_{k+1}) - \hat{f}_{k}(\tilde{\vx}^*_{k+1}) \\
         & \geq  \innp{\nabla \hat{f}_{k}(\tilde{\vx}^*_{k+1}), \tilde{\vx}_{k+1} - \tilde{\vx}^*_{k+1}} + \frac{\lambda_{\min}(H_k)}{2} \norm{\tilde{\vx}_{k+1} - \tilde{\vx}^*_{k+1}}^2  \\
         & \geq  \innp{\nabla \hat{f}_{k}(\tilde{\vx}^*_{k+1}), \tilde{\vx}_{k+1} - \tilde{\vx}^*_{k+1}} + \frac{\mu}{2\eta_k} \norm{\tilde{\vx}_{k+1} - \tilde{\vx}^*_{k+1}}^2  \label{Eq:Appx:BoundEigenval}\\
        & \geq  \frac{\mu}{2\eta_k} \norm{\tilde{\vx}_{k+1} - \tilde{\vx}^*_{k+1}}^2. \label{Eq:Appx:OptCond}
    \end{align}
    The inequality in Equation~\eqref{Eq:Appx:BoundEigenval} follows from Corollary~\ref{corollary:Eigenvalues} and the one in Equation~\eqref{Eq:Appx:OptCond} from the first-order optimality conditions for the scaled projection problem, of which $\tilde{\vx}^*_{k+1}$ is the exact solution. Rearranging the previous expression allows us to conclude that $\norm{\tilde{\vx}_{k+1} - \tilde{\vx}^*_{k+1}} \leq \sqrt{2\eta_k\varepsilon_k/\mu}$. If we plug in the value of $\varepsilon_k$ in the previous bound:
 \begin{align}
      \norm{\tilde{\vx}_{k+1} - \tilde{\vx}^*_{k+1}} &\leq \sqrt{\frac{2\eta_k}{\mu}\varepsilon_k} \\
       & = \sqrt{\frac{2\eta_{k}}{\mu}} \left(\frac{lb(\vx_k)}{\norm{\nabla f(\vx_k)}} \right)^{2} \\
       & \leq \sqrt{\frac{2\eta_k}{\mu}} \left(\frac{f(\vx_k) - f(\vx^*)}{\norm{\nabla f(\vx_k)}} \right)^{2}  \label{Eq:Appx:BoundPrimalGap}\\
       & \leq \sqrt{\frac{2\eta_k}{\mu}}  \left(\frac{\innp{\nabla f(\vx_k), \vx_k - \vx^*}}{\norm{\nabla f(\vx_k)}}\right)^{2}  \label{Eq:Appx:BoundPConvexity} \\
       & \leq \sqrt{\frac{2\eta_k}{\mu}}  \norm{\vx_k - \vx^*}^{2}. \label{Eq:Appx:BoundCauchySchwartz} 
\end{align}
Where the inequality in Equation~\eqref{Eq:Appx:BoundPrimalGap} follows from the fact that $lb \left( \vx_k\right)$ is a lower bound on the primal gap, the one in Equation~\eqref{Eq:Appx:BoundPConvexity} follows from the convexity of $f(\vx)$ and the last inequality, in Equation~\eqref{Eq:Appx:BoundCauchySchwartz}, follows from the Cauchy-Schwarz inequality.
 \end{proof}
\end{lemma}

Using the previous bound along with Corollary~\ref{Corollary:Appx:ConvExact} we can show that the iterates will converge quadratically in distance to the optimum (Lemma~\ref{Lemma:Appx:InexactConvergenceDistance}), despite not solving the problems to optimality.

\begin{lemma}[Quadratic convergence in distance to the optimum of the Inexact Projected-Variable Metric (PMV) steps] \label{Lemma:Appx:InexactConvergenceDistance}
Given a $\mu$-strongly convex and $L$-smooth function $f(\vx)$ with $L_2$-Lipschitz Hessian and a compact convex set $\mathcal{X}$, let $\tilde{x}_{k+1}$ denote an $\varepsilon_k$-optimal solution to $\tilde{\vx}^*_{k+1} = \argmin_{\vx \in \mathcal{X}} \hat{f}_k\left( \vx \right)$ where $\varepsilon_k =  (lb(\vx_k)/\norm{\nabla f(\vx_k)})^4$ and $lb(\vx_k)$ denotes a lower bound on the primal gap such that $lb(\vx_k) \leq f(\vx_k) - f(\vx^*)$, if Assumption~\ref{assumption:accuracyHessian} is satisfied then:
 \begin{align}
      \norm{\tilde{\vx}_{k+1} - \vx^*} &\leq  \frac{\sqrt{\eta_k}}{2\mu}\left(   \sqrt{8\mu}\left(1  + \sqrt{L\omega}\right)  + \sqrt{\eta_k}L_2 \right) \norm{\vx_k - \vx^*}^2. \label{Eq:Appx:DistanceContractionQuadratic}
\end{align}
    where the parameter $\eta_k = \max\{ \lambda_{\max}( H_k^{-1} \nabla^2 f(\vx_k)),  \lambda_{\max}([\nabla^2 f(\vx_k) ]^{-1}H_k )\} \geq 1$ measures how well $H_k$ approximates $\nabla^2 f(\vx_k)$ and $\omega$ is defined in Assumption~\ref{assumption:accuracyHessian}.
\end{lemma}
\begin{proof}
  Using the triangle inequality yields:
 \begin{align*}
     \norm{\tilde{\vx}_{k+1} - \vx^*} & \leq \norm{\tilde{\vx}_{k+1} - \tilde{\vx}^*_{k+1}} + \norm{\tilde{\vx}^*_{k+1} - \vx^*} \\
      & \leq \left( \frac{\eta_k L_2}{2\mu} + \sqrt{\frac{2 L \eta_k\omega}{\mu} } + \sqrt{\frac{2\eta_k}{\mu}} \right) \norm{\vx_{k} - \vx^*}^2 \\
      & = \frac{\sqrt{\eta_k}}{2\mu}\left(   \sqrt{8\mu}\left(1  + \sqrt{L\omega}\right)  + \sqrt{\eta_k}L_2 \right) \norm{\vx_k - \vx^*}^2.
 \end{align*}
 Where the second inequality follows from using the bounds shown in Corollary~\ref{Corollary:Appx:ConvExact} and Lemma~\ref{Lemma:Appx:InexactConvergenceOptimums}.
 \end{proof}

The SOCGS algorithm chooses at each iteration between the ACG step and the Inexact PVM step according to which one provides more progress in primal gap 
(Lines~\ref{algLine:MonotonicityFirst}-\ref{algLine:ACGStep} of Algorithm~\ref{algo:proj-Newton}). Therefore we need to translate the local rate in distance to the optimum of the PVM algorithm in Lemma~\ref{Lemma:Appx:InexactConvergenceDistance} to one in primal gap. It is immediate to see that we can upper bound the right-hand side of Equation~\eqref{Eq:Appx:DistanceContractionQuadratic} using $\mu$-strong convexity, as:
\begin{align*}
    \norm{\vx_k - \vx^*}^2 \leq \frac{2}{\mu} \left( f(\vx_k) - f(\vx^*)\right).
\end{align*}
However, when we try to lower bound the norm that appears on the left-hand side of Equation~\eqref{Eq:Appx:DistanceContractionQuadratic} using $L$-smoothness we arrive at:
\begin{align}
   \sqrt{\frac{2}{L}}\left( f(\tilde{\vx}_{k+1}) - f(\vx^*) - \innp{\nabla f(\vx^*), \tilde{\vx}_{k+1} - \vx^*} \right)^{1/2} \leq \norm{\tilde{\vx}_{k+1} - \vx^*}. \label{Eq:Appx:Smoothness}
\end{align}
The only term preventing us from expressing the left-hand side of Equation~\eqref{Eq:Appx:Smoothness} solely in terms of primal gap values is $- \innp{\nabla f(\vx^*), \tilde{\vx}_{k+1} - \vx^*}$. As by Assumption~\ref{assumption:strictComplementarity} for any $\vx \in \mathcal{F} \left( \vx^*\right)$ we have that $\innp{\nabla f \left( \vx^*\right), \vx - \vx^*} = 0$, if we can show that from some point onward the iterates $\tilde{\vx}_{k+1}$ remain in $\mathcal{F}\left( \vx^*\right)$, we will be able conclude that $\innp{\nabla f(\vx^*), \tilde{\vx}_{k+1} - \vx^*} = 0$.

The main tool that we will use for the analysis is based on the idea that for points $\vx_k$ sufficiently close to $\vx^*$, when we minimize $\hat{f}_k(\vx)$ over $\cx$ using the ACG algorithm, the iterates $\tilde{\vx}_{k+1}$ of the algorithm will reach $\mathcal{F}(\vx^*)$ in a finite number of iterations, remaining in $\mathcal{F}(\vx^*)$ for all subsequent iterations, that is, the ACG algorithm "identifies" the optimal face while computing the Inexact PVM steps. This is a variation of the proof originally presented in \citet{guelat1986some}, which was used to show for the first time that the ACG algorithm asymptotically converges linearly in primal gap when minimizing a strongly convex and smooth function over a polytope. We reproduce the original proof here, as it will be useful in the technical results to come.

\begin{theorem}[Identification of the optimal face]~\cite{guelat1986some}[Theorem 5] \label{Theorem:Appx:ConvOptFace}
Given a strongly convex and smooth function $f(\vx)$ and a polytope $\cx$, if Assumption~\ref{assumption:strictComplementarity} is satisfied, then there is a $r^{\text{ACG}}>0$ such that for $\vx_k^{\text{ACG}} \in \mathcal{B}(\vx^*,r^{\text{ACG}}) \cap \cx$ and $\vx_k^{\text{ACG}} \notin \mathcal{F}(\vx^*)$ then the ACG algorithm (Algorithm~\ref{Algo:Appx:ACGAlg}) with exact line search satisfies that $\abs{\mathcal{S}_{k+1}^{\text{ACG}}}<\abs{\mathcal{S}_{k}^{\text{ACG}}}$ and $\mathcal{S}_{k}^{\text{ACG}}\setminus \mathcal{S}_{k+1}^{\text{ACG}} \notin \mathcal{F}(\vx^*)$. That is, the ACG algorithm performs an away-step that drops a vertex from $\mathcal{S}_{k}^{\text{ACG}}$ that is not a vertex of the optimal face $\mathcal{F}(\vx^*)$. Moreover, there is a $K^{\text{ACG}}\geq 0$ such that for $k \geq K^{\text{ACG}}$ we have that $\vx_k^{\text{ACG}} \in \mathcal{F}(\vx^*)$.
\begin{proof}
The proof starts by showing that there is an index $T\geq 0$ such that for $k\geq T$ all the steps taken by the ACG algorithm will be away-steps that reduce the cardinality of the active set if $\vx_k^{\text{ACG}} \notin\mathcal{F}(\vx^*)$. Let $r_i>0$ and $c>0$ be such that:
\begin{alignat}{2}
  \innp{\mathbf{v}_i - \vx, \nabla f(\vx)} &\geq -\frac{c}{2}  \quad \quad && \text{if } \norm{\vx - \vx^*} \leq r_i \text{ and } \mathbf{v}_i\in\vertex(\mathcal{F}(\vx^*))  \label{Eq:Appx:OptFace}\\
  \innp{\mathbf{v}_i - \vx, \nabla f(\vx)} &\geq c  \quad \quad && \text{if } \norm{\vx - \vx^*} \leq r_i \text{ and } \mathbf{v}_i\in\vertex(\cx) \setminus \vertex(\mathcal{F}(\vx^*)). \label{Eq:Appx:NonOptFace}
\end{alignat}
Taking $r^{\text{ACG}} = \min_{\mathbf{v}_i\in\vertex(\cx)} r_i$, we know by strong convexity that there is an index $T\geq 0$ such that for $k \geq T$ we have that $\vx_k^{\text{ACG}}\in\mathcal{B}(\vx^*, r^{\text{ACG}})\cap \cx$. Furthermore, suppose that $\vx_k^{\text{ACG}}\notin\mathcal{F}(\vx^*)$, then we have that:
\begin{align*}
    \min_{\mathbf{v}_i\in\mathcal{S}_k^{\text{ACG}} \cap \vertex(\cx) \setminus \vertex(\mathcal{F}(\vx^*))} &\innp{\mathbf{v}_i - \vx_k^{\text{ACG}}, \nabla f(\vx_k^{\text{ACG}})} \geq c\\
    &\geq \frac{c}{2} \\
    & \geq \max_{\mathbf{v}_j\in \mathcal{S}_k^{\text{ACG}} \cap \vertex(\mathcal{F}(\vx^*))} \innp{\vx_k^{\text{ACG}} - \mathbf{v}_j , \nabla f(\vx_k^{\text{ACG}})}. %\label{Eq:Appx:ProofAwaySteps}
\end{align*}
Where the left-hand side follows from Equation~\eqref{Eq:Appx:NonOptFace} and the right-hand side from Equation~\eqref{Eq:Appx:OptFace}.  As $\vx_k^{\text{ACG}}\notin\mathcal{F}(\vx^*)$, then $\mathcal{S}_k^{\text{ACG}} \cap \vertex(\cx) \setminus \vertex(\mathcal{F}(\vx^*)) \neq \emptyset$, as the active set $\mathcal{S}_k^{\text{ACG}}$ must include vertices that are not in the optimal face $\mathcal{F}(\vx^*)$ (otherwise we would have $\vx_k^{\text{ACG}}\in\mathcal{F}(\vx^*)$). This means that the ACG algorithm in Line~\ref{algLine:ACG:Appx:ChooseStep} of Algorithm~\ref{algo:Appx:ACGSteps} will choose an away-step with a vertex $\mathbf{v}_i\in \mathcal{S}_k^{\text{ACG}} \cap\vertex(\cx) \setminus \vertex(\mathcal{F}(\vx^*))$, and not a Frank-Wolfe step with a vertex $\mathbf{v}_j\in \vertex(\mathcal{F}(\vx^*))$, for iterations $k\geq T$. We denote the vertex chosen in the away-step by $\mathbf{v} \in \mathcal{S}_k^{\text{ACG}} \cap\vertex(\cx) \setminus \vertex(\mathcal{F}(\vx^*))$, and we remark that $\vd = \vx_k^{\text{ACG}} - \mathbf{v}$ is a descent direction at $\vx_{k}^{\text{ACG}}$, and so the exact line search will output a step size $\gamma_k \in (0, \gamma_{\max}]$. The proof proceeds by showing that we must have that $\gamma_k = \gamma_{\max}$ in Line~\ref{algLine:ACG:Appx:LineSearch} of Algorithm~\ref{algo:Appx:ACGSteps} for iterations $k\geq T$. Using proof by contradiction, we assume that $\gamma_k < \gamma_{\max}$ and we apply the first-order optimality conditions for the exact line search:
\begin{align}
   0 &=  \innp{\vd, \nabla f(\vx_{k+1}^{\text{ACG}})} \\
   & =  \innp{\vx_{k+1}^{\text{ACG}} - \mathbf{v}, \nabla f(\vx_{k+1}^{\text{ACG}})} + \innp{\vx_{k}^{\text{ACG}} - \vx_{k+1}^{\text{ACG}}, \nabla f(\vx_{k+1}^{\text{ACG}})} \\
   & =  \innp{\vx_{k+1}^{\text{ACG}} - \mathbf{v}, \nabla f(\vx_{k+1}^{\text{ACG}})} - \gamma_k\innp{\vd, \nabla f(\vx_{k+1}^{\text{ACG}})} \\
   & =  \innp{\vx_{k+1}^{\text{ACG}} - \mathbf{v}, \nabla f(\vx_{k+1}^{\text{ACG}})} \label{Eq:Appx:GM1} \\
   & <  - c. \label{Eq:Appx:GM2}
\end{align}
Which is the desired contradiction as $c > 0$. The equality in Equation~\eqref{Eq:Appx:GM1} is due to $\innp{\vd, \nabla f(\vx_{k+1}^{\text{ACG}})} = 0$ because of the optimality conditions of the exact line search and the inequality in Equation~\eqref{Eq:Appx:GM2} is due to $\innp{\vx_{k+1}^{\text{ACG}} - \mathbf{v}, \nabla f(\vx_{k+1}^{\text{ACG}})} \leq -c$ as $\mathbf{v}\in\vertex(\cx) \setminus \vertex(\mathcal{F}(\vx^*))$ and $\vx_{k+1}^{\text{ACG}} \in\mathcal{B}(\vx^*, r^{\text{ACG}})$ (thus Equation~\eqref{Eq:Appx:NonOptFace} holds). This proves that we must have $\gamma_k = \gamma_{\max}$ and $\abs{\cs_k^{\text{ACG}}} > \abs{\cs_{k+1}^{\text{ACG}}}$. While $k\geq T$ and $\vx_k^{\text{ACG}} \notin \mathcal{F}(\vx^*)$ the ACG algorithm will drop a vertex $\mathcal{S}_k^{\text{ACG}} \cap \vertex(\cx) \setminus \vertex(\mathcal{F}(\vx^*))$ using an away-step. As $\abs{\mathcal{S}_k^{\text{ACG}}}$ is finite, we will have for some $K^{\text{ACG}}> T$ that $\cs_{K^{\text{ACG}}}^{\text{ACG}}\cap \vertex (\cx)\setminus \vertex(\mathcal{F}(\vx^*)) = \emptyset$, and therefore $\cs_{K^{\text{ACG}}}^{\text{ACG}} \subseteq \vertex(\mathcal{F}(\vx^*))$. This is equivalent to $\vx_{K^{\text{ACG}}}^{\text{ACG}} \in \mathcal{F}(\vx^*)$. Lastly, using Equation~\eqref{Eq:Appx:OptFace} and $\cs_K^{\text{ACG}}\cap \vertex (\cx)\setminus \vertex(\mathcal{F}(\vx^*)) = \emptyset$ we can show that the ACG algorithm will not perform any Frank-Wolfe steps with vertices $\mathbf{v} \in \vertex (\cx)\setminus \vertex(\mathcal{F}(\vx^*))$ for $k \geq K^{\text{ACG}}$, and so $\vx_k \in \mathcal{F}(\vx^*)$.
\end{proof}
\end{theorem}

The consequence of Theorem~\ref{Theorem:Appx:ConvOptFace} is that after a finite number of iterations $K^{\text{ACG}}\geq 0$ the iterates of the ACG algorithm applied to Problem~\eqref{eq:OptProblem} are "stuck" in the face $\mathcal{F}(\vx^*)$, that is,  we have that $\vx_k^{\text{ACG}} \in \mathcal{F}(\vx^*)$ for all $k\geq K^{\text{ACG}}$. The SOCGS algorithm (Algorithm~\ref{algo:proj-Newton}) uses the ACG algorithm to inexactly solve the scaled projection problem of the PVM steps in Lines~\ref{algLine:MonotonicityFirst}-\ref{algLine:ACGStep} of Algorithm~\ref{algo:proj-Newton}. The function being minimized in these steps is not $f(\vx)$, but rather an approximation $\hat{f}_k(\vx)$ that changes at each iteration. However for points sufficiently close to $\vx^*$ we show in Theorem~\ref{Theorem:Appx:ConvOptFace2} that the ACG steps that solve the scaled projection problem of the PVM steps (in Lines~\ref{algLine:PNStep1}-\ref{algLine:PNStep3} of Algorithm~\ref{algo:proj-Newton}) will also get "stuck" to $\mathcal{F}(\vx^*)$, that is, there is a $K\geq 0$ such that we will have that $\tilde{\vx}_{k+1}\in \mathcal{F}(\vx^*)$ for all $k\geq K$.

\begin{theorem}\label{Theorem:Appx:ConvOptFace2}
Let $f(\vx)$ be a strongly convex and smooth function with Lipschitz continuous Hessian and $\cx$ be a polytope such that Assumption~\ref{assumption:strictComplementarity} is satisfied. We denote the quadratic approximation of $f(\vx)$ at $\vx_k$ as $\hat{f}_{k}(\vx) = \innp{\nabla f(\vx_k), \vx_k - \vx} + 1/2\norm{\vx_k - \vx}^2_{H_k}$, where $H_k$ satisfies Assumption~\ref{assumption:accuracyHessian}. Assume that we use the ACG algorithm (Algorithm~\ref{Algo:Appx:ACGAlg}) with exact line search to minimize $\hat{f}_{k}(\vx)$ over $\cx$, and denote the iterate generated by this algorithm at iteration $t$ as $\tilde{\vx}^t_{k+1}$, then there is a $r>0$ such that if $\{\vx_k, \tilde{\vx}^t_{k+1}, \tilde{\vx}^{t+1}_{k+1} \}\subset \mathcal{B}(\vx^*,r) \cap \cx$ and $\tilde{\vx}^t_{k+1} \notin \mathcal{F}(\vx^*)$ then $\abs{\tilde{\mathcal{S}}^{t+1}_{k+1}}<\abs{\tilde{\mathcal{S}}^{t}_{k+1}}$ and $\tilde{\mathcal{S}}^{t}_{k+1} \setminus\tilde{\mathcal{S}}^{t+1}_{k+1} \notin \mathcal{F}(\vx^*)$. That is, at iteration $t$ the ACG algorithm drops a vertex from the active set $\tilde{\mathcal{S}}^{t}_{k+1}$ that is not a vertex of the optimal face $\mathcal{F}(\vx^*)$. 
\begin{proof}
This proof follows relies on the same concepts as the proof in Theorem~\ref{Theorem:Appx:ConvOptFace} from \citet{guelat1986some}. Let $r_i^* > 0$ and $c^* > 0$ be such that:
\begin{alignat}{2}
  \innp{\mathbf{v}_i - \vx, \nabla f(\vx^*) +\nabla^2 f(\vx^*)(\vx - \vx^*) } &\geq -\frac{c}{2}  \quad \quad && \text{if } \norm{\vx - \vx^*} \leq r_i^* \text{ and } \mathbf{v}_i\in\vertex(\mathcal{F}(\vx^*))  \label{Eq:Appx:OptFaceInexact}\\
  \innp{\mathbf{v}_i - \vx, \nabla f(\vx^*) +\nabla^2 f(\vx^*)(\vx - \vx^*) } &\geq c  \quad \quad && \text{if } \norm{\vx - \vx^*} \leq r_i^* \text{ and } \mathbf{v}_i\in\vertex(\cx) \setminus \vertex(\mathcal{F}(\vx^*)). \label{Eq:Appx:NonOptFaceInexact}
\end{alignat}
Where $\nabla f(\vx^*) +\nabla^2 f(\vx^*)(\vx - \vx^*)$ is the gradient of the quadratic approximation at $\vx^*$ using $\nabla^2 f(\vx^*)$ (note that the minimizer of this quadratic approximation is $\vx^*$ and that this approximation is strongly convex and smooth). We have that:
\begin{align}
    \innp{\mathbf{v}_i - \vx, \nabla f(\vx^*) +\nabla^2 f(\vx^*)(\vx - \vx^*) }  = & \innp{\mathbf{v}_i - \vx, \nabla f(\vx_k) + H_k(\vx - \vx_k) } \\
    & + \innp{ \vx - \mathbf{v}_i , \nabla f(\vx_k) - \nabla f(\vx^*) - \nabla^2 f(\vx^*)(\vx_k - \vx^*)  } \label{Eq:Appx:GMCustom1}\\
    & +\innp{ \vx - \mathbf{v}_i, \left(H_k - \nabla^2 f(\vx_k)\right)(\vx - \vx_k) } \label{Eq:Appx:GMCustom2} \\
    & +\innp{ \vx - \mathbf{v}_i, \left(\nabla^2 f(\vx_k) - \nabla^2 f(\vx^*)\right)(\vx - \vx_k) }. \label{Eq:Appx:GMCustom3}
\end{align}
The term shown in Equation~\eqref{Eq:Appx:GMCustom1} can be bounded using the triangle inequality and the fact that the Hessian of $f(\vx)$ is $L_2$-Lipschitz:
\begin{align*}
    \innp{ \vx - \mathbf{v}_i , \nabla f(\vx_k) - \nabla f(\vx^*) - \nabla^2 f(\vx^*)(\vx_k - \vx^*)  }  & \leq  \norm{\mathbf{v}_i - \vx  } \norm{\nabla f(\vx_k) - \nabla f(\vx^*) - \nabla^2 f(\vx^*)(\vx_k - \vx^*)  }  \\
    &  \leq \frac{L_2}{2} \norm{\mathbf{v}_i - \vx  } \norm{\vx_k - \vx^*}^2.
\end{align*}
The term shown in Equation~\eqref{Eq:Appx:GMCustom2}, can be bounded using the triangle inequality and Lemma~\ref{Lemma:Appx:AccuracyHessian}, leading to:
\begin{align}
    \innp{ \vx - \mathbf{v}_i, \left(H_k - \nabla^2 f(\vx_k)\right)(\vx - \vx_k) } & \leq L\eta_k\omega \norm{ \vx - \mathbf{v}_i}\norm{\vx - \vx_k} \norm{\vx_k - \vx^*}^2 \\
    & \leq L\omega (1 + \omega D^2)\norm{ \vx - \mathbf{v}_i}\norm{\vx - \vx_k} \norm{\vx_k - \vx^*}^2,
\end{align}
where $1 + \omega D^2 \geq \eta_k$ for all $k\geq 0$ from Assumption~\ref{assumption:accuracyHessian}. Lastly, the term in Equation~\eqref{Eq:Appx:GMCustom3} can be bounded using the triangle inequality and the $L_2$-Lipschitz continuity of the Hessian, which allows us to write:
\begin{align}
    \innp{ \vx - \mathbf{v}_i, \left(\nabla^2 f(\vx_k) - \nabla^2 f(\vx^*)\right)(\vx - \vx_k) } \leq L_2\norm{ \vx - \mathbf{v}_i}\norm{\vx - \vx_k} \norm{\vx_k - \vx^*}.
\end{align}
Using these bounds we have:
\begin{align}
    \innp{\mathbf{v}_i - \vx, \nabla f(\vx^*) +\nabla^2 f(\vx^*)(\vx - \vx^*) }  \leq & \innp{\mathbf{v}_i - \vx, \nabla f(\vx_k) + H_k(\vx - \vx_k) } \\
    & + \frac{L_2}{2} \norm{\mathbf{v}_i - \vx  } \norm{\vx_k - \vx^*}^2 \\
    & + L \omega (1 + \omega D^2) \norm{\mathbf{v}_i- \vx}\norm{\vx - \vx_k} \norm{\vx_k - \vx^*}^2 \\
    & + L_2\norm{ \mathbf{v}_i - \vx }\norm{\vx - \vx_k} \norm{\vx_k - \vx^*} \\
     \leq & \innp{\mathbf{v}_i - \vx, \nabla \hat{f}_k(\vx)  } \\
    & + \left(3L_2/2 + L\omega D(1 + \omega D^2) \right) D^2 \norm{\vx_k - \vx^*}. \label{Eq:Lemma:GMBound}
\end{align}
Where we note that $\hat{f}_k(\vx) = \innp{\nabla f(\vx_k), \vx - \vx_k} + 1/2\norm{\vx - \vx_k}_{H_k}^2$. Using the bound in Equation~\eqref{Eq:Lemma:GMBound} along with Equations~\eqref{Eq:Appx:OptFaceInexact}-\eqref{Eq:Appx:NonOptFaceInexact}, and setting $C = \left(3L_2/2 + L\omega D(1 + \omega D^2) \right) D^2$ we have:

\begin{alignat}{2}
  \innp{\mathbf{v}_i - \vx,  \nabla \hat{f}_k(\vx)  } &\geq -\frac{c}{2}  - C \norm{\vx_k - \vx^*}  \quad \quad && \text{if } \norm{\vx - \vx^*} \leq r_i^* \text{ and } \mathbf{v}_i\in\vertex(\mathcal{F}(\vx^*))  \label{Eq:Appx:OptFaceInexactMod}\\
  \innp{\mathbf{v}_i - \vx,  \nabla \hat{f}_k(\vx)  } &\geq c  - C \norm{\vx_k - \vx^*}   \quad \quad && \text{if } \norm{\vx - \vx^*} \leq r_i^* \text{ and } \mathbf{v}_i\in\vertex(\cx) \setminus \vertex(\mathcal{F}(\vx^*)). \label{Eq:Appx:NonOptFaceInexactMod}
\end{alignat}
Let $r^* = \min_{\mathbf{v}_i\in\vertex(\cx)} r_i^*$ and $r = \min\left\{r^*,  c/(4C) \right\}$ and assume that $\vx_k\in\mathcal{B}\left(\vx^*, r\right)\cap \cx$ (we know by strong convexity that there is an index $T\geq 0$ such that for $k \geq T$ the iterates $\vx_k$ of the SOCGS algorithm (Algorithm~\ref{algo:proj-Newton}) will be in the aforementioned ball). If $\tilde{\vx}^t_{k+1}\in\mathcal{B}\left(\vx^*, r\right)\cap \cx$ then the bounds in Equations~\eqref{Eq:Appx:OptFaceInexactMod}-\eqref{Eq:Appx:NonOptFaceInexactMod} hold, as $\norm{\tilde{\vx}^t_{k+1} - \vx^*} \leq r^*$, this leads to:
\begin{align}
  \min_{\mathbf{v}_i\in \tilde{\cs}^t_{k+1}\cap \vertex(\cx) \setminus \vertex(\mathcal{F}(\vx^*))}  &\innp{\mathbf{v}_i - \tilde{\vx}^t_{k+1}, \nabla \hat{f}_k(\tilde{\vx}^t_{k+1})  }   \geq c - C \norm{\vx_k - \vx^*} \label{Eq:Appx:proofGMModified1}\\
  & \geq \frac{c}{2} + C \norm{\vx_k - \vx^*}  \label{Eq:Appx:proofGMModified2}\\
  & \geq \max\limits_{\mathbf{v}_i \in \tilde{\cs}_{k+1}^t \cap \vertex(\mathcal{F}(\vx^*))} \innp{\tilde{\vx}^t_{k+1} - \mathbf{v}_i, \nabla \hat{f}_k(\tilde{\vx}^t_{k+1})  }, \label{Eq:Appx:proofGMModified3}
\end{align}
Where the inequality in Equation~\eqref{Eq:Appx:proofGMModified1} follows from Equation~\eqref{Eq:Appx:NonOptFaceInexactMod}, the inequality in Equation~\eqref{Eq:Appx:proofGMModified2} from the fact that $\norm{\vx_k - \vx^*} < r \leq c/(4C)$ and the last inequality from Equation~\eqref{Eq:Appx:OptFaceInexactMod}. Therefore if $\tilde{\vx}^t_{k+1} \notin \mathcal{F}(\vx^*)$ the ACG algorithm will take an away-step with a vertex $\mathbf{v} \in \tilde{\cs}_{k+1}^t\cap \vertex(\cx) \setminus \vertex(\mathcal{F}(\vx^*))$ and direction $\vd = \tilde{\vx}^t_{k+1} - \mathbf{v}$ (where $\tilde{\cs}_{k+1}^t\cap \vertex(\cx) \setminus \vertex(\mathcal{F}(\vx^*)) \neq \emptyset$ as $\tilde{\vx}^t_{k+1} \notin \mathcal{F}(\vx^*)$). Similarly as in the proof of Theorem~\ref{Theorem:Appx:ConvOptFace}, we show that $\gamma_k = \gamma_{\max}$ if $\tilde{\vx}^{t+1}_{k+1} \in\mathcal{B}\left(\vx^*, r\right)\cap \cx$. We use proof by contradiction, and assume that $\gamma_k < \gamma_{\max}$. Using the optimality of the line search:
\begin{align}
    0 & = \innp{\vd, \nabla \hat{f}_k ( \tilde{\vx}^{t+1}_{k+1})} \\
    & = \innp{\tilde{\vx}^{t+1}_{k+1} - \mathbf{v}, \nabla \hat{f}_k ( \tilde{\vx}^{t+1}_{k+1})} + \innp{\tilde{\vx}^{t}_{k+1} - \tilde{\vx}^{t+1}_{k+1}, \nabla \hat{f}_k ( \tilde{\vx}^{t+1}_{k+1})}  \\
    & = \innp{\tilde{\vx}^{t+1}_{k+1} - \mathbf{v}, \nabla \hat{f}_k ( \tilde{\vx}^{t+1}_{k+1})} -\gamma_k \innp{\vd, \nabla \hat{f}_k ( \tilde{\vx}^{t+1}_{k+1})}  \\
    & = \innp{\tilde{\vx}^{t+1}_{k+1} - \mathbf{v}, \nabla \hat{f}_k ( \tilde{\vx}^{t+1}_{k+1})}  \\
    & \leq -c  + C \norm{\vx_k - \vx^*}   \label{Theorem:Appx:Eq1}\\
    & < - \frac{3}{4}c \label{Theorem:Appx:Eq2} \\
    & < 0.
\end{align}
The inequality in Equation~\eqref{Theorem:Appx:Eq1} follows from Equation~\eqref{Eq:Appx:NonOptFaceInexactMod}, as $\norm{\tilde{\vx}^{t+1}_{k+1} - \vx^*} < r \leq r^*$, and the one in Equation~\eqref{Theorem:Appx:Eq2} follows from $\norm{\vx_k - \vx^*} < r \leq c/(4C)$. This is the desired contradiction, and we must therefore have that $\gamma_k = \gamma_{\max}$. This means that $\abs{\tilde{\cs}^t_{k+1}} >\abs{\tilde{\cs}^{t+1}_{k+1}}$ and $\tilde{\mathcal{S}}^{t}_{k+1} \setminus\tilde{\mathcal{S}}^{t+1}_{k+1} \notin \vertex( \mathcal{F}(\vx^*))$, or stated equivalently, the ACG algorithm has dropped one of the vertices in its active set $\tilde{\mathcal{S}}^{t}_{k+1}$ that is not present in $\mathcal{F}(\vx^*)$.
\end{proof}
\end{theorem}

One of the key requirements in Theorem~\ref{Theorem:Appx:ConvOptFace2} is that $\{\vx_k, \tilde{\vx}^t_{k+1}, \tilde{\vx}^{t+1}_{k+1} \}\subset \mathcal{B}(\vx^*,r) \cap \cx$. As the SOCGS algorithm (Algorithm~\ref{algo:proj-Newton}) decreases the primal gap of Problem~\eqref{eq:OptProblem} at least linearly (Theorem~\ref{theorem:GlobalConvergence}), we can guarantee by strong convexity that there is an index $K \geq0$ after which for $k\geq K$ we have that $\vx_k \in \mathcal{B}(\vx^*,r) \cap \cx$. But in order for Theorem~\ref{Theorem:Appx:ConvOptFace2} to apply for all ACG iterations in Line~\ref{algLine:PNStep2}, when computing the Inexact PVM step, we also need to ensure that $\tilde{\vx}^t_{k+1} \in \mathcal{B}(\vx^*,r) \cap \cx$ for all $t\geq 0$. In the next Lemma we show that $\norm{\tilde{\vx}^t_{k+1} - \vx^*} \leq \mathcal{O}( \norm{\vx_k- \vx^*}^{1/2})$, allowing us to claim that for any $r>0$ we can ensure that $\norm{\tilde{\vx}^t_{k+1} - \vx^*} \leq r$ for small enough $\norm{\vx_k - \vx^*}$.

\begin{lemma}
\label{Lemma:Appx:DistanceConv}
Given a $\mu$-strongly convex and $L$-smooth function $f(\vx)$, a polytope $\mathcal{X}$, and a quadratic approximation $\hat{f}_k(\vx)$ that satisfies Assumption~\ref{assumption:accuracyHessian}, let $\tilde{\vx}_{k+1}^t$ denote the iterate obtained after applying $t$ steps of the ACG algorithm (Line~\ref{algLine:PNStep2} of Algorithm~\ref{algo:proj-Newton}) to minimize $\hat{f}_k(\vx)$ over $\cx$, starting from $\tilde{\vx}_{k+1}^0 = \vx_k$, then for any $t\geq 0$:
 \begin{align*}
      \norm{\tilde{\vx}_{k+1}^t - \vx^*}  \leq & \frac{\sqrt{\eta_k}}{2\mu}\left(   \sqrt{8\mu}\left(1  + \sqrt{L\omega}\right)  + \sqrt{\eta_k}L_2 \right) \norm{\vx_k - \vx^*}^2 \\
      & + \frac{\sqrt{\eta_k^{3/2} G}}{\mu}  \left(   \sqrt{8\mu}\left(1  + \sqrt{L\omega}\right)  + \sqrt{\eta_k}L_2 \right)^{1/2} \norm{\vx_k - \vx^*} \\
      & + \sqrt{\frac{2\eta_k G}{\mu}} \norm{\vx_k - \vx^*}^{1/2},
\end{align*}
where $G = \max_{\vx \in \cx} \norm{\nabla f(\vx)}$. And so for small enough $\norm{\vx_k - \vx^*}$ we can ensure that:
\begin{align*}
  \norm{\tilde{\vx}_{k+1}^t - \vx^*} \leq \mathcal{O}(\norm{\vx_k - \vx^*}^{1/2}).
\end{align*}
\begin{proof}
By the triangle inequality we have:
\begin{align}
    \norm{\tilde{\vx}_{k+1}^t - \vx^*} & \leq  \norm{\tilde{\vx}_{k+1}^t - \tilde{\vx}_{k+1}^*} + \norm{\tilde{\vx}_{k+1}^* - \vx^*} . \label{Eq:Appx:BallLemma}
\end{align}
The first term in Equation~\eqref{Eq:Appx:BallLemma} can be bounded as follows:
\begin{align}
    \norm{\tilde{\vx}_{k+1}^t  -  \tilde{\vx}_{k+1}^*} & \leq \sqrt{\frac{2\eta_k}{\mu}} (\hat{f}_k(\tilde{\vx}_{k+1}^t) - \hat{f}_k(\tilde{\vx}_{k+1}^*))^{1/2} \\
    & \leq \sqrt{\frac{2\eta_k}{\mu}} (\hat{f}_k(\tilde{\vx}_{k+1}^0) - \hat{f}_k(\tilde{\vx}_{k+1}^*))^{1/2} \label{Eq:Appx:BallLemma1}\\
    & = \sqrt{\frac{2\eta_k}{\mu}} (\hat{f}_k(\vx_{k}) - \hat{f}_k(\tilde{\vx}_{k+1}^*))^{1/2} \\
    & = \sqrt{\frac{2\eta_k}{\mu}} \left(\innp{-\nabla f(\vx_k), \tilde{\vx}_{k+1}^* - \vx_{k}} - 1/2\norm{\tilde{\vx}_{k+1}^* - \vx_{k}}^2_{H_k}\right)^{1/2} \label{Eq:Appx:BallLemma2_1_1} \\
    & \leq \sqrt{\frac{2\eta_k}{\mu}} \norm{\nabla f(\vx_k)}^{1/2} \norm{\tilde{\vx}_{k+1}^* - \vx_{k}}^{1/2} \label{Eq:Appx:BallLemma2_1} \\
    & \leq \sqrt{\frac{2\eta_k G}{\mu}}  \norm{\tilde{\vx}_{k+1}^* - \vx_{k}}^{1/2} \label{Eq:Appx:BallLemma2} \\
    & \leq \sqrt{\frac{2\eta_k G}{\mu}}  \left(\norm{\tilde{\vx}_{k+1}^* - \vx^*} + \norm{\vx_{k} - \vx^*}\right)^{1/2}.  \label{Eq:Appx:BallLemma3} 
\end{align}
Where Equation~\eqref{Eq:Appx:BallLemma1} follows from the fact that the ACG algorithm decreases the primal gap at each iteration $t$ and Equation~\eqref{Eq:Appx:BallLemma2_1} is obtained by applying the Cauchy-Schwarz inequality to the first term in Equation~\eqref{Eq:Appx:BallLemma2_1_1} and using the fact that $-\norm{\tilde{\vx}_{k+1}^* - \vx_{k}}^2_{H_k} \leq 0$. Moreover, in Equation~\eqref{Eq:Appx:BallLemma2} we have set $G = \max_{\vx \in \cx} \norm{\nabla f(\vx)}$. Note that the $\norm{\tilde{\vx}_{k+1}^* - \vx^*}$ term appearing in Equations~\eqref{Eq:Appx:BallLemma} and \eqref{Eq:Appx:BallLemma3} can be bounded using Corollary~\ref{Corollary:Appx:ConvExact}, which results in $\norm{\tilde{\vx}_{k+1}^* - \vx^*} \leq\mathcal{O}( \norm{\vx_{k} - \vx^*}^2)$. Combining the bound shown in Equation~\eqref{Eq:Appx:BallLemma3} with the bound in Lemma~\ref{Lemma:Appx:InexactConvergenceDistance} allows us to conclude that that:
 \begin{align*}
      \norm{\tilde{\vx}_{k+1}^t - \vx^*}  \leq & \frac{\sqrt{\eta_k}}{2\mu}\left(   \sqrt{8\mu}\left(1  + \sqrt{L\omega}\right)  + \sqrt{\eta_k}L_2 \right) \norm{\vx_k - \vx^*}^2 \\
      & + \frac{\sqrt{\eta_k^{3/2} G}}{\mu}  \left(   \sqrt{8\mu}\left(1  + \sqrt{L\omega}\right)  + \sqrt{\eta_k}L_2 \right)^{1/2} \norm{\vx_k - \vx^*} \\
      & + \sqrt{\frac{2\eta_k G}{\mu}} \norm{\vx_k - \vx^*}^{1/2}.
\end{align*}
\end{proof}
\end{lemma}

With Lemma~\ref{Lemma:Appx:DistanceConv} we can guarantee that for any radius $r>0$, there is a $K\geq 0$ such that $\tilde{\vx}^t_{k+1} \in \mathcal{B}(\vx^*,r) \cap \cx$ for all $k \geq K$ and all $t\geq 0$. With this, we can move on to prove that after a finite number of iterations $K\geq 0$ we can guarantee that $\vx_k \in \mathcal{F}(\vx^*)$ for all $k\geq K$.

\begin{corollary} \label{Corollary:Appx:StuckToFaceRadius}
Given a strongly convex and smooth function $f(\vx)$ with Lipschitz continuous Hessian and a polytope $\cx$, if Assumptions~\ref{assumption:strictComplementarity} and \ref{assumption:accuracyHessian} are satisfied, then there is a $r^{\text{PVM}}>0$ such that if $\vx_k \in \mathcal{B}(\vx^*, r^{\text{PVM}})\cap \cx$ and for any $t\geq 0$ we have that $\tilde{\vx}_{k+1}^t \notin \mathcal{F}(\vx^*)$ then $\abs{\tilde{\mathcal{S}}^{t+1}_{k+1}}<\abs{\tilde{\mathcal{S}}^{t}_{k+1}}$ and $\tilde{\mathcal{S}}^{t}_{k+1} \setminus\tilde{\mathcal{S}}^{t+1}_{k+1} \notin \mathcal{F}(\vx^*)$.
\begin{proof}
Let $r>0$ be the radius in Theorem~\ref{Theorem:Appx:ConvOptFace2} such that if $\{\vx_k, \tilde{\vx}^t_{k+1}, \tilde{\vx}^{t+1}_{k+1} \}\subset \mathcal{B}(\vx^*,r) \cap \cx$ then $\abs{\tilde{\mathcal{S}}^{t+1}_{k+1}}<\abs{\tilde{\mathcal{S}}^{t}_{k+1}}$ and $\tilde{\mathcal{S}}^{t}_{k+1} \setminus\tilde{\mathcal{S}}^{t+1}_{k+1} \notin \mathcal{F}(\vx^*)$. Since we want this to hold for all $t\geq 0$ for a given $\vx_k$, we need to ensure that $\tilde{\vx}^{t}_{k+1} \in \mathcal{B}(\vx^*,r) \cap \cx$ for $t \geq 0$. This can be accomplished with Lemma~\ref{Lemma:Appx:DistanceConv}, which allows us to ensure that there is a $r^{\text{PVM}}>0$ such that for any $\vx_k \in \mathcal{B}(\vx^*, r^{\text{PVM}})\cap \cx$ we have that $\{\vx_k, \tilde{\vx}^t_{k+1}, \tilde{\vx}^{t+1}_{k+1} \}\subset \mathcal{B}(\vx^*,r) \cap \cx$ for all $t\geq 0$.
\end{proof}
\end{corollary}

\begin{corollary} \label{Corollary:Appx:StuckToFace}
Given a strongly convex and smooth function $f(\vx)$ with Lipschitz continuous Hessian and a polytope $\cx$, if Assumptions~\ref{assumption:strictComplementarity} and \ref{assumption:accuracyHessian} are satisfied, then there is a $K>0$ such that for all $k\geq K$ the iterates of the SOCGS algorithm (Algorithm~\ref{algo:proj-Newton}) satisfy that $\vx_k \in \mathcal{F}(\vx^*)$.
\begin{proof}
By Theorem~\ref{Theorem:Appx:ConvOptFace} we know that there is a $K^{\text{ACG}}\geq 0$ such that for $k \geq K^{\text{ACG}}$ we have that $\vx_k^{\text{ACG}} \in \mathcal{F}(\vx^*)$. Moreover, from Corollary~\ref{Corollary:Appx:StuckToFaceRadius} we know that there is a radius $r^{\text{PVM}}>0$ such that if $\vx_k \in \mathcal{B}(\vx^*, r^{\text{PVM}})\cap \cx$ then $\{\vx_k, \tilde{\vx}^t_{k+1}, \tilde{\vx}^{t+1}_{k+1} \}\subset \mathcal{B}(\vx^*,r) \cap \cx$ for all $t \geq 0$, where $r>0$ is the radius in Theorem~\ref{Theorem:Appx:ConvOptFace2}. As the SOCGS algorithm contracts the primal gap at least linearly, there is a $K^{\text{PVM}}\geq 0$ after which we can guarantee that $\vx_k \in \mathcal{B}(\vx^*, r^{\text{PVM}})\cap \cx$ for all $k\geq K^{\text{PVM}}$.

Assume that $K' = \max\{K^{\text{ACG}}, K^{\text{PVM}} \}$ and $\vx_{K'}\notin \mathcal{F}(\vx^*)$. Then for all subsequent iterations $k \geq K'$ we either choose the ACG step (Line~\ref{algLine:ACGStep} in Algorithm~\ref{algo:proj-Newton}) and have that $\vx_{k+1} = \vx_{k+1}^{\text{ACG}}\in \mathcal{F}(\vx^*)$ and the claim is true, or we choose the Inexact PVM step (Line~\ref{algLine:PVMStep} in Algorithm~\ref{algo:proj-Newton}) and have that $\abs{\mathcal{\cs}_{k}}  > \abs{\mathcal{\cs}_{k+1}}$ and $\abs{\mathcal{\cs}_{k}} \setminus \abs{\mathcal{\cs}_{k+1}}\in (\vertex(\cx)\setminus \vertex(\mathcal{F}(\vx^*)))$ by Theorem~\ref{Theorem:Appx:ConvOptFace2}. The latter case can only happen a finite number of times before $\vx_{K}\in \mathcal{F}(\vx^*)$ for some $K > K'$, as $\abs{\mathcal{\cs}_{K'}}$ is finite. Thereafter we will have that $\vx_{k}\in \mathcal{F}(\vx^*)$ for all $k > K$ (as Theorem~\ref{Theorem:Appx:ConvOptFace} and Theorem~\ref{Theorem:Appx:ConvOptFace2} will still hold).
\end{proof}
\end{corollary}

This allows us to conclude in the next theorem that the quadratic convergence in distance to the optimum of the Inexact PVM steps translates into quadratic convergence in the primal gap for the SOCGS algorithm.

\begin{theorem}[Quadratic convergence in primal gap of the SOCGS algorithm] \label{Lemma:Appx:InexactConvergencePrimalGap}
Given a $\mu$-strongly convex and $L$-smooth function $f(\vx)$ with $L_2$-Lipschitz Hessian and a polytope $\mathcal{X}$, if Assumption~\ref{assumption:strictComplementarity} and Assumption~\ref{assumption:accuracyHessian} are satisfied, then there is a $K\geq 0$ such that for $k\geq K$ the iterates of the SOCGS algorithm (Algorithm~\ref{algo:proj-Newton}) satisfy:
  \begin{align*}
  f(\vx_{k+1}) - f(\vx^*) \leq \frac{L \eta_k}{2\mu^4} \left(\sqrt{8\mu} (1+ \sqrt{L\omega}) + \sqrt{\eta_k} L_2\right)^2\left(f(\vx_{k}) - f(\vx^*)\right)^{2}.
 \end{align*}
 where the parameter $\eta_k = \max\{ \lambda_{\max}( H_k^{-1} \nabla^2 f(\vx_k)),  \lambda_{\max}([\nabla^2 f(\vx_k) ]^{-1}H_k )\} \geq 1$ measures how well $H_k$ approximates $\nabla^2 f(\vx_k)$ and $\omega$ is defined in Assumption~\ref{assumption:accuracyHessian}.
\begin{proof}
 From Corollary~\ref{Corollary:Appx:StuckToFace} we know that there is an index $K\geq 0$ such that for $k \geq K$ we know that the Inexact PVM iterates and the ACG iterates will be contained in $\mathcal{F}(\vx^*)$. This allows us to convert the quadratic convergence in distance to the optimum in Lemma~\ref{Lemma:Appx:InexactConvergenceDistance} for the Inexact PVM steps to a quadratic convergence in primal gap. Using strong-convexity we can bound bound $\norm{\vx_k - \vx^*}^2\leq 2/\mu(f(\vx_k) - f(\vx^*))$. Using $L$-smoothness along with the strict-complementary assumption (Assumption~\ref{assumption:strictComplementarity}) and the fact that $\tilde{\vx}_{k+1}\in\mathcal{F}(\vx^*)$ leads to $\norm{\tilde{\vx}_{k+1} - \vx^*}^2\geq 2/L(f(\tilde{\vx}_{k+1}) - f(\vx^*)))$. Plugging these bounds into the convergence in distance to the optimum from Lemma~\ref{Lemma:Appx:InexactConvergenceDistance} results in:
  \begin{align}
  f(\tilde{\vx}_{k+1}) - f(\vx^*) \leq \frac{L \eta_k}{2\mu^4} \left(\sqrt{8\mu} (1+ \sqrt{L\omega}) + \sqrt{\eta_k} L_2\right)^2\left(f(\vx_{k}) - f(\vx^*)\right)^{2}. \label{Eq:Appx:PrimalGapPVM}
 \end{align}
 As the SOCGS contracts the primal gap at least linearly (see Theorem~\ref{theorem:GlobalConvergence}), then for small enough $f(\vx_k) - f(\vx^*)$ with $k\geq K$ we know that the quadratic convergence shown in Equation~\eqref{Eq:Appx:PrimalGapPVM} for the Inexact PVM steps in Line~\ref{algLine:PNStep1}-\ref{algLine:PNStep3} will provide more primal progress than the ACG steps in Line~\ref{algLine:ACG}. Therefore the Inexact PVM steps will be chosen in Line~\ref{algLine:MonotonicityFirst} and we will have that:
   \begin{align*}
  f(\vx_{k+1}) - f(\vx^*) \leq \frac{L \eta_k}{2\mu^4} \left(\sqrt{8\mu} (1+ \sqrt{L\omega}) + \sqrt{\eta_k} L_2\right)^2\left(f(\vx_{k}) - f(\vx^*)\right)^{2}.
 \end{align*}
 \end{proof}
 \end{theorem}

\subsection{Complexity Analysis} \label{Appx:ComplexityAnalysis}

Throughout this section we make the simplifying assumption that we have at our disposal the tightest possible lower bound $lb(\vx_k)$ on the primal gap, that is, $lb(\vx_k) = f(\vx_k) - f(\vx^*)$. Providing a looser lower bound on the primal gap does not affect the number of first-order or Hessian oracle calls, however it can significantly increase the number of linear optimization oracle calls used to compute the Inexact PVM steps. Let $r = \min\left\{r^\text{ACG}, r^{\text{PVM}}\right\}>0$, where $r^\text{ACG}$ is described in Theorem~\ref{Theorem:Appx:ConvOptFace} and $r^{\text{PVM}}$ in Corollary~\ref{Corollary:Appx:StuckToFaceRadius}. Note that $r$ is independent of the target accuracy $\varepsilon$. For ease of exposition we can divide the behaviour of the SOCGS algorithm (Algorithm~\ref{algo:proj-Newton}) into three phases:
\begin{enumerate}
    \item \textbf{Phase 1: $\vx_k \notin \mathcal{B}(\vx^*, r)\cap \cx$ or $\vx_k^{\text{ACG}} \notin \mathcal{B}(\vx^*, r)\cap \cx$.} In this phase the SOCGS algorithm will contract the primal gap at least linearly, as dictated by Theorem~\ref{theorem:GlobalConvergence}. Using strong-convexity we can upper bound the number of iterations needed until $\{\vx_k, \vx_k^{\text{ACG}}\} \in \mathcal{B}(\vx^*, r)$, which marks the end of this first phase.
    
    \item \textbf{Phase 2: $\{\vx_k, \vx_{k}^{\text{ACG}}\} \in \mathcal{B}(\vx^*, r)\cap \cx$ and $\{\vx_k, \vx_{k}^{\text{ACG}}\} \notin \mathcal{F}(\vx^*)$.} The primal gap convergence of the SOCGS algorithm in this phase is also at least linear, and the convergence bound of Theorem~\ref{theorem:GlobalConvergence} still holds. However in this phase, the ACG steps in Line~\ref{algLine:ACG} and the ACG steps used to compute the Inexact PVM iterates in Lines~\ref{algLine:PNStep1}-\ref{algLine:PNStep3} will drop any vertices in their respective active sets that are not in $\mathcal{F}(\vx^*)$. That is, if $\vx_k^{\text{ACG}}\in\mathcal{B}(\vx^*, r)\cap \cx\setminus \mathcal{F}(\vx^*)$ then $\abs{\mathcal{S}_k^{\text{ACG}}} > \abs{\mathcal{S}_{k+1}^{\text{ACG}}}$ and $\mathcal{S}_k^{\text{ACG}} \setminus \mathcal{S}_{k+1}^{\text{ACG}} \notin \vertex(\mathcal{F}(\vx^*))$. Similarly, if $\vx_k\in\mathcal{B}(\vx^*, r)\cap \cx \setminus \mathcal{F}(\vx^*)$ then $\tilde{\vx}_{k+1}$ in Line~\ref{algLine:OutputPVM} in Algorithm~\ref{algo:proj-Newton} satisfies after exiting the while loop in Lines~\ref{algLine:PNStep1}-\ref{algLine:PNStep3} that $\abs{\mathcal{S}_k} > \abs{\tilde{\mathcal{S}}_{k+1}}$ and $\mathcal{S}_k \setminus \tilde{\mathcal{S}}_{k+1} \not\subset \vertex(\mathcal{F}(\vx^*))$. As the cardinality of both active sets is finite, after a finite number of iterations we must have that $\{\vx_k, \vx_{k}^{\text{ACG}}\} \in \mathcal{B}(\vx^*, r)\cap \mathcal{F}(\vx^*)$, which marks the end of this phase.
    
    \item \textbf{Phase 3: $\{\vx_k, \vx_{k}^{\text{ACG}}\} \in \mathcal{B}(\vx^*, r)\cap \mathcal{F}(\vx^*)$.} In this final phase the SOCGS algorithm has a quadratic convergence rate in primal gap, as shown in Theorem~\ref{Lemma:Appx:InexactConvergencePrimalGap}. Once $\{\vx_k, \vx_{k}^{\text{ACG}}\} \in \mathcal{B}(\vx^*, r)\cap \mathcal{F}(\vx^*)$ the ACG steps in Line~\ref{algLine:ACG} and in Lines~\ref{algLine:PNStep1}-\ref{algLine:PNStep3} will not pick up any vertices in $\vertex(\cx) \setminus \vertex(\mathcal{F}(\vx^*))$, and the iterates will remain in $\mathcal{B}(\vx^*, r)\cap \mathcal{F}(\vx^*)$ for all subsequent steps.
    
\end{enumerate}
As in the classical analysis of PVM and Newton algorithms, the SOCGS algorithm shows local quadratic convergence (in primal gap and distance to the optimum) after a number of iterations that is independent of $\varepsilon$ (but dependent on $f(\vx)$ and $\cx$). The SOCGS algorithm makes use of three different types of oracle calls, namely, Hessian, first-order and linear optimization oracle calls. The Hessian oracle is called once per iteration (in Line~\ref{algLine:updateHessian}), while the first-order oracle is called at most twice (to compute the independent ACG step in Line~\ref{algLine:ACG} and to build the quadratic approximation in Line~\ref{algLine:buildQuadratic}). The linear minimization oracle will be called once in Line~\ref{algLine:ACG} for the independent ACG step and potentially multiple times in Line~\ref{algLine:PNStep2} while computing the Inexact PVM step.

In order to study the number of linear optimization oracle calls needed to achieve a $\varepsilon$-optimal solution to Problem~\eqref{eq:OptProblem} we first review the convergence of the Frank-Wolfe gap of the ACG algorithm, which is used as a stopping criterion in the SOCGS algorithm to compute the Inexact PVM steps (Line~\ref{algLine:PNStep1} in Algorithm~\ref{algo:proj-Newton}).

\begin{theorem}[Convergence of the Frank-Wolfe gap of the ACG algorithm]~\citep[Theorem 2]{lacoste2015global} \label{theorem:convergencedual}
Given a $\mu$-strongly convex and $L$-smooth function $f(\vx)$ and a polytope $\mathcal{X}$, then for any $k\geq 0$ the ACG algorithm satisfies:
  \begin{align*}
\max\limits_{\mathbf{v} \in \cx}   \langle \nabla f \left( \vx_{k } \right), \vx_{k} - \mathbf{v} \rangle \leq \begin{cases}
    LD^2/2 + f(\vx_k) - f(\vx^*),& \text{if } f(\vx_k) - f(\vx^*) \geq LD^2/2\\
    D\sqrt{2L(f(\vx_k) - f(\vx^*))},              & \text{otherwise},
\end{cases}
 \end{align*}
 where $D$ denotes the diameter of the polytope $\cx$.
 \end{theorem}

With the previous Theorem at hand we can move on to study the number of oracle calls of each type that we need in the aforementioned phases.

\noindent \textbf{Phase 1: $\vx_k \notin \mathcal{B}(\vx^*, r)\cap \cx$ or $\vx_k^{\text{ACG}} \notin \mathcal{B}(\vx^*, r)\cap \cx$.}

The number of outer iterations needed for $\vx_k^{\text{ACG}}$ and $\vx_k$ to reach $\mathcal{B}(\vx^*, r)\cap \cx$ can be upper bounded using strong convexity. As $f(\vx) - f(\vx^*) \geq \mu/2\norm{\vx - \vx^*}^2$ then if $f(\vx) - f(\vx^*) \leq \mu/2r^2$ we can conclude that $\vx \in \mathcal{B}(\vx^*, r)\cap \cx$. As the iterates $\vx_k$ and $\vx_k^{\text{ACG}}$ have a primal gap convergence that is at least linear (see Theorem~\ref{theorem:GlobalConvergence} and Theorem~\ref{theorem:ConvergenceACG} respectively) then the number of iterations $T_1$ needed to ensure that $\{\vx_k,\vx_k^{\text{ACG}} \}\in\mathcal{B}(\vx^*, r)\cap \cx$ for all $k \geq T_1$ can be upper bounded by:
\begin{align}
    T_1 \leq \frac{8L}{\mu}\left(\frac{D}{\delta} \right)^2 \log \left(\frac{2 (f(\vx_0) - f(\vx^*))}{\mu r^2}\right). \label{Eq:Appx:BoundNumberIteration}
\end{align}
Where we have used the primal gap convergence of Theorem~\ref{theorem:GlobalConvergence} and $\mu$-strong convexity. If we denote by $N_{k,1}$ the number of inner ACG steps in Line~\ref{algLine:PNStep2} that we need to take to satisfy the exit criterion shown in Line~\ref{algLine:PNStep1} of Algorithm~\ref{algo:proj-Newton} at iteration $k$ during this phase, and we use Theorem~\ref{theorem:convergencedual} we have that:
\begin{align}
N_{k, 1}  & \leq \frac{64 \lambda_{\max}(H_k)}{\lambda_{\min}(H_k)}\left( \frac{D}{\delta}\right)^2 \log \left( \frac{\max\left\{(2(\hat{f}_k(\vx_k) -\hat{f}_k(\vx^*_{k+1})))^{1/4}, (2\lambda_{\max}(H_k)D^2(\hat{f}_k(\vx_k) -\hat{f}_k(\vx^*_{k+1})))^{1/8}\right\}}{(f(\vx_k) - f(\vx^*))/\norm{\nabla f(\vx_k)}}\right)  \label{eq:ComplexityBound1}\\
& \leq \frac{64L \eta_k^2}{\mu}\left( \frac{D}{\delta}\right)^2 \log \left( \frac{\max\left\{(2(\hat{f}_k(\vx_k) -\hat{f}_k(\vx^*_{k+1})))^{1/4}, (2L\eta_k D^2(\hat{f}_k(\vx_k) -\hat{f}_k(\vx^*_{k+1})))^{1/8}\right\}}{(f(\vx_k) - f(\vx^*))/\norm{\nabla f(\vx_k)}}\right)  \\
& \leq \frac{64L\eta_k^2}{\mu}\left( \frac{D}{\delta}\right)^2 \log \left( \frac{2\max\left\{(2D\norm{\nabla f(\vx_k)} )^{1/4}, (2L\eta_k D^3\norm{\nabla f(\vx_k)} )^{1/8}\right\}\norm{\nabla f(\vx_k)}}{\mu r^2}\right).\label{eq:ComplexityBound2}
\end{align}

The inequality follows from the fact that for $\vx_k \notin \mathcal{B}(\vx^*, r)\cap \cx$ we can bound $\mu r^2/2  \leq f(\vx_k) - f(\vx^*)$, and the fact that $\hat{f}_k(\vx_k) -\hat{f}_k(\vx^*_{k+1}) = \innp{-\nabla f(\vx_k),\vx_k -  \vx^*_{k+1}} -1/2 \norm{\vx_k -  \vx^*_{k+1}}^2_{H_k} \leq \norm{\nabla f(\vx_k)}\norm{\vx_k -  \vx^*_{k+1}} \leq \norm{\nabla f(\vx_k)} D$. If we denote:
\begin{align*}
    G = \max_{\vx \in \cx} \norm{\nabla f(\vx)} \quad \text{and} \quad \beta = \max\{ (2DG)^{1/4}, (2L (1 + \omega D^2) D^3G)^{1/8},
\end{align*}
then, using the fact that  $\eta_k \leq 1 + \omega D^2$, we can bound the number of inner ACG steps in Line~\ref{algLine:PNStep2} needed for any iteration $k\geq 0$ in the first phase such that $\vx_k \notin \mathcal{B}(\vx^*, r)\cap \cx$ as:
  \begin{align}
N_{k,1}  & \leq \mathcal{O}\left(\frac{L (1 + \omega D^2)^2}{\mu}\left( \frac{D}{\delta}\right)^2 \log \left( \frac{\beta G}{\mu r^2}\right)\right). \label{Eq:Appx:LMOperIteration}
 \end{align}
  As the SOCGS algorithm calls the Hessian oracle once, and the first-order oracle at most twice per iteration we can upper bound the total number of first-order and Hessian oracle calls using the bound shown in Equation~\eqref{Eq:Appx:BoundNumberIteration}. Combining the aforementioned bound with the bound on the total number of linear minimization oracle calls per iteration in Equation~\eqref{Eq:Appx:LMOperIteration} we can bound the total number of linear minimization oracle calls. Therefore in this phase we will need:
  \begin{alignat}{2}
   &\mathcal{O}\left( \frac{8L}{\mu}\left(\frac{D}{\delta} \right)^2 \log \left(\frac{1}{\mu r^2}\right) \right)  \quad \quad &&\text{first-order and Hessian oracle calls.} \\
  &\mathcal{O}\left( \left(\frac{L(1 + \omega D^2)}{\mu}\right)^2\left(\frac{D}{\delta} \right)^4 \log \left(\frac{1}{\mu r^2}\right) \log \left( \frac{\beta G}{\mu r^2}\right) \right)   \quad \quad && \text{Linear minimization oracle calls.}
\end{alignat}

\noindent \textbf{Phase 2: $\{\vx_k, \vx_{k}^{\text{ACG}}\} \in \mathcal{B}(\vx^*, r)\cap \cx$ and $\{\vx_k, \vx_{k}^{\text{ACG}}\} \notin \mathcal{F}(\vx^*)$.}

In this phase we can guarantee that if $\vx_k^{\text{ACG}}\in\mathcal{B}(\vx^*, r)\cap \cx\setminus \mathcal{F}(\vx^*)$ then the ACG step in Line~\ref{algLine:ACG} will be an away-step that reduces the cardinality of the active set $\mathcal{S}_k^{\text{ACG}}$, satisfying that $\abs{\mathcal{S}_k^{\text{ACG}}} > \abs{\mathcal{S}_{k+1}^{\text{ACG}}}$ and $\mathcal{S}_k^{\text{ACG}} \setminus \mathcal{S}_{k+1}^{\text{ACG}} \notin \vertex(\mathcal{F}(\vx^*))$. Similarly, if $\vx_k\in\mathcal{B}(\vx^*, r)\cap \cx \setminus \mathcal{F}(\vx^*)$ then the ACG steps in Line~\ref{algLine:PNStep2} will also be away-steps that reduce the cardinality of the active set $\mathcal{S}_{k}$, that is, after exiting the while loop in Line~\ref{algLine:PNStep3} of Algorithm~\ref{algo:proj-Newton} we have that $\abs{\mathcal{S}_k} > \abs{\tilde{\mathcal{S}}_{k+1}}$ and $\mathcal{S}_k \setminus \tilde{\mathcal{S}}_{k+1} \not\subset \vertex(\mathcal{F}(\vx^*))$. This behaviour will continue until $\vx_{k}^{\text{ACG}} \in \mathcal{F}(\vx^*)$ and $\tilde{\vx}^{t+1}_{k+1} \in \mathcal{F}(\vx^*)$.

Therefore we need to bound the number of vertices that have to be dropped from both $\cs^{\text{ACG}}_k$ and $\cs_k$ in order for $\cs^{\text{ACG}}_k\subseteq \vertex(\mathcal{F}(\vx^*))$ and $\cs_k\subseteq \vertex(\mathcal{F}(\vx^*))$. The ACG algorithm in Line~\ref{algLine:ACG} will have picked up at most $T_1$ vertices in the first phase (as each iteration can only add one vertex to $\cs^{\text{ACG}}$ in Line~\ref{algLine:ACG}), on the other hand, the PVM steps in Lines~\ref{algLine:PNStep1}-\ref{algLine:PNStep3} will have picked up at most $\sum_{k=1}^{T_1}N_{k,1}$ vertices. As once inside the ball all ACG steps (both in Line~\ref{algLine:ACG} and Lines~\ref{algLine:PNStep1}-\ref{algLine:PNStep3}) reduce the cardinality of the active set, and using the bounds in Equation~\eqref{Eq:Appx:BoundNumberIteration} and \eqref{Eq:Appx:LMOperIteration}, we will need:
  \begin{alignat}{2}
  &\mathcal{O}\left( \left(\frac{L(1 + \omega D^2)}{\mu}\right)^2\left(\frac{D}{\delta} \right)^4 \log \left(\frac{1}{\mu r^2}\right) \log \left( \frac{\beta G}{\mu r^2}\right) \right)   \quad \quad && \text{Linear minimization oracle calls.}
\end{alignat}

We now need to bound the number of first-order oracle calls needed to drop the aforementioned vertices. The ACG algorithm in Line~\ref{algLine:ACG} will need to call the first-order oracle at most $T_1$ times. On the other hand, we need to bound the number of vertices that the PVM steps will drop per first-order oracle call in Lines~\ref{algLine:PNStep1}-\ref{algLine:PNStep3}, for which we will use the following Lemma:
\begin{lemma}
If $f(\vx_k) - f(\vx^*) \leq 4\mu^2$ then the Inexact PVM steps in Lines~\ref{algLine:PNStep1}-\ref{algLine:PNStep3} of Algorithm~\ref{algo:proj-Newton} will perform at least one ACG step in Line~\ref{algLine:PNStep2}.
\begin{proof}
 We use proof by contradiction, and we assume that to compute the Inexact PVM step to the necessary accuracy we did not perform any ACG steps in Line~\ref{algLine:PNStep2}, that is:
 \begin{align*}
\left(\frac{f(\vx_k) - f(\vx^*)}{\norm{\nabla f(\vx_k)}}\right)^4 & > \max\limits_{\mathbf{v} \in \cx}   \innp{ \nabla \hat{f}_k \left( \tilde{\vx}^0_{k + 1} \right), \tilde{\vx}^0_{k + 1} - \mathbf{v} } \\
& = \max\limits_{\mathbf{v} \in \cx} \innp{\nabla f (\vx_k), \vx_k - \mathbf{v}}  \\
& \geq \innp{\nabla f(\vx_k), \vx_k - \vx^*} \\ 
& \geq f(\vx_k) - f(\vx^*).
\end{align*}
Where the last inequality follows from convexity. Using the previous chain of inequalities along with $f(\vx_k) - f(\vx^*)\leq \norm{\nabla f(\vx_k)}^2/2\mu$ from $\mu$-strong convexity we have that $f(\vx_k) - f(\vx^*) > 4\mu^2$, which is the desired contradiction. 
\end{proof}
\end{lemma}

We assume that $r< \sqrt{8\mu}$, which allows us to claim that the primal gap for any point $\vx_k \in \mathcal{B}(\vx^*, r)$ satisfies $f(\vx_k) - f(\vx^*) \leq 4\mu^2$ (otherwise it simply takes a constant number of iterations to achieve this once in $\mathcal{B}(\vx^*, r)$, as the primal gap contracts at least linearly). Therefore in this phase we will need:
  \begin{alignat}{2}
  &\mathcal{O}\left( \left(\frac{L(1 + \omega D^2)}{\mu}\right)^2\left(\frac{D}{\delta} \right)^4 \log \left(\frac{1}{\mu r^2}\right) \log \left( \frac{\beta G}{\mu r^2}\right) \right)   \quad \quad && \text{first-order and Hessian oracle calls.}
\end{alignat}

\noindent \textbf{Phase 3: $\{\vx_k, \vx_{k}^{\text{ACG}}\} \in \mathcal{B}(\vx^*, r)\cap \mathcal{F}(\vx^*)$.}

Let $T$ denote the first iteration of the final phase, where $\{\vx_T, \vx_{T}^{\text{ACG}}\} \in \mathcal{B}(\vx^*, r)\cap \mathcal{F}(\vx^*)$ and the quadratic rate dominates over the linear rate. Using the quadratic convergence in primal gap shown in Theorem~\ref{Lemma:Appx:InexactConvergencePrimalGap} we have that:
\begin{align*}
    f(\vx_{k+T+1}) - f(\vx^*) & \leq \left[\frac{L (1 + \omega r^2)}{2\mu^4} \left(8\mu (1+ \sqrt{L\omega}) + \sqrt{(1 + \omega r^2) L_2}\right)^2 \right]^{2^k - 1}\left( f(\vx_T) - f(\vx^*)\right)^{2^k} \\
    & \leq \left[\frac{L (1 + \omega r^2)}{2\mu^4} \left(8\mu (1+ \sqrt{L\omega}) + \sqrt{(1 + \omega r^2) L_2}\right)^2 \left( f(\vx_T) - f(\vx^*)\right) \right]^{2^k}
\end{align*}
Where we have used the fact that by Assumption~\ref{assumption:accuracyHessian} we have that $\eta_k \leq  1 + \omega\norm{\vx_k - \vx^*}^2 \leq 1 + \omega r^2$. Therefore in order to reach a $\varepsilon$-optimal solution starting from this phase we need:
\begin{alignat}{2}
  &\mathcal{O}\left( \log \log \frac{1}{\varepsilon} \right)   \quad \quad && \text{first-order and Hessian oracle calls.}
\end{alignat}
 Where we have only included the dependence on $\varepsilon$ for notational convenience. If we denote by $N_{k,3}$ the number of inner ACG steps in Line~\ref{algLine:PNStep2} that we need to take to satisfy the exit criterion shown in Line~\ref{algLine:PNStep1} of Algorithm~\ref{algo:proj-Newton} at iteration $k$ during this last phase and we use the fact that $f(\vx_k) - f(\vx^*) \geq \varepsilon$ for all suboptimal iterates, resulting in:
  \begin{align}
N_{k,3}  & \leq \mathcal{O}\left(\frac{L\eta_k^2}{\mu}\left( \frac{D}{\delta}\right)^2 \log \left( \frac{\beta G}{\varepsilon}\right)\right).\label{eq:ComplexityBound3}
 \end{align}
Therefore combining the bound on the total number of iterations in this phase with the bound on the number of linear minimization oracle calls per iteration we need:
\begin{alignat}{2}
  &\mathcal{O}\left(\frac{L (1 + \omega D^2)^2}{\mu}\left( \frac{D}{\delta}\right)^2 \log \left( \frac{\beta G}{\varepsilon}\right)  \log \log \frac{1}{\varepsilon} \right)   \quad \quad && \text{Linear minimization oracle calls.}
\end{alignat}

The results for all these phases can be seen in Table~\ref{Table:Appx}.

\begin{table*}[h!]
\footnotesize
 \begin{center} 
  {\renewcommand{\arraystretch}{2.0}
\begin{tabular}{ccc}
\toprule
\textbf{Phase}   & \textbf{FO and Hessian Oracle Calls} & \textbf{LO Oracle Calls} \\ \hline
\midrule
Phase 1  &  $\mathcal{O}\left(  \frac{L}{\mu}\left(\frac{D}{\delta} \right)^2 \log \left(\frac{1}{\mu r^2}\right)\right)$  & $\mathcal{O}\left(  \left(\frac{L(1 + \omega D^2)}{\mu}\right)^2\left(\frac{D}{\delta} \right)^4 \log\left(\frac{1}{\mu r^2}\right)\log \left( \frac{\beta G}{\mu r^2}\right)\right)$  \\ 
Phase 2  & $\mathcal{O}\left(  \left(\frac{L(1 + \omega D^2)}{\mu}\right)^2\left(\frac{D}{\delta} \right)^4 \log\left(\frac{1}{\mu r^2}\right)\log \left( \frac{\beta G}{\mu r^2}\right)\right)$  & $\mathcal{O}\left(  \left(\frac{L(1 + \omega D^2)}{\mu}\right)^2\left(\frac{D}{\delta} \right)^4 \log\left(\frac{1}{\mu r^2}\right)\log \left( \frac{\beta G}{\mu r^2}\right)\right)$  \\ \hline
Phase 3  & $\mathcal{O}\left(\log\log \left(\frac{1}{\varepsilon}\right)\right)$  & $\mathcal{O}\left(\frac{L(1 + \omega D^2)^2}{\mu}\left( \frac{D}{\delta}\right)^2 \log \left( \frac{\beta G}{\varepsilon}\right)\log\log \left(\frac{1}{\varepsilon}\right)\right)$  \\ \hline

\end{tabular}
}
\end{center}
 \caption{Oracle complexity to reach an $\varepsilon$-optimal solution to Problem~\ref{eq:OptProblem} for the SOCGS algorithm (Algorithm~\ref{algo:proj-Newton}).} \label{Table:Appx}
\end{table*}

\newpage

\begin{remark}
The constant $r$ is an invariant of the function and feasible region under consideration and has been used in a similar fashion in \cite{wolfe1970convergence, guelat1986some} and more recently in \cite{garber2020revisiting}, and although unknown, still makes the convergence analysis and complexity estimate conceptually useful, as it adds at most a constant number of iterations independent of $\varepsilon$.
\end{remark}

\begin{remark}
Note that for simplicity we are implicitly assuming in the complexity analysis that the last iterate of the SOCGS algorithm at the end of Phase 2 satisfies $f(\vx_k) - f(\vx^*) \leq [L\eta_k/(2 \mu^4) (\sqrt{8\mu} (1 + \sqrt{L\omega}) + \sqrt{\eta_k}L_2 )]^{-2}$, as otherwise the convergence guarantee in Theorem~\ref{Lemma:Appx:InexactConvergencePrimalGap} does not provide a contraction. If this is not the case at the end of Phase 2, then after an additional finite number of linearly convergent iterations in primal gap, the iterates will indeed satisfy $f(\vx_k) - f(\vx^*) \leq [L\eta_k/(2 \mu^4) (\sqrt{8\mu} (1 + \sqrt{L\omega}) + \sqrt{\eta_k}L_2 )]^{-2}$, after which the complexity analysis from Phase 3 will apply.
\end{remark}

\section{Computational Results} \label{Appx:Computations}

In this section we compare the performance of the SOCGS algorithm with that of other first-order projection-free algorithms for several problems of interest. In the first problem the Hessian oracle will be inexact, but will satisfy Assumption~\ref{assumption:accuracyHessian} with $\omega = 0.1$, moreover we will also assume knowledge of the primal gap, by first computing a solution to high accuracy. In the remaining problems the Hessian oracle will be exact, and we will assume that we do not have knowledge of the primal gap, and will use the strategy outlined in Remark~\ref{primal_gap_not_available}. In the second experiment, in addition to using the exact Hessian, we will also implement SOCGS with an LBFGS  Hessian update (SOCGS LBFGS) (note that this does not satisfy Assumption~\ref{assumption:accuracyHessian}). In the second and third experiment we will also cap the maximum number of inner iterations for the SOCGS and NCG algorithms, as is done in the computational experiments of NCG and SVRCG.

In all three experiments we compare the performance of the SOCGS algorithm with the vanilla Conditional Gradients algorithm (denoted by CG), the Away-Step and Pairwise-Step Conditional Gradients algorithms (ACG and PCG), the \emph{Lazy Away-Step Conditional Gradients algorithm} \cite{braun2017lazifying} (ACG (L)). In the first problem the Hessian oracle will be inexact, but will satisfy Assumption~\ref{assumption:accuracyHessian}. In the remaining problems the Hessian oracle will be exact.

In the first experiment we also compare the performance of the algorithm with the Decomposition Invariant Conditional Gradient (DICG) algorithm \cite{garber2016linear}, as the feasible region is a $0-1$ polytope. 

We also compare against the Conditional Gradient Sliding (CGS) algorithm \cite{lan2016conditional} in the first experiment. This algorithm was also used in the second and third experiment, however the results were not competitive with the ones obtained for the other algorithms, both in terms of iteration count and wall-clock time, and so the CGS results are not included in the images for the second and third experiment.

Additionally, in the first experiment we also compare against the Stochastic Variance-Reduced Conditional Gradients (SVRCG) algorithm \cite{hazan2016variance}, as we can take stochastic first-order oracles of the objective function in question. The third experiment has an objective function that is also amenable to stochastic first-order oracle calls, however the results obtained were not competitive with the other algorithms, both in terms of iteration count and wall-clock time, and so the results for this algorithm were not included in the images for the third experiment.

In the second and third experiments, which use an exact second-order oracle, we also compare the performance against the \emph{Newton Conditional Gradients} (NCG) algorithm in \citet{liu2020newton} which is similar in spirit to the SOCGS algorithm. One of the key features of this algorithm is that it does not require an exact line search strategy, as it provides a specific step size strategy (however it requires selecting five hyperparameters), and it does not require estimating an upper bound on the primal gap.

\begin{remark}[Hyperparameter search for the NCG algorithm]
We tested 27 hyperparameters for the NCG algorithm, and the one that provided the best performance was selected. The parameters used (see \cite{liu2020newton} for their meaning) were combinations of $C_1 \in \{ 0.1, 0,25, 0.4 \}$,  $\delta \in \{ 0.01, 0,5, 0.99 \}$ and $C = \{1.1, 1.5, 2 \}$. The two remaining hyperparemeters were chosen as $\beta = \frac{1}{2}(1 - \frac{1}{2 - 1/C})$ and $\sigma = \frac{1}{C(1 - \beta)} + \frac{\beta}{(1 - 2\beta)(1 - \beta)^2}$ so as to satisfy the requirements in Theorem 4.2 in \cite{liu2020newton}. The hyperparameters that gave the best performance were $\sigma = 0.96$, $\beta = 1/6.0$, $C = 2.0$, $C_1 = 0.25$ and $\delta = 0.99$.
\end{remark}

One of the key challenges that we found when implementing the NCG algorithm is the management of the active set. Starting from a given point $\vx_k$ the algorithm builds a quadratic approximation and performs a series of CG variant steps until the algorithm reaches a certain Frank-Wolfe gap (like in the SOCGS algorithm), which we denote by $\tilde{\vx}_k^{\text{NCG}}$. At that point the algorithm either takes a step with $\gamma_k = 1$ (what is called a full step), or it takes a step size $\gamma_k \neq 1$ (which is called a damped step). In the former case the active set and the barycentric coordinates used for $\vx_{k+1}$ are simply those of $\tilde{\vx}_k^{\text{NCG}}$, which is the point returned by the CG variant steps. In the latter case, however, we set $\vx_{k+1} = \vx_{k} + \gamma_k (\tilde{\vx}_k^{\text{NCG}} -\vx_{k} )$ with $\gamma_k \neq 1$, and we need to combine the active sets and barycentric coordinates of the points $\vx_{k}$ and $\tilde{\vx}_k^{\text{NCG}}$ to form $\vx_{k+1}$. This is a computationally expensive task in general, as the CG variant can drop and pick-up an arbitrary number of vertices going from $\vx_{k}$ to $\tilde{\vx}_k^{\text{NCG}}$, and we need to reconcile the two active sets and barycentric coordinates. This process involves checking if each vertex in the active set of $\tilde{\vx}_k^{\text{NCG}}$ is in the active set of $\vx_{k}$, and vice-versa. When the dimensionality of the problem and the cardinality of the active set is high this can become too costly. That is why in general this algorithm is easiest to implement with CG variants that do not maintain an active set, like the vanilla CG algorithm or the DICG algorithm. We have chosen to use the vanilla CG algorithm in out implementation, as it gave good performance. Note however that there are simple feasible regions where updating the active set and the barycentric coordinates is trivial, like in the probability simplex.

The experiments were run on a laptop with Windows 10, an Intel Core i7 2.4GHz CPU and 6GB RAM.

\subsection{Sparse Coding over the Birkhoff Polytope} \label{Appx:SparseCoding}

Given a set of $m$ input data points $Y = \left[\vy_1, \cdots, \vy_m \right]$ with $\vy_i\in\rr^d$, sparse dictionary learning attempts to find a dictionary $X\in \rr^{d \times n}$ and a sparse representation $Z = \left[\vz_1, \cdots, \vz_m \right]$ with $\vz_i \in \rr^n$ that minimizes:
\begin{align}
    \min\limits_{\substack{X \in \mathcal{C} \\ \vz_i \in \rr^n}} \sum\limits_{i =1}^m \norm{\vy_i - X\vz_i}_2^2 + \lambda \norm{\vz_i}_1. \label{Eq:SparseApprox}
\end{align}
Where $\mathcal{C} = \{ X \in \rr^{d \times n} \mid \sum_{j = 1}^n X_{j,i}^2 \leq 1, \forall i\in[1,d]\}$ is the set of matrices with columns with $\ell_2$ norm less than one. This problem is of interest as many signal processing tasks see performance boosts when given a learned dictionary $X$ that is able to give a sparse representation \cite{mairal2010online}, as opposed to a predefined dictionary obtained from Fourier or wavelet transforms. The elements in this learned dictionary are not required to be orthogonal, and they can form an undercomplete or an overcomplete dictionary.

The problem in Equation~\eqref{Eq:SparseApprox} is convex with respect to $X$ when $Z$ is fixed, and vice-versa, and can be solved by alternating between minimizing with respect to $Z$ with fixed $X$, and minimizing with respect to $X$ with fixed $Z$ \cite{lee2007efficient, mairal2010online}. The latter problem is typically solved with a stochastic projected gradient descent \cite{aharon2006k}. We focus on a variation of the minimization with respect to $X$ with fixed $Z$, more concretely, we also require the rows of $X$ to have norm bounded below $1$, the elements of $X$ be non-negative, and $d = n$. A natural way to impose this is to solve the problem over the Birkhoff polytope. Given a set of vectors $Y = \left\{\vy_1, \cdots, \vy_m \right\}$ and $Z = \left\{\vz_1, \cdots, \vz_m \right\}$, such that $\vy_i, \vz_i \in \rr^n$ for all $i\in \llbracket1, m\rrbracket$, we aim to solve the problem $\min_{X \in \mathcal{X}} f(X)$ where $\mathcal{X}$ is the Birkhoff polytope and $f(X)$ has the form:
\begin{align*}
    f(X) = \sum\limits_{i = 1}^m \norm{\vy_i - X\vz_i}^2,
\end{align*}
The gradient of $f(X)$ amounts to computing $\nabla f(X) = \sum_{i = 1}^m -2(\vy_i - X\vz_i) \vz_i^T$ and the Hessian is given by the block diagonal matrix $\nabla^2 f(X) \in \rr^{n^2 \times n^2}$ with $\nabla^2 f(X) =\diag \left[ B, \cdots, B \right]$ where $B\in \rr^{n\times n}$ has the form $B = \sum_{i = 1}^m \vz_i\vz_i^T$. Therefore $B$ will be positive definite as long as we can form a basis for $\mathbb{R}^n$ with the vectors $\vz_i$, with $m\in [1,m]$. This is verified numerically. As the eigenvalues of a block-diagonal matrix are the eigenvalues of the blocks that form the diagonal, and as we verify that $B$ is positive definite, the function $f(X)$ is $\mu$-strongly convex and $L$-smooth. The complexity of the gradient computation scales as $\mathcal{O}(mn^2)$.

\begin{remark}[On the complexity of linear oracles for the Birkhoff polytope]
Solving an LP exactly over the Birkhoff polytope using the Hungarian algorithm (from combinatorial optimization) has complexity $\mathcal{O}(n^3)$.  Thus it is more expensive to compute the gradient $\nabla f(X)$ than it is to solve an LP over the Birkhoff polytope if $m$ is large.
\end{remark}

\begin{remark}[On the complexity of projection oracles for the Birkhoff polytope]
There are no known algorithms to compute exact projections onto the Birkhoff polytope, and as such projections onto this feasible region have to be computed approximately. For example, if we use an interior-point method to compute a projection onto the Birkhoff polytope, the projection is computed to a certain accuracy (say $\hat{\varepsilon}$), and as such the complexity will depends on a $\log 1/\hat{\varepsilon}$ term. Moreover, to represent the constraints of the Birkhoff polytope we need $n^2$ linear inequality constraints and $2n - 1$ linear equality constraints. We can get rid of the equality constraints by adding $2(2n - 1)$ inequality constraints. We can transform the projection problem with a quadratic objective function and linear inequality constraints into a problem with a linear objective function and quadratic/linear inequality constraints using standard optimization techniques. This means that we have $\mathcal{O}(n^2)$ inequality constraints, and the dimensionality of our problem is $n^2$. If we use a path following interior-point method, and we use the complexity guarantee from Equation 10.12 in \cite{nemirovski2004interior} the resulting complexity to reach an $\hat{\varepsilon}$-optimal solution is $\mathcal{O}(n^7 \log 1/\hat{\varepsilon})$. Note that in the complexity guarantee in the reference, the ambient dimension is $n$, whereas in our case it is $n^2$, and the number of constraints is $n^2$ as opposed to $m$. The cost of these projection oracles justifies the use of conditional gradient algorithms to minimize convex functions over the Birkhoff polytope.
\end{remark}

We generate synthetic data by creating a matrix $B\in \rr^{n \times n}$ with $n = 80$ and entries sampled from a standard normal distribution, and $m$ vectors $\vx\in \rr^n$, with entries sampled from a standard normal distribution, in order to form $Z = \left\{\vz_1, \cdots, \vz_m \right\}$. The set of vectors $Y = \left\{\vy_1, \cdots, \vy_m \right\}$ is generated by computing $\vy_i = B\vz_i$ for all $i\in \llbracket1, m\rrbracket$.

Let us denote the Frobenius norm by $\norm{\cdot}^2_F$, and the uniform distribution between $a$ and $b$ as \textbf{$\mathcal{U}(a,b)$}. In this problem the Hessian oracle will return a matrix $H_k = \nabla^2 f(X_k) + \beta_k \omega \norm{X_k - X^*}_F^2 I^n$, where $\beta_k \in \mathcal{U} (-\lambda_{\max}(\nabla^2 f(X_k))/(\omega\norm{X_k - X^*}_F^2 + 1) , \lambda_{\min}(\nabla^2 f(X_k)))$.

\begin{remark}
  The approximate matrix $H_k = \nabla^2 f(X_k) + \beta_k \omega\norm{X_k - X^*}_F^2 I^n$ with:
  \begin{align}
  \beta_k \in \left[\frac{-\lambda_{\max}\left(\nabla^2 f(X_k)\right)}{\omega \norm{X_k - X^*}_F^2 + 1} ,  \lambda_{\min}\left(\nabla^2 f(X_k)\right) \right],
  \end{align}
  satisfies Assumption~\ref{assumption:accuracyHessian}.
  \begin{proof}
   To see this note that $\eta_k = \max\{ \lambda_{\max}(H_k^{-1}\nabla^2 f(X_k)), \lambda_{\max}([\nabla^2 f(X_k)]^{-1} H_k )\}$ and if we plug in the approximation for the Hessian we have that:
\begin{align}
    \lambda_{\max}([\nabla^2 f(X_k)]^{-1} H_k ) & =\lambda_{\max}([\nabla^2 f(X_k)]^{-1} (\nabla^2 f(X_k) + \beta_k \omega \norm{X_k - X^*}_F^2 I^n) ) \\
    & = 1 + \beta_k \omega \norm{X_k - X^*}_F^2 \lambda_{\max} ([\nabla^2 f(X_k)]^{-1}) \\
    & = 1 + \beta_k\omega \norm{X_k - X^*}_F^2/  \lambda_{\min} (\nabla^2 f(X_k)). \label{Eq:Appx:HessianApprox1}
\end{align}
On the other hand:
\begin{align}
    \lambda_{\max}(H_k^{-1} \nabla^2 f(X_k) ) & = \frac{1}{\lambda_{\min}([\nabla^2 f(X_k)]^{-1} H_k )} \\
    & = \frac{1}{\lambda_{\min}([\nabla^2 f(X_k)]^{-1}) (\nabla^2 f(X_k) + \beta_k \omega \norm{X_k - X^*}_F^2 I^n )} \\
    & = \frac{1}{1 + \beta_k \omega \norm{X_k - X^*}_F^2 \lambda_{\min}([\nabla^2 f(X_k)]^{-1} )} \\
    & = \frac{1}{1 + \beta_k \omega \norm{X_k - X^*}_F^2 / \lambda_{\max}(\nabla^2 f(X_k) )}. \label{Eq:Appx:HessianApprox2}
\end{align}
The conditions on Assumption~\ref{assumption:accuracyHessian} state that $\eta_k \leq 1 + \omega \norm{X_k - X^*}_F^2$. Using Equations~\eqref{Eq:Appx:HessianApprox1} and \eqref{Eq:Appx:HessianApprox2} we can see that the approximate Hessian $H_k = \nabla^2 f(X_k) + \beta_k \omega\norm{X_k - X^*}_F^2 I^n$, with $\beta_k \in [ (-\lambda_{\max}(\nabla^2 f(X_k))/(\omega\norm{X_k - X^*}_F^2 + 1) , \lambda_{\min}(\nabla^2 f(X_k)))]$ satisfies Assumption~\ref{assumption:accuracyHessian}.
  \end{proof}
\end{remark}

The results for $m = 10000$ and $m = 100000$ can be seen in Figure~\ref{fig:Birkhoff1:Appx} and Figure~\ref{fig:Birkhoff2:Appx} respectively. In both cases, the initial point used for all the algorithms is the identity matrix $I^{n\times n}$. We can see that the SOCGS algorithm (with the DICG algorithm as a subproblem solver for the PVM steps) outperforms all the other algorithms being considered for both moderate to high values of $m$. The performance of the SVRCG algorithm improves relative to the other algorithms as we increase the value of $m$, as expected. We use the original implementation of the CGS algorithm for strongly-convex and smooth functions shown in \citet{lan2016conditional}, which uses CG to solve the Euclidean projection subproblems that arise in Nesterov's Accelerated Gradient Descent. The poor performance of the CGS algorithm can be explained with the fact that the CG algorithm does not contract the Frank-Wolfe gap linearly in general, and the accuracy to which the subproblems are solved increases with each iteration, and so at some point the subproblems become very computationally expensive to solve.

\subsection[Structured Logistic Regression]{Structured Logistic Regression over $\ell_1$ unit ball}
Given a binary classification task with $m$ labels $Y = \{\vy_1, \cdots, \vy_m \}$ and $m$ samples $Z = \{\vz_1, \cdots, \vz_m \}$ with $y_i \in \{-1,1 \}$ and $\vz_i \in \rr^n$ for all $i \in [1,m]$, we wish to solve:
\begin{align*}
    \min\limits_{x\in\mathcal{X}}f(x) = \min\limits_{x\in\mathcal{X}} \frac{1}{m} \sum\limits_{i = 1}^{m}\log\left( 1+e^{-y_i \innp{\vx, \vz_i}}\right) + \frac{\lambda}{2}\norm{\vx}^2,
\end{align*}
where $\mathcal{X}$ is the $\ell_1$ unit ball centered at the origin and $\lambda = 1/m$. Although projecting into the $\ell_1$ ball has complexity $\mathcal{O}(n)$ \cite{condat2016fast}, and so projections are cheap, this feasible region is often used to compare the performance of projection-free algorithms between each other (see \citet{lacoste2015global, rao2015forward, braun2019blended}). Solving a linear program over the $\ell_1$ ball also has complexity $\mathcal{O}(n)$. This experiment was also considered in \citet{ghanbari2018proximal} and \citet{scheinberg2016practical} to compare the performance of several Proximal Quasi-Newton methods in the context of minimization with a projection oracle. The gradient of the objective function has the form given by:
\begin{align*}
  \nabla f(\vx) = - \frac{1}{m}\sum_{i = 1}^{m}  \frac{y_i \vz_i}{1 + e^{y_i \innp{\vx, \vz_i}}} + \lambda \vx.
\end{align*}
 The Hessian of the objective function can be written as:
 \begin{align}
     \nabla^2 f(\vx) = \frac{1}{m} \sum_{i = 1}^{m}   \frac{\vz_i \vz_i^T}{(1 + e^{-y_i \innp{\vx, \vz_i}})(1 + e^{y_i \innp{\vx, \vz_i}})} + \lambda I^n. \label{Eq:Appx:HessianLogReg}
 \end{align}
 Note that the $\nabla^2 f(\vx) \in \rr^{n\times n}$ in Equation~\eqref{Eq:Appx:HessianLogReg}, and so for large $n$ even storing the Hessian might become problematic. However, the quadratic approximation does not need to store the matrix, as the function $\hat{f}_k(\vx)$ can be written as:
 \begin{align*}
    \hat{f}_k(\vx) = & - \frac{1}{m}\sum_{i = 1}^{m}  \frac{y_i \innp{\vz_i, \vx - \vx_k}}{1 + e^{-y_i \innp{\vx_k, \vz_i}}} + \lambda \innp{\vx_k , \vx - \vx_k} \\
    & + \frac{1}{2m} \sum_{i = 1}^{m}   \frac{ \innp{\vz_i, \vx - \vx_k}^2}{(1 + e^{-y_i \innp{\vx_k, \vz_i}})(1 + e^{y_i \innp{\vx_k, \vz_i}})} + \frac{\lambda}{2} \norm{\vx - \vx_k}^2 \\
    & =  \innp{\nabla f(\vx_k), \vx - \vx_k} + \frac{1}{2m} \sum_{i = 1}^{m}   \frac{ \innp{\vz_i, \vx - \vx_k}^2}{(1 + e^{-y_i \innp{\vx_k, \vz_i}})(1 + e^{y_i \innp{\vx_k, \vz_i}})} + \frac{\lambda}{2} \norm{\vx - \vx_k}^2.
 \end{align*}
 Which means that the gradient of $\hat{f}_k(\vx)$ is given by:
  \begin{align*}
    \nabla \hat{f}_k(\vx) & = \nabla f(\vx_k) + \frac{1}{m} \sum_{i = 1}^{m}   \frac{ \innp{\vz_i, \vx - \vx_k} \vz_i}{(1 + e^{-y_i \innp{\vx_k, \vz_i}})(1 + e^{y_i \innp{\vx_k, \vz_i}})} + \lambda(\vx - \vx_k).
 \end{align*}
 When computing the Inexact PVM steps we compute $\nabla f(\vx_k)$ and $1/((1 + e^{-y_i \innp{\vx_k, \vz_i}})(1 + e^{y_i \innp{\vx_k, \vz_i}}))$ for each $i\in [1,m]$ at the beginning of the iteration, as these quantities do not change for a fixed $k$. This significantly decreases the time it takes to compute an ACG step with $\nabla \hat{f}_k(\vx)$ in Line~\ref{algLine:Appx:ACG} of Algorithm~\ref{algo:Appx:proj-Newton}, as we only perform operations with transcendental operations once at the beginning of the PVM step. Moreover, as in the previous numerical experiments, we can find a closed-form expression for the line search, that is:
 \begin{align*}
     \argmin\limits_{\gamma\in \rr} \hat{f}_k(\tilde{\vx}^{t}_{k + 1} + \gamma \vd) = -\frac{\innp{\nabla \hat{f}_k(\tilde{\vx}^{t}_{k + 1}), \vd}}{\lambda \norm{\vd}^2 + \frac{1}{m}\sum\limits_{i=1}^m\frac{\innp{\vz_i, d}^2}{(1 + e^{-y_i \innp{\vx_k, \vz_i}})(1 + e^{y_i \innp{\vx_k, \vz_i}})}}.
 \end{align*}
 Where we only need to compute a series of inner products with quantities that in many cases we have already pre-computed in previous operations and stored. This makes line searches with $\hat{f}_k(\vx)$ significantly cheaper than line searches with $f(\vx)$. 

The labels and samples used are taken from the training set of the
\texttt{gissette} \cite{guyon2007competitive} (Figure~\ref{fig:LogReg:Appx}) and the \texttt{real-sim} \cite{chang2011libsvm} (Figure~\ref{fig:LogReg2:Appx}) dataset, where
$n = 5000$ and $m = 6000$ and $n = 72309$ and $m = 20958$, respectively.
Figure~\ref{fig:LogReg} shows the performance of
Algorithm~\ref{algo:proj-Newton} with the Lazy Away-Step Conditional
Gradient algorithm \cite{braun2019blended}. We also limit the maximum number of inner iterations that the SOCGS algorithm and the NCG algorithm perform at each outer iteration to $1000$. In this last example we substituted the step size strategy of the NCG algorithm with a line search, as otherwise we were not getting comparable performance to the other algorithms using the step size strategy defined in \citet{liu2020newton}. We use a golden-section bounded line search for all the line searches for which we cannot find a closed-form solution.

The results for this experiment can be seen in Figure~\ref{fig:LogReg:Appx} and \ref{fig:LogReg2:Appx}. The initial point used for all the algorithms is the vector $\vx_0 = (1,0,\cdots, 0)$. We can see that the SOCGS algorithm (with the AFW algorithm as a subproblem solver for the PVM steps) and the NCG algorithm outperform all the other algorithms, with the SOCGS performing better than the NCG algorithm. The quadratic approximation in this example is easier to evaluate than the original function, as we only need to perform operations with transcendental functions once when we build the approximation, reusing these quantities for all remaining inner iterations. Like in the previous two examples, the SOCGS algorithm and the NCG algorithm benefit from the fact that there is a closed-form solution to the step size at each inner iteration when computing the PVM steps, and so avoid a potentially expensive golden section line search.

\subsection{Inverse covariance estimation over spectrahedron}
\label{sec:inverse-covar-estim}
In many applications the relationships between variables can be modeled with the use of undirected graphical models, such is the case for example in gene expression problems, where the goal is to find out which groups of genes are responsible for producing a certain outcome, given a gene dataset. When the underlying distribution of these variables is Gaussian, the problem of determining the relationship between variables boils down to finding patterns of zeros in the inverse covariance matrix $\Sigma^{-1}$ of the distribution. A common approach to solving this problem relies on finding a $\ell_1$-regularized maximum likelihood estimator of $\Sigma^{-1}$, so as to encourage sparsity, over the positive definite cone \cite{banerjee2008model, friedman2008sparse}, this is often called the Graphical Lasso.

Several optimization algorithms have been used to tackle this problem, such as interior point methods \cite{yuan2007model}, block coordinate descent or accelerated first-order algorithms \cite{banerjee2008model}, coordinate descent algorithms \cite{friedman2008sparse} and even projected limited-memory quasi-Newton algorithms \cite{schmidt2009optimizing}. We solve a variation of the Graphical Lasso problem over the space of positive semidefinite matrices of unit trace, that is:
\begin{gather}
  \min\limits_{\substack{X\succeq 0 \\ \trace\left(X\right) = 1}} -\log \det( X + \delta I^n)+ \trace\left( S X\right) + \frac{\lambda}{2} \norm{X}^2_F. \label{Eq:Appx:InverseCov}
\end{gather}
Where $\delta >0$ is a small constant that we add to make to problem smooth, $S = \sum_{i=1}^N (\vz_i - \mu)(\vz_i - \mu)^T$ is the empirical covariance matrix of a set of datapoints $Z = \{\vz_1, \cdots, \vz_N\}$  drawn from a Gaussian distribution with $\vz_i \in \rr^m$ and $\lambda>0$ is a regularization parameter. This feasible region (known as the spectrahedron) is not a polytope, and so the guarantees shown in the paper do not apply as they crucially rely on Theorem~\ref{theorem:ConvergenceACG}. However, we include the results to show the promising numerical performance of the method. Evaluating $f(X)$ has complexity $\mathcal{O}(n^3)$ if we compute the determinant with a LU decomposition, and evaluating the gradient  $\nabla f(X) = -(X + \delta I^n)^{-1} + S +\lambda X$ has complexity $\mathcal{O}(n^3)$, dominated by the matrix inversion. Solving the linear program $\min_{Y\in\mathcal{X}} \sum_{i,j = 1}^n (\nabla f(X) \otimes Y)_{i,j}$, where $\otimes$ denotes the Hadamard product, amounts to finding the largest eigenvector of $-\nabla f(X)$. We do this approximately by using the Implicitly Restarted Lanczos algorithm \cite{lehoucq1998arpack} (implemented in \texttt{eigsh} in the \texttt{scipy.sparse.linalg} library).

The quadratic approximation $\hat{f}_k(X)$ of $f(X)$ that the PVM steps in Line~\ref{algLine:Appx:PNStep2} of Algorithm~\ref{algo:Appx:proj-Newton} uses can be written as:
\begin{align}
    \hat{f}_k(X) = & \trace \left( \left( -(X_k + \delta I^n)^{-1}  + S + \lambda X_k \right) (X - X_k)  \right) \\
    & + \frac{1}{2}\norm{(X_k + \delta I^n)^{-1} (X - X_k)}_F^2 + \frac{\lambda}{2} \norm{X - X_k}_F^2 \\
    = & \trace \left( \left( -(X_k + \delta I^n)^{-1}  + S + \lambda X_k \right) (X - X_k)  \right) \\
    & + \frac{1}{2}\trace \left((X - X_k)^T \left( (X_k + \delta I^n)^{-T}(X_k + \delta I^n)^{-1} + \lambda I^n \right)(X - X_k) \right). \label{eq:used_in_linesearch}
\end{align}
This allows us to write the gradient $\nabla \hat{f}_k(X)$ of the quadratic approximation as:
\begin{align*}
    \nabla \hat{f}_k(X) = &  \nabla f(X_k) + (X_k + \delta I^n)^{-1}(X- X_k)(X_k + \delta I^n)^{-1} + \lambda (X - X_k).
\end{align*}
The complexity of evaluating the gradient of $\hat{f}_k(X)$ is also $\mathcal{O}(n^3)$, dominated by the matrix inversion and the matrix multiplication operations. In practice, we only invert the matrix $(X_k + \delta I^n)^{-1}$ once per iteration when we form the quadratic approximation in Line~\ref{algLine:Appx:buildQuadratic} of Algorithm~\ref{algo:Appx:proj-Newton}. Nevertheless, this means that the complexity of computing $\nabla f(X)$ and $\nabla\hat{f}_k (X)$ is the same, so in this respect there is no advantage to using the quadratic approximation. However, for the quadratic approximation $\hat{f}_k(X)$ we can find a closed-form expression for the optimal step size when moving along a direction $D$. It suffices to take the derivative of  $\hat{f}_k(X + \gamma D)$ with respect to $\gamma$ using the expression shown in Equation~\eqref{eq:used_in_linesearch} and set the derivative to zero. This leads to:
\begin{align}
  \argmin_{\gamma \in \rr} \hat{f}_k(\tilde{X}^t_k + \gamma D) = -\frac{\trace \left( \nabla \hat{f}_k(\tilde{X}^t_k)D\right)}{\lambda \norm{D}_F^2 + \norm{(X_k + \delta I^n)^{-1}D}_F^2} \label{eq:Appx:LineSearchGLasso}
\end{align}
If we use a golden section search to perform a line search over the original function $f(X)$ to compute the optimal step size we will potentially need to evaluate $f(X)$ multiple times, and each evaluation has complexity $\mathcal{O}(n^3)$. On the other hand, to compute the exact line search for $\hat{f}_k(X)$ we only need to evaluate the expression in Equation~\eqref{eq:Appx:LineSearchGLasso} once, with complexity $\mathcal{O}(n^3)$. This makes the line search operation with $\hat{f}_k(X)$ significantly cheaper than the line search with $f(X)$, and makes the ACG iterations in Line~\ref{algLine:Appx:PNStep2} of Algorithm~\ref{algo:Appx:proj-Newton} significantly cheaper than the iterations in Line~\ref{algLine:Appx:ACGStep} of Algorithm~\ref{algo:Appx:proj-Newton}.

The matrix $S$ is generated by computing a random orthonormal basis $\mathcal{B} = \{ \mathbf{v}_1, \cdots,  \mathbf{v}_m \}$ in $\rr^m$ and computing $S = \sum_{i=1} \sigma_i \mathbf{v}_1\mathbf{v}_1^T$, where $\sigma_i$ is uniformly distributed between $0.5$ and $1$ for $i\in [1,m]$. We use $\lambda = 0.05$ and $\delta = 10^{-5}$ in the experiments. We also limit the maximum number of inner iterations that the SOCGS algorithm and the NCG algorithm perform at each outer iteration to $1000$. We use a golden-section bounded line search for all the line searches for which we cannot find a closed-form solution.

We also implemented an LBFGS algorithm to build an approximate Hessian from first order information from previous iterations. This is specially useful if we cannot find an analytical expression to the exact Hessian, or its matrix-vector products. Note however that the matrix outputted by the LBFGS algorithm does not satisfy Assumption~\ref{assumption:accuracyHessian}, and so the best we can hope for is for the linear-quadratic convergence in primal gap of the SOCGS algorithm. The implementation used stores the Hessian approximation in outer-product form, and so does not explicitly store the full Hessian matrix, as that could be computationally prohibitive (see Section 7.2 in \citet{nocedal2006numerical}).

The results for this experiment can be seen in Figures~\ref{fig:GLasso:Appx} and \ref{fig:GLasso2:Appx}. The initial point for all the algorithms is the matrix $1/n I^n$. We can see that the SOCGS (with the PCG algorithm as a subproblem solver for the PVM steps) and the NCG algorithm outperform all the other algorithms, with the SOCGS performing better than the NCG algorithm. Note that the in this case the main advantage that the SOCGS and the NCG algorithms have over all the other algorithms is the fact that there is a closed-form solution to the step size at each inner iteration when computing the PVM steps. As discussed earlier, the complexity of evaluating the original function $f(X)$ is the same as that of evaluating $\hat{f}_k(X)$. The SOCGS algorithm that uses the LBFGS algorithm to build up an approximate Hessian also performs well in terms of iterations and in terms of time, despite Assumption~\ref{assumption:accuracyHessian} not holding in this case.

\newpage

\begin{figure*}[th!]
    \centering
    \vspace{-10pt}
    \hspace{\fill}
    \subfloat[Iterations]{{\includegraphics[width=6.8cm]{Images/New_Images/Birkhoff_PG_iteration_v2.pdf} }\label{fig:BirkhoffPGIt1:Appx}}%
    %\qquad
    \hspace{\fill}
    \subfloat[Seconds]{{\includegraphics[width=6.8cm]{Images/New_Images/Birkhoff_PG_time_v2.pdf} }\label{fig:BirkhoffPGTime1:Appx}}%
    \hspace{\fill}
    
    \bigskip 
    
    \hspace{\fill}
    \subfloat[Iterations]{{\includegraphics[width=6.8cm]{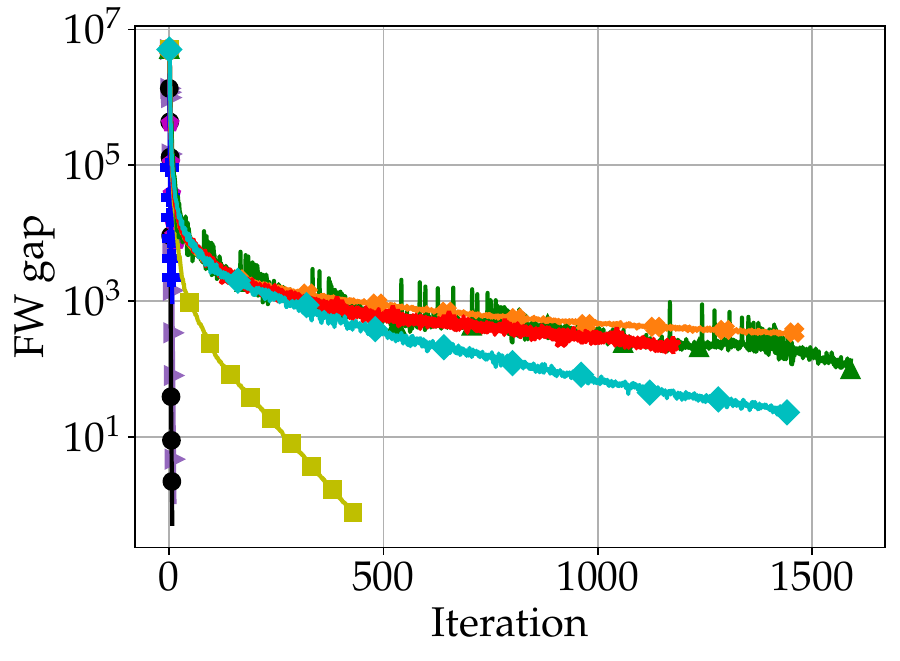} }\label{fig:BirkhoffDGIt1:Appx}}%
    %\qquad
    \hspace{\fill}
    \subfloat[Seconds]{{\includegraphics[width=6.8cm]{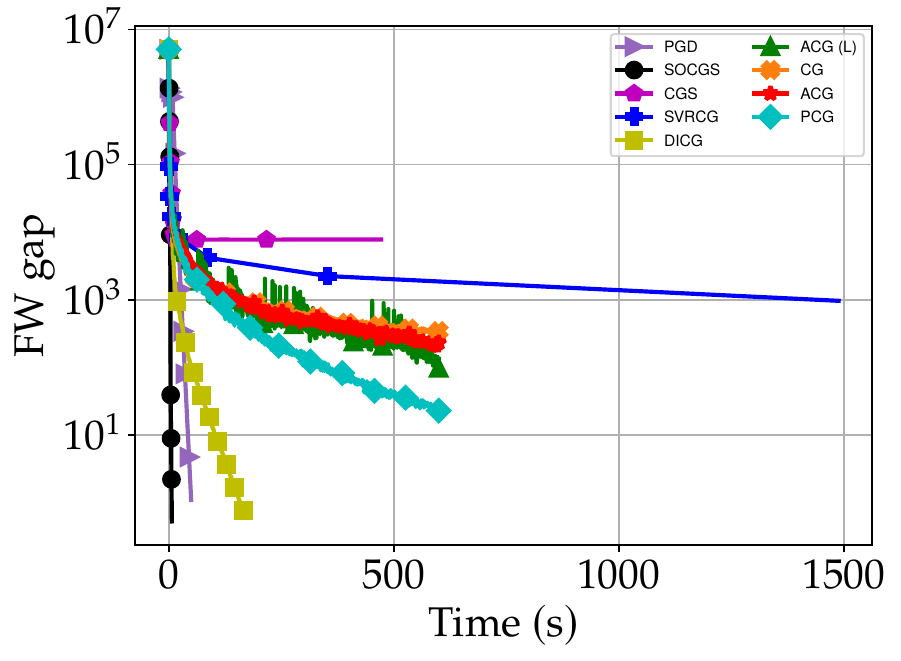} }\label{fig:BirkhoffDGTime1:Appx}}%
    \hspace{\fill}
    
    \bigskip
    
    \hspace{\fill}
    \subfloat[Iterations]{{\includegraphics[width=6.8cm]{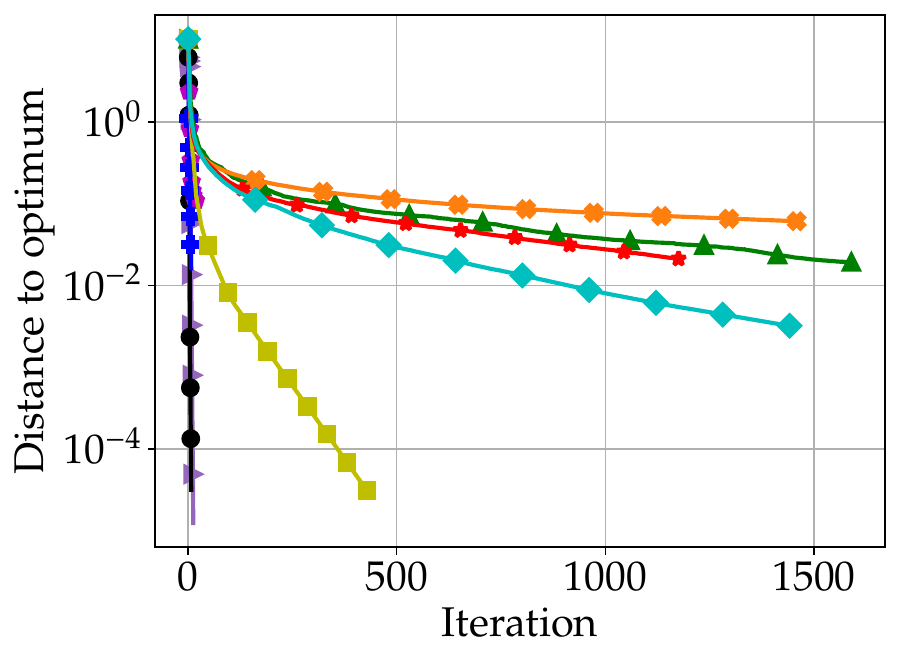} }\label{fig:BirkhoffDistanceIt1:Appx}}%
    %\qquad
    \hspace{\fill}
    \subfloat[Seconds]{{\includegraphics[width=6.8cm]{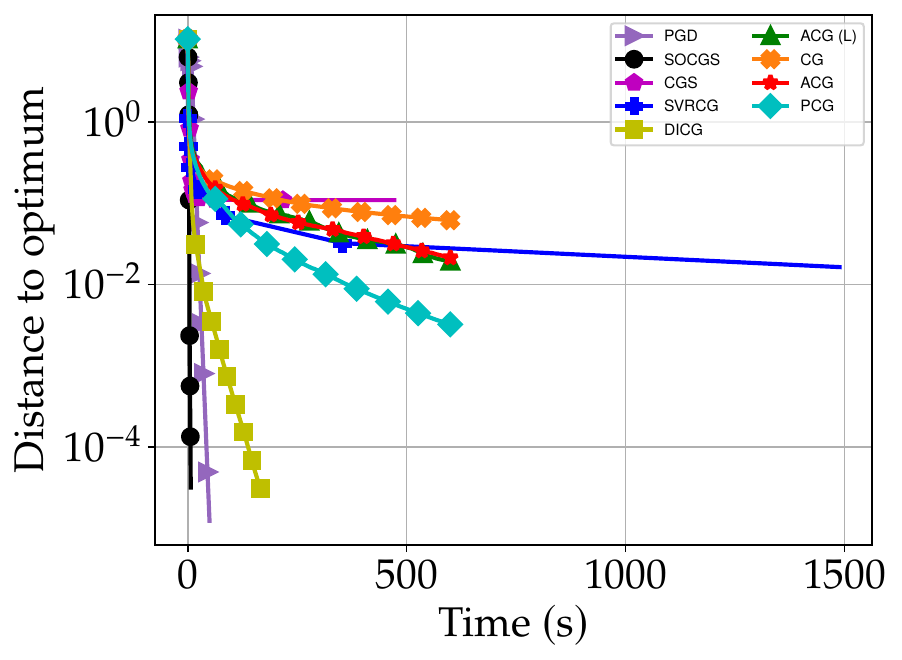} }\label{fig:BirkhoffDistanceTime1:Appx}}%
    \hspace{\fill}
    \caption{\textbf{Sparse Coding over the Birkhoff polytope: } Algorithm comparison for $m = 10,000$ (medium size) samples in terms of primal gap \protect\subref{fig:BirkhoffPGIt1:Appx},\protect\subref{fig:BirkhoffPGTime1:Appx}, Frank-Wolfe gap \protect\subref{fig:BirkhoffDGIt1:Appx},\protect\subref{fig:BirkhoffDGTime1:Appx} and distance to the optimum \protect\subref{fig:BirkhoffDistanceIt1:Appx},\protect\subref{fig:BirkhoffDistanceTime1:Appx}.}%
    \label{fig:Birkhoff1:Appx}%
\end{figure*}

\newpage

\begin{figure*}[th!]
    \centering
    \vspace{-10pt}
    \hspace{\fill}
    \subfloat[Iterations]{{\includegraphics[width=6.8cm]{Images/New_Images/Birkhoff_PG_iteration.pdf} }\label{fig:BirkhoffPGIt2:Appx}}%
    %\qquad
    \hspace{\fill}
    \subfloat[Seconds]{{\includegraphics[width=6.8cm]{Images/New_Images/Birkhoff_PG_time.pdf} }\label{fig:BirkhoffPGTime2:Appx}}%
    \hspace{\fill}
    
    \bigskip 
    
    \hspace{\fill}
    \subfloat[Iterations]{{\includegraphics[width=6.8cm]{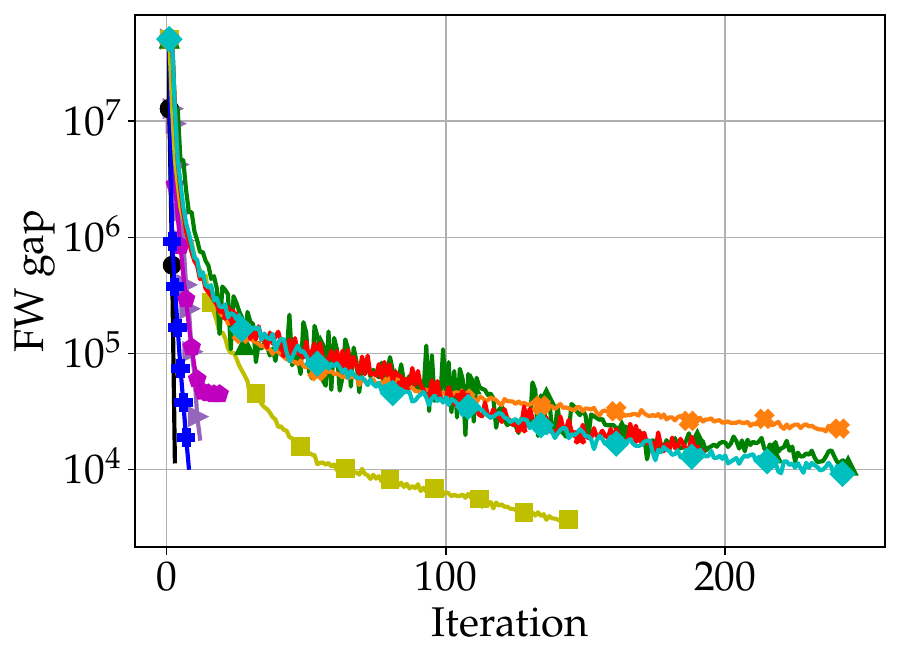} }\label{fig:BirkhoffDGIt2:Appx}}%
    %\qquad
    \hspace{\fill}
    \subfloat[Seconds]{{\includegraphics[width=6.8cm]{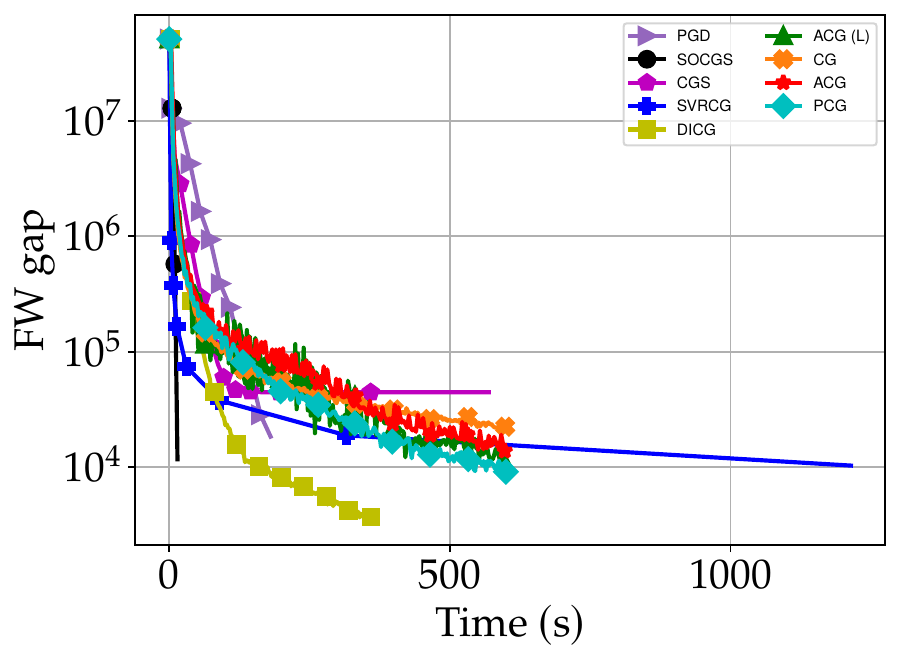} }\label{fig:BirkhoffDGTime2:Appx}}%
    \hspace{\fill}
    
    \bigskip
    
    \hspace{\fill}
    \subfloat[Iterations]{{\includegraphics[width=6.8cm]{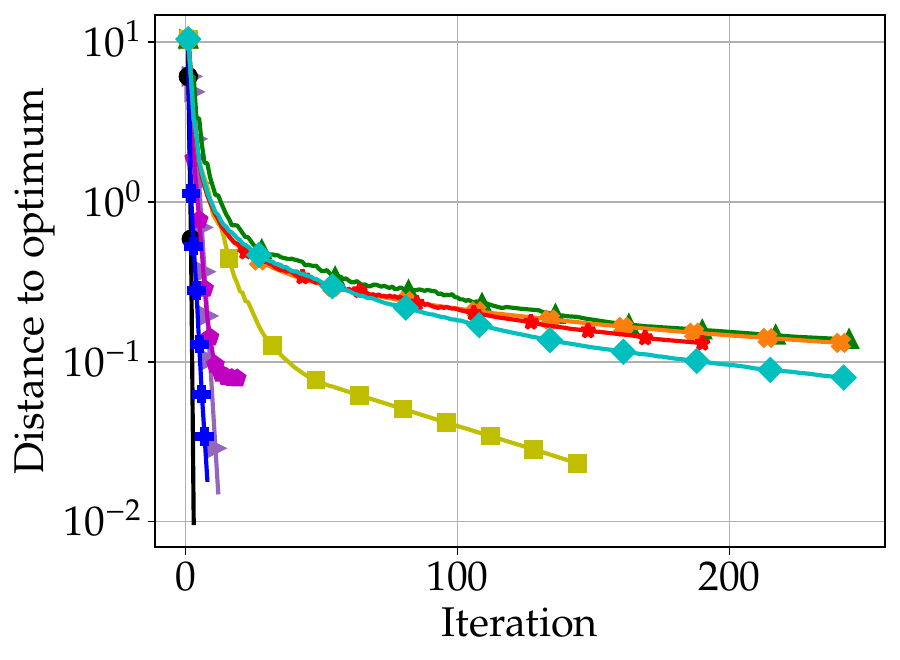} }\label{fig:BirkhoffDistanceIt2:Appx}}%
    %\qquad
    \hspace{\fill}
    \subfloat[Seconds]{{\includegraphics[width=6.8cm]{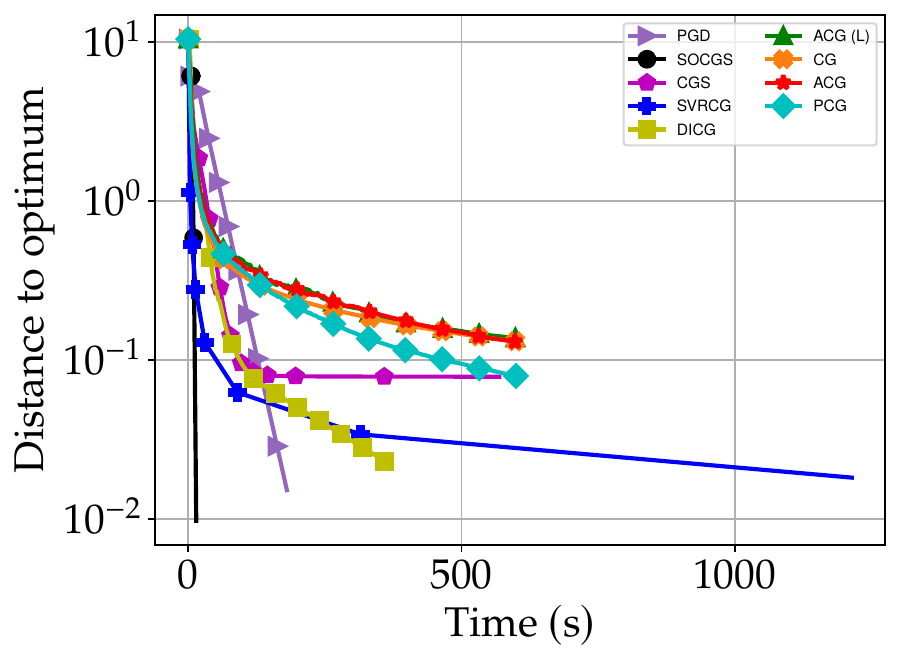} }\label{fig:BirkhoffDistanceTime2:Appx}}%
    \hspace{\fill}
    \caption{\textbf{Sparse Coding over the Birkhoff polytope: } Algorithm comparison for $m = 100,000$ (large size) samples in terms of primal gap \protect\subref{fig:BirkhoffPGIt2:Appx},\protect\subref{fig:BirkhoffPGTime2:Appx}, Frank-Wolfe gap \protect\subref{fig:BirkhoffDGIt2:Appx},\protect\subref{fig:BirkhoffDGTime2:Appx} and distance to the optimum \protect\subref{fig:BirkhoffDistanceIt2:Appx},\protect\subref{fig:BirkhoffDistanceTime2:Appx}.}%
    \label{fig:Birkhoff2:Appx}%
\end{figure*}

\newpage

\begin{figure*}[th!]
    \centering
    \vspace{-10pt}
    \hspace{\fill}
    \subfloat[Iterations]{{\includegraphics[width=6.8cm]{Images/New_Images/L1_PG_iteration_v2.pdf} }\label{fig:LogRegPGIt:Appx}}%
    %\qquad
    \hspace{\fill}
    \subfloat[Seconds]{{\includegraphics[width=6.8cm]{Images/New_Images/L1_PG_time_v2.pdf} }\label{fig:LogRegPGTime:Appx}}%
    \hspace{\fill}
    
    \bigskip 
    
    \hspace{\fill}
    \subfloat[Iterations]{{\includegraphics[width=6.8cm]{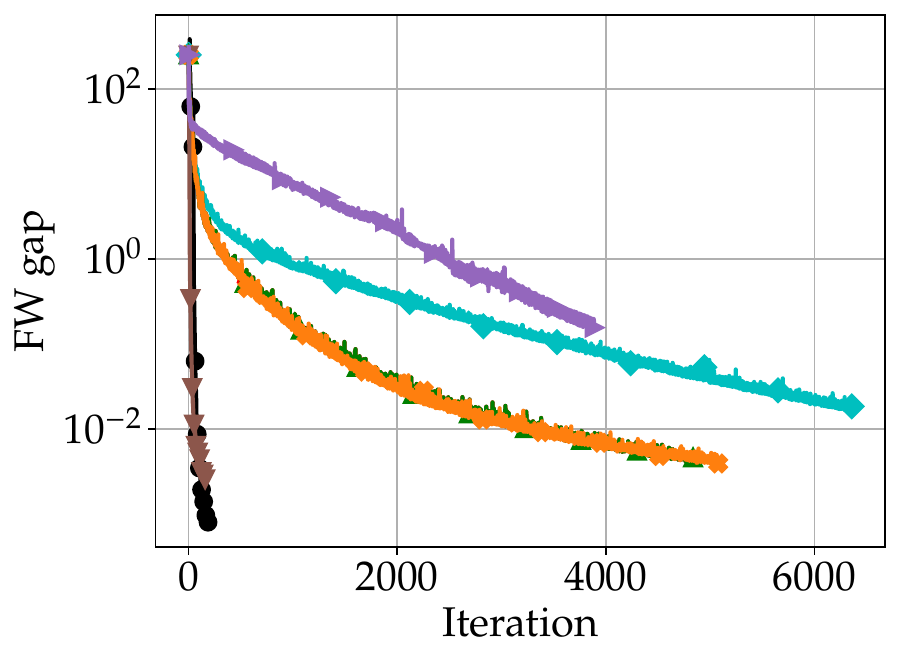} }\label{fig:LogRegDGIt:Appx}}%
    %\qquad
    \hspace{\fill}
    \subfloat[Seconds]{{\includegraphics[width=6.8cm]{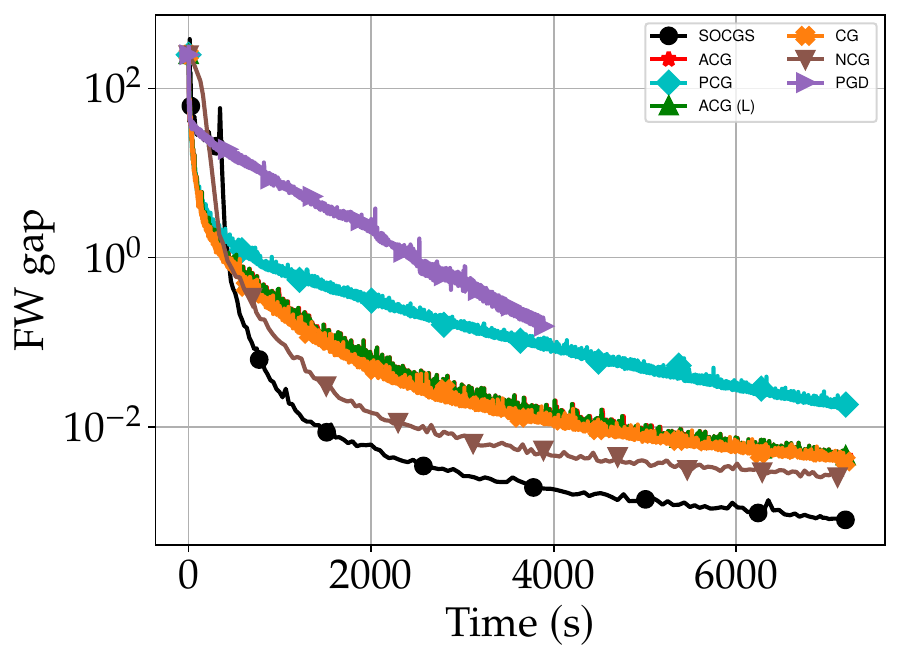} }\label{fig:LogRegDGTime:Appx}}%
    \hspace{\fill}
    
    \bigskip
    
    \hspace{\fill}
    \subfloat[Iterations]{{\includegraphics[width=6.8cm]{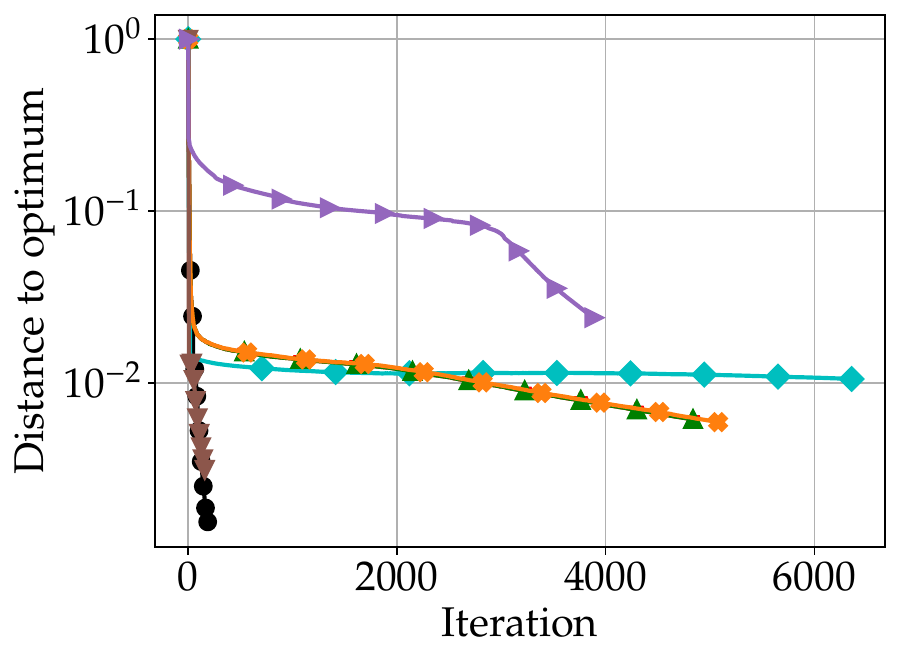} }\label{fig:LogRegDistanceIt:Appx}}%
    %\qquad
    \hspace{\fill}
    \subfloat[Seconds]{{\includegraphics[width=6.8cm]{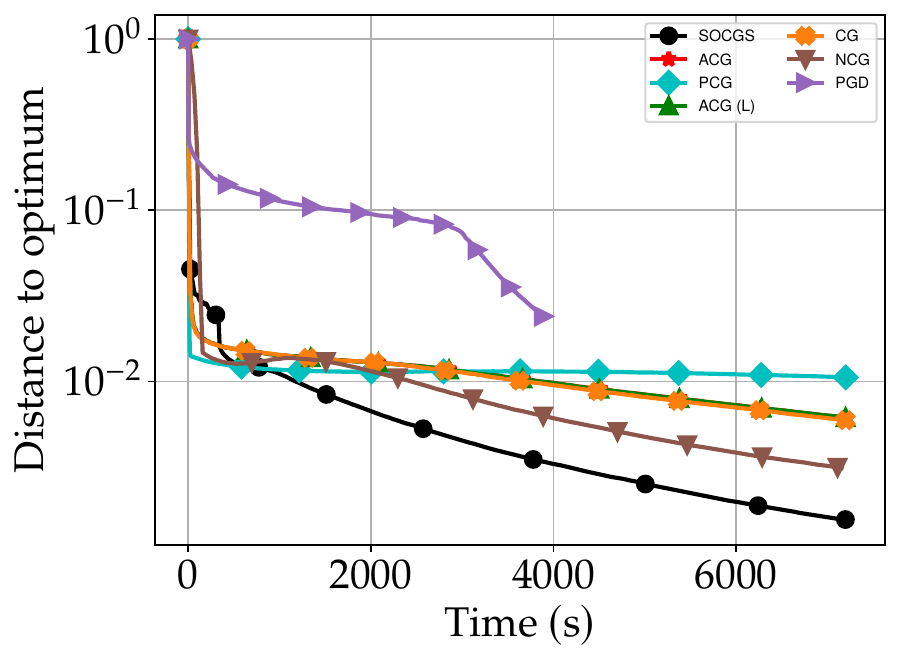} }\label{fig:LogRegDistanceTime:Appx}}%
    \hspace{\fill}
    \caption{\textbf{Structured Logistic Regression over $\ell_1$ unit ball: } Algorithm comparison in terms of primal gap \protect\subref{fig:LogRegPGIt:Appx},\protect\subref{fig:LogRegPGTime:Appx}, Frank-Wolfe gap \protect\subref{fig:LogRegDGIt:Appx},\protect\subref{fig:LogRegDGTime:Appx} and distance to the optimum \protect\subref{fig:LogRegDistanceIt:Appx},\protect\subref{fig:LogRegDistanceTime:Appx} for the \texttt{gissette} \cite{guyon2007competitive} dataset, where
$n = 5000$ and $m = 6000$.}%
    \label{fig:LogReg:Appx}%
\end{figure*}

\newpage

\begin{figure*}[th!]
    \centering
    \vspace{-10pt}
    \hspace{\fill}
    \subfloat[Iterations]{{\includegraphics[width=6.8cm]{Images/New_Images/L1_PG_iteration.pdf} }\label{fig:LogRegPGIt2:Appx}}%
    %\qquad
    \hspace{\fill}
    \subfloat[Seconds]{{\includegraphics[width=6.8cm]{Images/New_Images/L1_PG_time.pdf} }\label{fig:LogRegPGTime2:Appx}}%
    \hspace{\fill}
    
    \bigskip 
    
    \hspace{\fill}
    \subfloat[Iterations]{{\includegraphics[width=6.8cm]{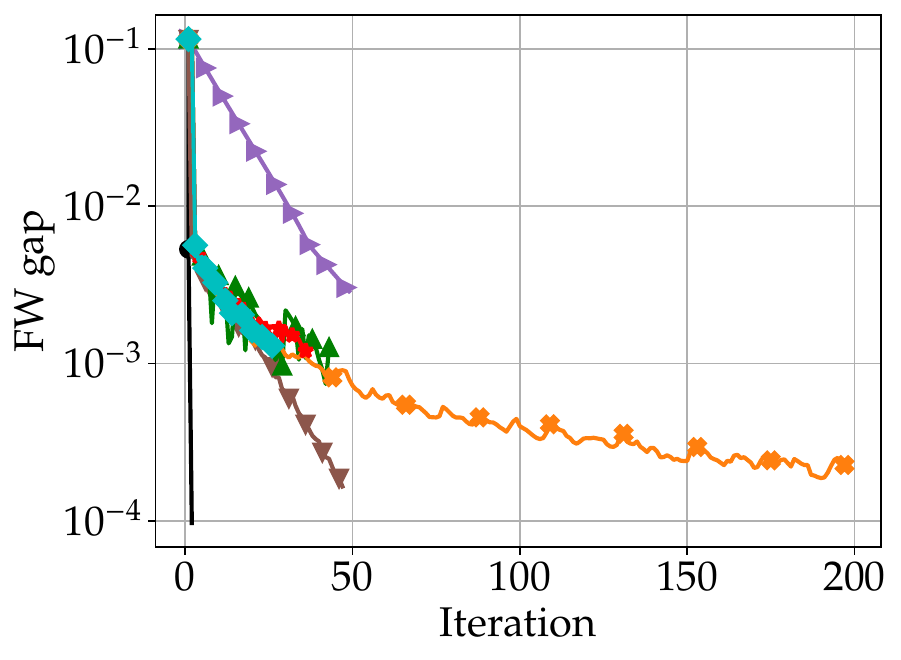} }\label{fig:LogRegDGIt2:Appx}}%
    %\qquad
    \hspace{\fill}
    \subfloat[Seconds]{{\includegraphics[width=6.8cm]{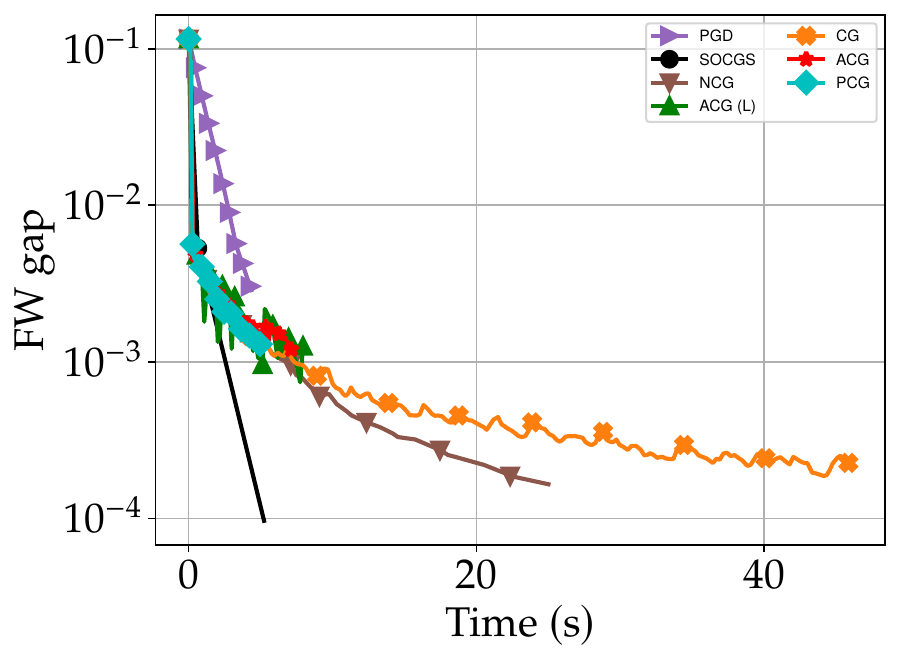} }\label{fig:LogRegDGTime2:Appx}}%
    \hspace{\fill}
    
    \bigskip
    
    \hspace{\fill}
    \subfloat[Iterations]{{\includegraphics[width=6.8cm]{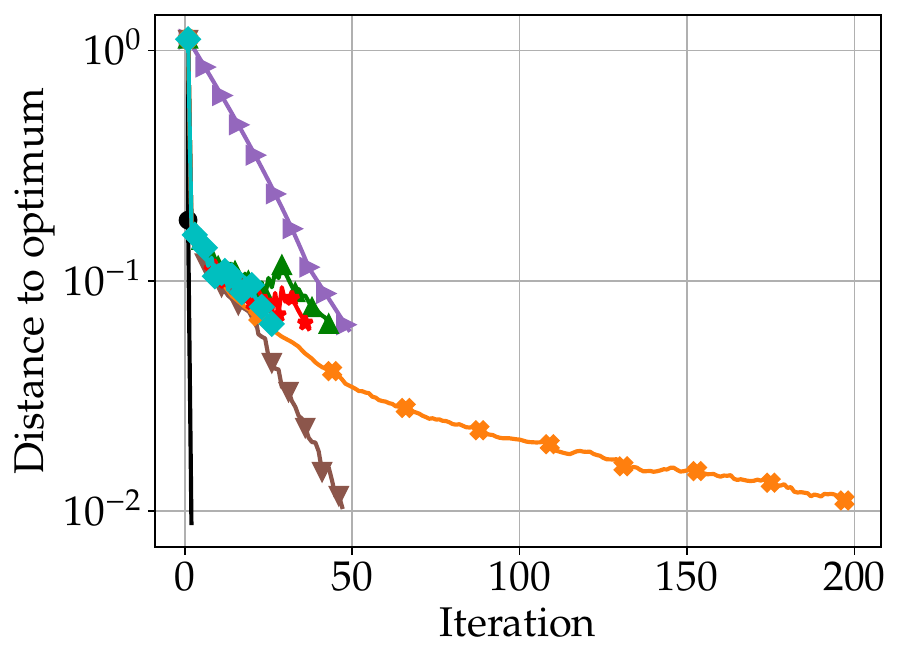} }\label{fig:LogRegDistanceIt2:Appx}}%
    %\qquad
    \hspace{\fill}
    \subfloat[Seconds]{{\includegraphics[width=6.8cm]{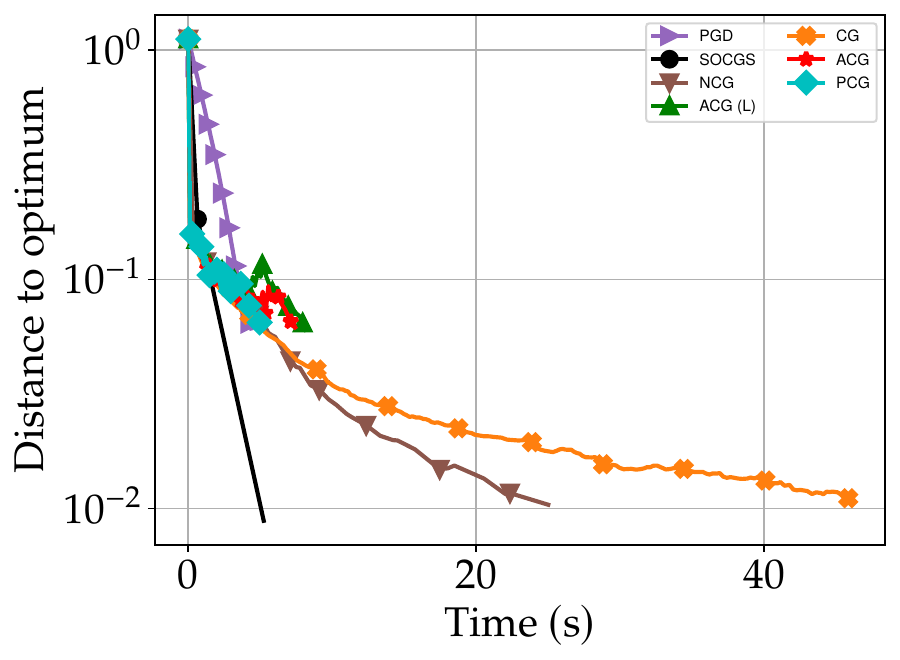} }\label{fig:LogRegDistanceTime2:Appx}}%
    \hspace{\fill}
    \caption{\textbf{Structured Logistic Regression over $\ell_1$ unit ball: } Algorithm comparison in terms of primal gap \protect\subref{fig:LogRegPGIt2:Appx},\protect\subref{fig:LogRegPGTime2:Appx}, Frank-Wolfe gap \protect\subref{fig:LogRegDGIt2:Appx},\protect\subref{fig:LogRegDGTime2:Appx} and distance to the optimum \protect\subref{fig:LogRegDistanceIt2:Appx},\protect\subref{fig:LogRegDistanceTime2:Appx} for the \texttt{real-sim} \cite{chang2011libsvm} dataset, where
$n = 72309$ and $m = 20958$.}%
    \label{fig:LogReg2:Appx}%
\end{figure*}

\newpage

\begin{figure*}[th!]
    \centering
    \vspace{-10pt}
    \hspace{\fill}
    \subfloat[Iterations]{{\includegraphics[width=6.8cm]{Images/New_Images/Spectrahedron_PG_iteration_v2.pdf} }\label{fig:GLassoPGIt:Appx}}%
    %\qquad
    \hspace{\fill}
    \subfloat[Seconds]{{\includegraphics[width=6.8cm]{Images/New_Images/Spectrahedron_PG_time_v2.pdf} }\label{fig:GLassoPGTime:Appx}}%
    \hspace{\fill}
    
    \bigskip 
    
    \hspace{\fill}
    \subfloat[Iterations]{{\includegraphics[width=6.8cm]{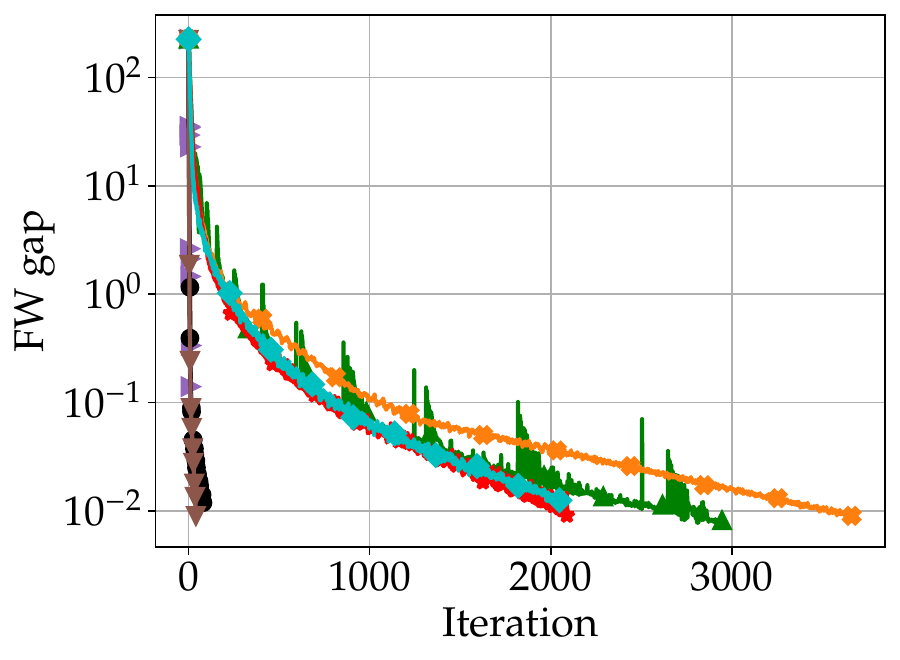} }\label{fig:GLassoDGIt:Appx}}%
    %\qquad
    \hspace{\fill}
    \subfloat[Seconds]{{\includegraphics[width=6.8cm]{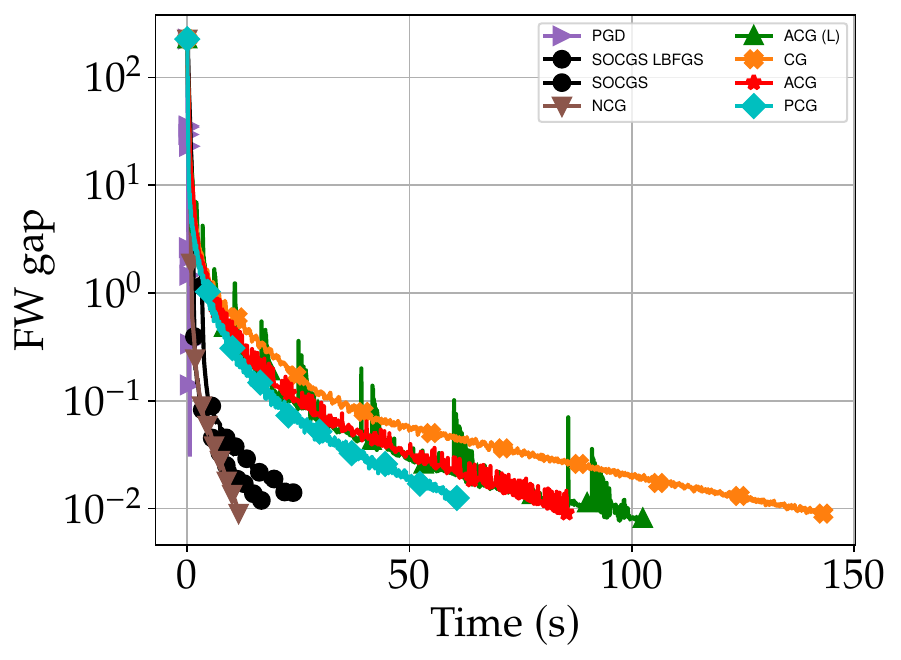} }\label{fig:GLassoDGTime:Appx}}%
    \hspace{\fill}
    
    \bigskip
    
    \hspace{\fill}
    \subfloat[Iterations]{{\includegraphics[width=6.8cm]{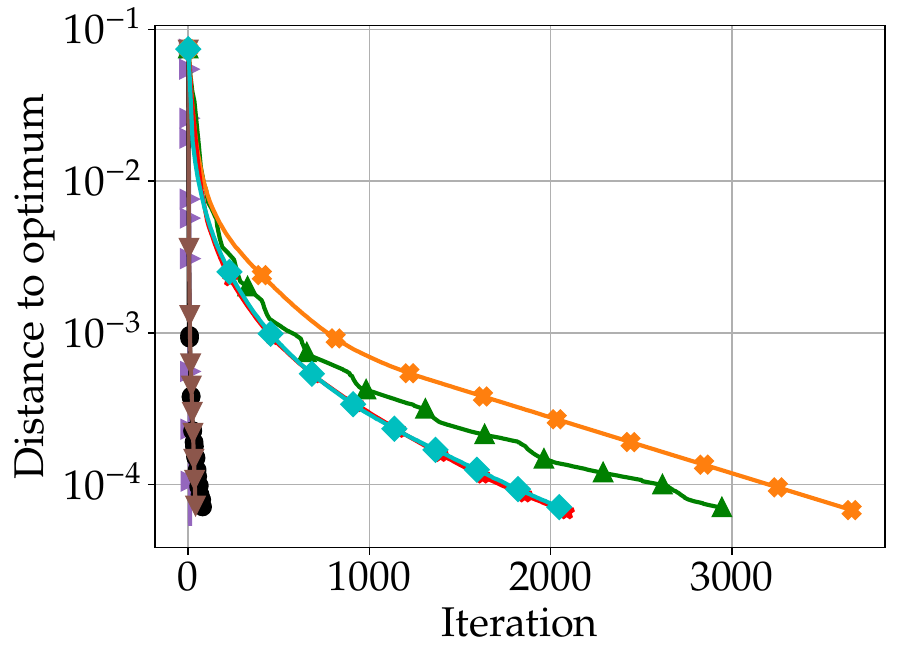} }\label{fig:GLassoDistanceIt:Appx}}%
    %\qquad
    \hspace{\fill}
    \subfloat[Seconds]{{\includegraphics[width=6.8cm]{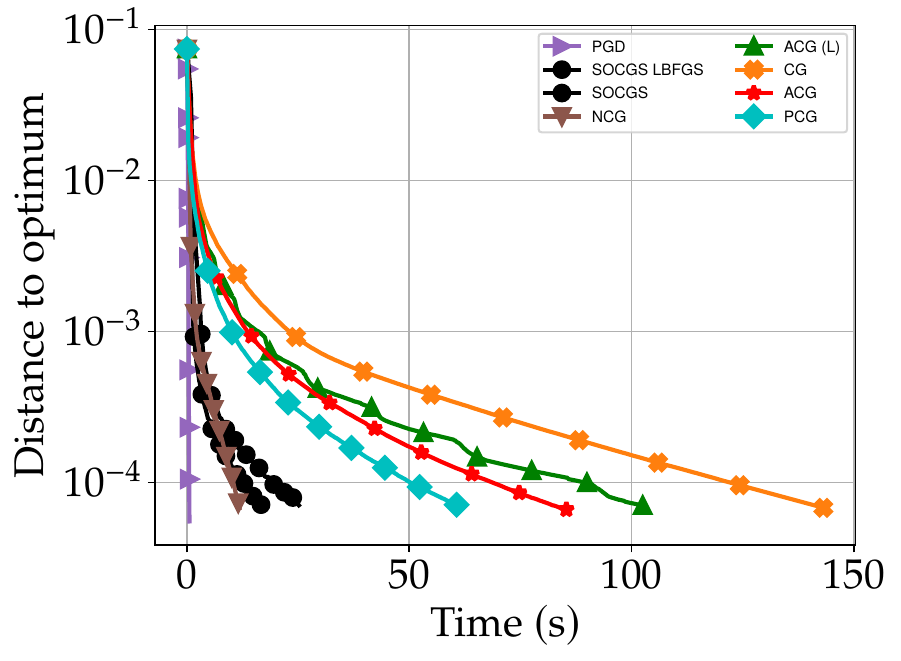} }\label{fig:GLassoDistanceTime:Appx}}%
    \hspace{\fill}
    \caption{\textbf{Inverse covariance estimation over spectrahedron: } Algorithm comparison for $n = 100$ in terms of primal gap \protect\subref{fig:GLassoPGIt:Appx},\protect\subref{fig:GLassoPGTime:Appx}, Frank-Wolfe gap \protect\subref{fig:GLassoDGIt:Appx},\protect\subref{fig:GLassoDGTime:Appx} and distance to the optimum \protect\subref{fig:GLassoDistanceIt:Appx},\protect\subref{fig:GLassoDistanceTime:Appx}.}%
    \label{fig:GLasso:Appx}%
\end{figure*}

\newpage

\begin{figure*}[th!]
    \centering
    \vspace{-10pt}
    \hspace{\fill}
    \subfloat[Iterations]{{\includegraphics[width=6.8cm]{Images/New_Images/Spectrahedron_PG_iteration.pdf} }\label{fig:GLassoPGIt2:Appx}}%
    %\qquad
    \hspace{\fill}
    \subfloat[Seconds]{{\includegraphics[width=6.8cm]{Images/New_Images/Spectrahedron_PG_time.pdf} }\label{fig:GLassoPGTime2:Appx}}%
    \hspace{\fill}
    
    \bigskip 
    
    \hspace{\fill}
    \subfloat[Iterations]{{\includegraphics[width=6.8cm]{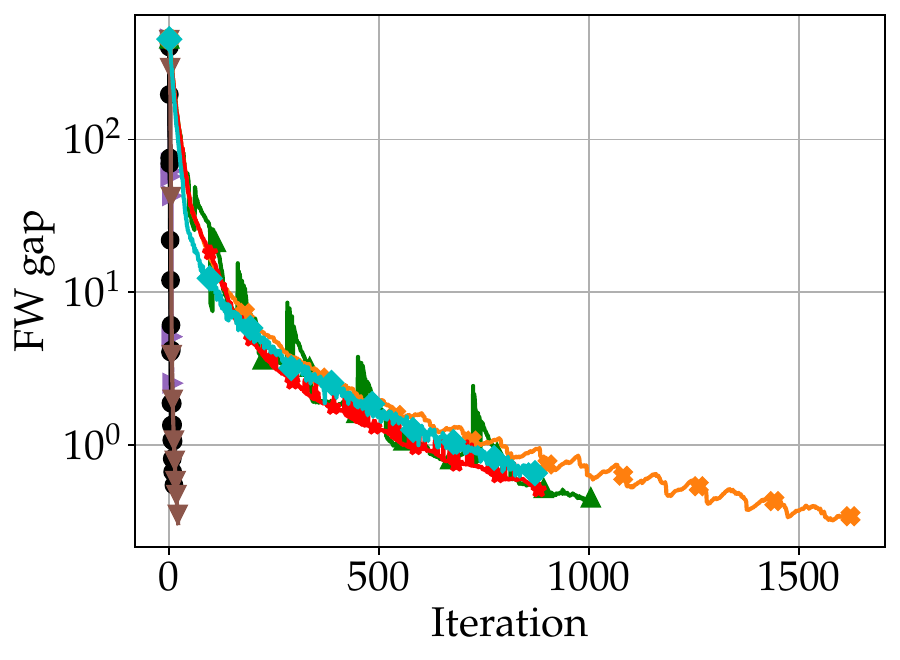} }\label{fig:GLassoDGIt2:Appx}}%
    %\qquad
    \hspace{\fill}
    \subfloat[Seconds]{{\includegraphics[width=6.8cm]{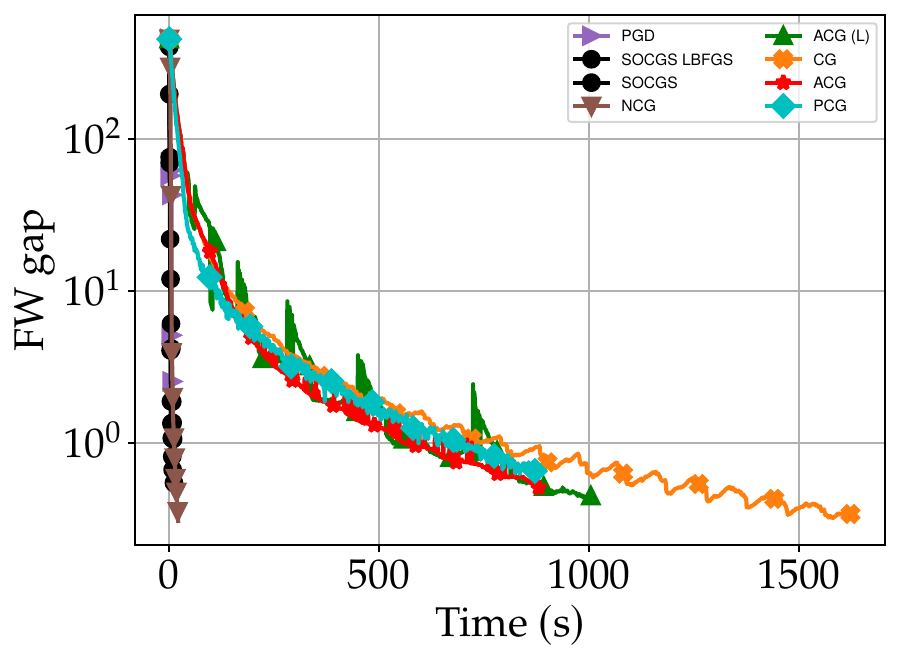} }\label{fig:GLassoDGTime2:Appx}}%
    \hspace{\fill}
    
    \bigskip
    
    \hspace{\fill}
    \subfloat[Iterations]{{\includegraphics[width=6.8cm]{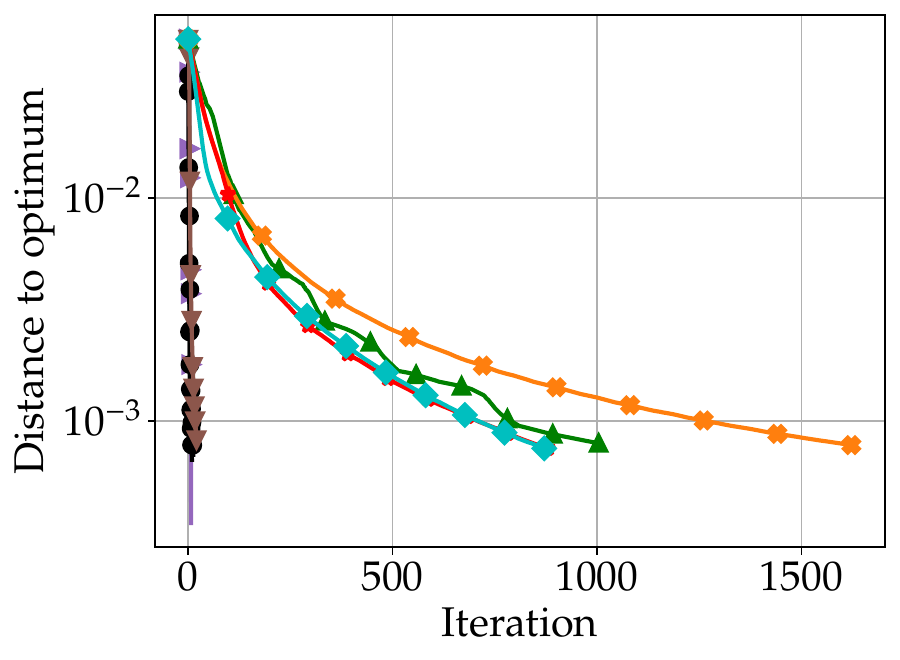} }\label{fig:GLassoDistanceIt2:Appx}}%
    %\qquad
    \hspace{\fill}
    \subfloat[Seconds]{{\includegraphics[width=6.8cm]{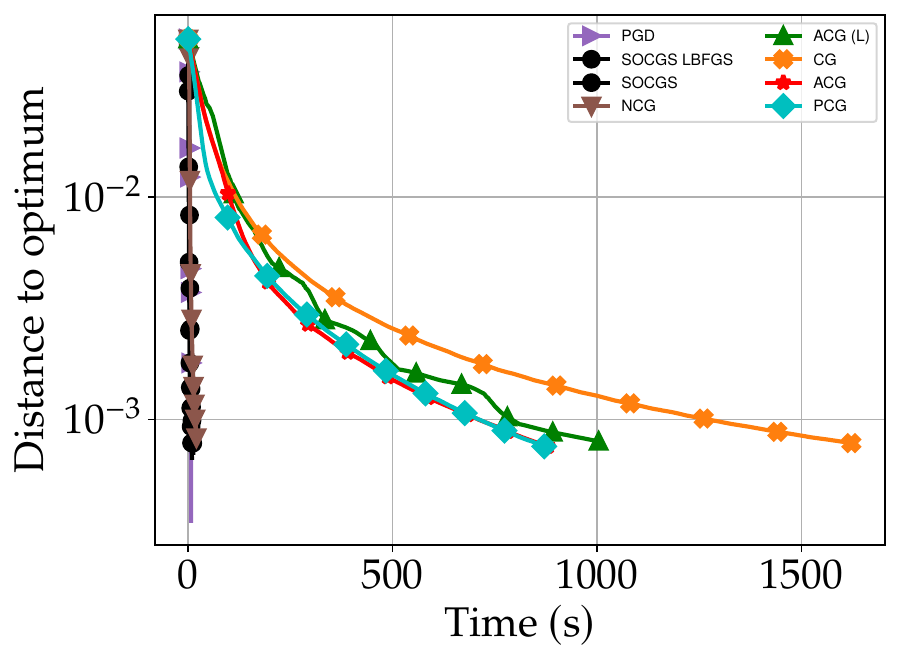} }\label{fig:GLassoDistanceTime2:Appx}}%
    \hspace{\fill}
    \caption{\textbf{Inverse covariance estimation over spectrahedron: } Algorithm comparison for $n = 50$ in terms of primal gap \protect\subref{fig:GLassoPGIt2:Appx},\protect\subref{fig:GLassoPGTime2:Appx}, Frank-Wolfe gap \protect\subref{fig:GLassoDGIt2:Appx},\protect\subref{fig:GLassoDGTime2:Appx} and distance to the optimum \protect\subref{fig:GLassoDistanceIt2:Appx},\protect\subref{fig:GLassoDistanceTime2:Appx}.}%
    \label{fig:GLasso2:Appx}%
\end{figure*}

\end{document}